\documentclass[11pt, a4paper, abstract=yes]{scrartcl}

\usepackage[utf8]{inputenc}
\usepackage[T1]{fontenc}
\usepackage{graphicx}
\usepackage{longtable}
\usepackage{wrapfig}
\usepackage{rotating}
\usepackage[normalem]{ulem}
\usepackage{amsmath}
\usepackage{amssymb}
\usepackage{hyperref}
\usepackage{amsthm,amssymb,amsmath}
\usepackage{hyperref}
\usepackage{graphicx}
\usepackage[T1]{fontenc}\usepackage{lmodern}
\theoremstyle{definition}
\newtheorem{definition}{Definition}
\newtheorem{mexample}[definition]{Example}
\theoremstyle{remark}
\newtheorem{remark}[definition]{Remark}
\theoremstyle{plain}
\newtheorem{lemma}[definition]{Lemma}

\newtheorem{proposition}[definition]{Proposition}
\newtheorem{theorem}[definition]{Theorem}
\newtheorem{corollary}[definition]{Corollary}

\author{Jan Snellman\thanks{Department of Mathematics, Linköping University, 581 83 Linköping, Sweden. \texttt{jan.snellman@liu.se}. ORCID: \href{https://orcid.org/0009-0002-6676-5068}{0009-0002-6676-5068}}}
\usepackage{orcidlink}
\usepackage{booktabs}
\usepackage{array}  
\usepackage{longtable}
\usepackage{float}
\hypersetup{hidelinks}
\newcommand{\Koz}{\operatorname{Koz}}
\newcommand{\shift}{\operatorname{shift}}
\newcommand{\rasimp}{\angle}
\newcommand{\RA}{\operatorname{RA}}
\newcommand{\kozvec}{\vec\kappa}
\newcommand{\extten}{\mathsf{M}}
\newcommand{\proj}{\pi}
\newcommand{\onesvec}{\vec{\mathbf 1}}
\usepackage{marginnote}
\definecolor{leanfull}{HTML}{2E7D32}
\definecolor{leannone}{HTML}{8A8A8A}
\newcommand{\leandiscof}[1]{\tikz[baseline=-0.6ex]\fill[#1] (0,0) circle (0.6ex);}
\newcommand{\leanfull}{\leandiscof{leanfull}}
\newcommand{\leanpartial}{\tikz[baseline=-0.6ex]{\fill[leanfull] (0,-0.6ex) arc (-90:90:0.6ex) -- cycle;\draw[leanfull,line width=0.5pt] (0,0) circle (0.6ex);}}
\newcommand{\leannone}{\tikz[baseline=-0.6ex]\draw[leannone,line width=0.6pt] (0,0) circle (0.55ex);}
\definecolor{compmark}{HTML}{1565C0}
\newcommand{\compmark}{\tikz[baseline=-0.6ex]\fill[compmark] (-0.55ex,-0.55ex) rectangle (0.55ex,0.55ex);}
\newcommand{\leanmark}[1]{\leavevmode{\reversemarginpar\marginnote{\raggedleft#1}}}
\newcommand{\regionbox}[1]{\marginnote{\centering\footnotesize\begin{tikzpicture}[x=1.35cm,y=1.35cm,line width=0.35pt]
\fill[black!5] (0,0) rectangle (1,1);
#1
\draw[black!35,dashed] (0,0.7)--(1,0.7); \draw[black!35,dashed] (0.7,0)--(0.7,1);
\draw (0,0) rectangle (1,1);
\node[anchor=east,font=\tiny] at (0,0.85) {$d$};
\node[anchor=north,font=\tiny] at (0.35,0) {$m$};
\end{tikzpicture}\par}[1\baselineskip]}
\newcommand{\regones}[1]{\fill[blue!22] #1;}
\newcommand{\regzeros}[1]{\fill[red!25] #1;}
\newcommand{\regdiag}{\draw[black!45,line width=0.4pt] (0,0.7)--(0.7,0);}
\usepackage{etoolbox}
\AtEndEnvironment{mexample}{\hfill$\Diamond$}
\newcommand{\ff}{\operatorname{ff}}
\newcommand{\IM}{\operatorname{IM}}
\newcommand{\inn}{\operatorname{in}}
\newcommand{\gin}{\operatorname{gin}}
\newcommand{\conv}{\operatorname{conv}}

\author{Jan Snellman \orcidlink{0009-0002-6676-5068}}
\date{\today}
\title{Hilbert functions over the exterior algebra and \(f\)-vectors in higher rank\\\medskip
\large I: Amata--Crupi monomial modules, Kozlov polytopes, and r-vectors of simplicial complexes}
\hypersetup{
 pdfauthor={Jan Snellman \orcidlink{0009-0002-6676-5068}},
 pdftitle={Hilbert functions over the exterior algebra and \(f\)-vectors in higher rank},
 pdfkeywords={},
 pdfsubject={},
 pdfcreator={},
 pdflang={English}}
\let\autocite\cite

\begin{document}

\maketitle
\begin{abstract}
Let \(E=\bigwedge(k^n)\) be the exterior algebra on \(n\) generators and let
\(F=\bigoplus_{i=1}^{r}Eg_i\) be a graded free \(E\)-module with \(\deg g_i=d_i\). We determine the
convex hull of the set of Hilbert functions of the quotients \(F/M\), where \(M\) runs over the
monomial submodules of Amata and Crupi --- those of the form \(M=I_1g_1\oplus\cdots\oplus I_rg_r\)
with each \(I_i\subseteq E\) a monomial ideal without minimal generators of degree \(\le1\). The
hull is the Minkowski sum
\begin{equation}\label{eq:eq}
  \Koz(n;d_1,\dots,d_r)\;=\;\sum_{i=1}^{r}\shift^{d_i}\bigl(\Koz(n)\bigr)
\end{equation}
of \(r\) shifted copies of one and the same Kozlov simplex \(\Koz(n)\), the \((n-1)\)-simplex which
Kozlov proved \autocite{Kozlov1997} to be the convex hull of the \(f\)-vectors of simplicial complexes on \(n\)
vertices. For \(r=1\) the statement is Kozlov's theorem. We give a self-contained proof, including a
proof of Kozlov's theorem from the local Lubell--Yamamoto--Meshalkin inequality, which is in
essence Kozlov's own second proof, and the facet description of \(\Koz(n)\) that this proof
contains. Along the way we
isolate the class of submodules for which the argument works --- the \emph{ideal-direct-sum modules},
those \(M\subseteq F\) that already split as \(\bigoplus_i(M\cap Eg_i)\), and show by example that
the Hilbert function of a graded submodule outside that class need not be a shifted sum of
\(f\)-vectors at all. The combinatorial content is recorded separately as a statement about
\(r\)-vectors of simplicial complexes and their \emph{ff}-vectors, in language that uses no algebra.
Finally we determine the vertices of \(\Koz(n;d_1,\dots,d_r)\) at rank two: which
sums of vertices of the summands survive as vertices of the sum. The facet descriptions at rank two,
conjectured in the first version of this paper, and the vertex and facet descriptions at every rank
are proved in Part IV of this series.
\end{abstract}


\medskip

\setcounter{tocdepth}{2}
\tableofcontents
\section{Introduction}
\label{sec:org9b6bb99}

\subsection{The rank-one picture}
\label{sec:orgd2ccc67}

The Hilbert function of a graded quotient of the exterior algebra is a \textbf{simpler} combinatorial
object than its symmetric-algebra counterpart. Let \(k\) be a field, \(V=k^n\), and
\(E=E_n=\bigwedge V\). Because \(e_i^2=0\) is a defining relation, \emph{every} monomial of \(E\) is
squarefree, so every monomial ideal of \(E\) is the Stanley--Reisner ideal \(I_\Delta\) of a
order ideal \(\Delta\subseteq\mathcal P([n])\), with no squarefreeness hypothesis needed, and
with no correction term relating \(\dim_k(E/I_\Delta)_j\) to \(\Delta\): the two agree on the nose,
\begin{equation}\label{eq:eq-2}
  \dim_k(E/I_\Delta)_j \;=\; f_j(\Delta),
\end{equation}
where \(f_j(\Delta)\) counts the faces of cardinality \(j\). Passing to an initial ideal preserves
Hilbert functions, and initial ideals of \(E\) are monomial; so the set of Hilbert functions of
graded quotients of \(E\) \emph{is} the set of \(f\)-vectors of simplicial complexes on at most \(n\)
vertices, exactly. Three classical results then describe that set from three different distances:

\begin{itemize}
\item the \textbf{Kruskal--Katona theorem} determines it exactly, by the cascade/shadow inequalities;
\item \textbf{Linusson} \autocite{Linusson1999} counts it;
\item \textbf{Kozlov} \autocite{Kozlov1997} computes its convex hull, and finds an \((n-1)\)-dimensional simplex whose vertices
are the \(f\)-vectors of the complete skeleta of the full simplex on \([n]\).
\end{itemize}

Kozlov's theorem is the coarsest of the three and the one this paper generalizes. It requires the
normalization \(f_1(\Delta)=n\), meaning that every vertex is present, which we call
\emph{properness}; the
theorem is false without it, not merely less sharp (Section \ref{sec-permissive}).
\subsection{The rank-\(r\) question}
\label{sec:org3edf4ef}

Amata and Crupi \autocite{AmataCrupi2020KK,AmataCrupi2020} generalized the Kruskal--Katona theorem from ideals of \(E\) to graded
submodules of a graded free module \(F=\bigoplus_{i=1}^rEg_i\) with \(\deg g_i=d_i\in\mathbb Z\).
Their basic structural observation (their Definition 2.1) is that a submodule generated by monomials
of \(F\) is automatically a direct sum \(M=I_1g_1\oplus\cdots\oplus I_rg_r\) of monomial ideals of
\(E\), one per free generator, because multiplication in \(E\) never moves a term from one
generator's slot to another's. The Hilbert function of the quotient then splits additively,
\begin{equation}\label{eq:eq-3}
  H_{F/M}(j)\;=\;\sum_{i=1}^{r}H_{E/I_i}(j-d_i),
\end{equation}
so a rank-\(r\) Hilbert function is a sum of \(r\) independently chosen \(f\)-vectors, each shifted
by its own generator degree. The question this paper answers is: \textbf{what is the convex hull of the set
of such Hilbert functions?}

The answer (Theorem \ref{thm-main}) is as clean as the question:
\begin{equation}\label{eq:eq-4}
  \conv\bigl\lbrace H_{F/M}\;:\;M\ \text{a proper monomial submodule of }F\bigr\rbrace
  \;=\;\sum_{i=1}^{r}\shift^{d_i}\bigl(\Koz(n)\bigr).
\end{equation}
We call this polytope the \textbf{Kozlov polytope} \(\Koz(n;d_1,\dots,d_r)\); at \(r=1\), \(d_1=0\) it is
\(\Koz(n)\) and the theorem is Kozlov's.

The proof is short, and deliberately so. Once the right class of submodules is fixed, the set of
achievable Hilbert functions is \emph{exactly} a sumset of \(r\) shifted copies of one finite point set
(Proposition \ref{prop-sumset}), and convex hull commutes with Minkowski sum and with linear maps. The
work is in fixing the class: Section \ref{sec-hierarchy} gives a graded submodule whose Hilbert
function is not a shifted sum of \(f\)-vectors at all, so the restriction cannot simply be dropped.

We also give the \textbf{vertex description} of the Kozlov polytope in the rank-two case: Theorem
\ref{thm-kozlov-matrix} says exactly which sums of a vertex of one summand and a vertex of the other
survive as vertices of the sum, and counts them. The facet description at rank two was stated as a
conjecture in the first version of this paper. It is proved in Part IV of this series
\autocite{SnellmanArticle4}, and appears below, without proof, as Theorem \ref{conj-ones-hrep} and
Theorem \ref{conj-kozlov-hrep2}.
\subsection{What is not treated here}
\label{sec-scope}
\textbf{This paper treats rank at most two.} Rank one is Kozlov's theorem; rank two is the first genuinely
new case, and its vertex structure is settled here: the extremal matrix in closed form, the vertex
set, and the vertex count are all proved. The facet description at rank two is \emph{not} proved here:
it is stated, was verified exactly as Section \ref{sec-methods} records, and is proved in Part IV
\autocite{SnellmanArticle4}. Rank three and beyond is the subject of Part IV, where the vertices are
described explicitly at every rank, for every weakly log-concave leg vector, and the facets of
\(\Koz(n;d_1,\dots,d_r)\) are determined for every \(\vec d\). The convex-hull
theorem of Section \ref{sec-main} is stated and proved for every rank, since its proof is uniform and
splitting it would serve nobody; the restriction to rank two governs the finer vertex and facet
structure of Section \ref{sec-vertices}, which is where rank genuinely matters.

This paper stops at the convex hull. The finer question is which \emph{lattice points} of
\(\Koz(n;d_1,\dots,d_r)\) are achievable. Those lattice points are exactly the candidate Hilbert
functions: a point of the polytope with integer coordinates is a tuple that could be an
\(r\)-graded Hilbert function on numerical grounds, and the question is which of them actually
occur. That is the rank-\(r\) analogue of Kruskal--Katona rather than of Kozlov, and it is the
subject of Parts III and V of this series \autocite{SnellmanArticle3,SnellmanArticle5}, and is
answered there by a repaired form of Amata and Crupi's own generalized Kruskal--Katona theorem. Example \ref{ex-worked} below exhibits an achievable
Hilbert function that is a lattice point of the hull but not a vertex of it, which is the shortest
way to see that the two questions genuinely differ.
\subsection{Changes from the first version}
\label{sec-changes}
The first version of this paper is arXiv:2609.23605v1, and is also report LiTH-MAT-R 2026:04 of the
Department of Mathematics, Linköping University; the report is not changed. This version differs
from it as follows. The numbering of statements and equations is unchanged; this subsection is
new, so the one after it has moved by one.

\begin{itemize}
\item The first version said that the facet description of the Kozlov simplex (Corollary
\ref{cor-facets}) appeared not to be recorded in the literature. The inequalities are in Kozlov's
second proof of his theorem \autocite[Section 5]{Kozlov1997}, and the proof given here is in essence
that proof. The abstract, the table below and Section \ref{sec-kozlov} now say so.
\item Conjectures \ref{conj:ones-hrep} and \ref{conj:kozlov-hrep2} of the first version, the facet
descriptions at rank two, are now theorems. They are proved in Part IV \autocite{SnellmanArticle4},
which also describes the vertices and the facets at every rank; the statements keep their numbers
and are given here without proof.
\item Example \ref{ex-ones-facets} said ``with the equation \(x_1=5\) understood''; for the all-ones leg
vector the equation is \(x_1=1\). A remark at the end of Section \ref{sec-hrep2}, on how the
rank-two facet description could not be proved, was mistaken and has been removed.
\item References to ``the companion paper'' and to ``Part III'' have been replaced by references to Parts
III, IV and V, which did not exist under those names when the first version was written.
\item Section \ref{sec-methods} is rewritten in the form used in Part III: the two tables are replaced
by a list of the computations, by marks in the margin for the Lean status of each result and for
results checked by computation, and by
a document in the repository that pairs statements with Lean declarations. Section
\ref{sec-methods-anc} now points to the repository of the series, its mirror, and the Zenodo archive.
The scripts and the Lean development are unchanged.
\end{itemize}
\subsection{What is new here, and what is not}
\label{sec:orgf7807eb}

To be explicit about provenance, since much of the material below is classical.

\begin{center}
\small
\begin{tabular}{ll}
result & status\\
\hline
Thm. \ref{thm-kk}, Kruskal--Katona for \(E\) & classical; cited\\
Thm. \ref{thm-invariance}, invariance under \(\inn\) and \(\gin\) & classical; cited, proof sketched\\
Thm. \ref{thm-kozlov}, Kozlov's theorem & Kozlov 1997; his second proof, retold\\
Cor. \ref{cor-facets}, facets of \(\Koz(n)\) & implicit in Kozlov's second proof\\
Rem. \ref{rem-ac-reduction}, \(\inn(M)\) is monomial & Amata--Crupi; cited\\
Def. \ref{def-ids} and Ex. \ref{ex-nonids}, the hierarchy & new\\
Prop. \ref{prop-in-ids}, \(\inn\) of an ids-module & new\\
Section \ref{sec-combinatorial}, \emph{ff}-vectors & new\\
Prop. \ref{prop-sumset}, the sumset identity & new\\
Thm. \ref{thm-main}, the main theorem & new\\
Thm. \ref{thm-ones}, extremal matrix, all-ones & new; proved\\
Thm. \ref{thm-kozlov-matrix}, extremal matrix and vertex set, Kozlov & new; proved\\
Thm. \ref{conj-ones-hrep}, facets of \(\RA(\onesvec(n);0,d)\), rank two & new; proved in Part IV\\
Thm. \ref{conj-kozlov-hrep2}, facets of \(\RA(\kozvec(n);0,d)\), rank two & new; proved in Part IV\\
\end{tabular}

\end{center}
\section{Notation}
\label{sec-notation}
Collected here, before anything uses it, so that a reader can consult one table rather than hunt
backwards. Pointers are to the defining statement, and nothing in this section is needed to read
the introduction.

One map is used throughout and is worth stating separately. \textbf{\textbf{\(\proj\) forgets the
degree-\(0\) entry}}, in whatever ambient dimension the surrounding argument has fixed: it is
\(\mathbb R^{n+1}\to\mathbb R^{n}\), \((x_0,x_1,\dots,x_n)\mapsto(x_1,\dots,x_n)\), and equally
\(\mathbb R^{N+1}\to\mathbb R^{N}\) where that is what is wanted. The overloading is deliberate
and costs nothing, since the dimension is always determined by what \(\proj\) is applied to.
\textbf{\textbf{Named objects live in the larger space; their images under \(\proj\) live in the smaller one}},
and the projection is written out rather than left to be understood. The shift operator
\(\shift^{d}\) of Definition \ref{def-shift} is overloaded in the same way and for the same reason: it
prepends \(d\) zeros, whether to a degree-\(0\)-inclusive vector (\(\mathbb R^{n+1}\to\mathbb
R^{N+1}\), which is the setting of Definition \ref{def-shift} and of Theorem \ref{thm-main}) or to a leg
vector's simplex in the dropped space (\(\mathbb R^{n}\to\mathbb R^{N}\), which is the setting of
Section \ref{sec-vertices}).

Two further conventions, stated once. \textbf{\textbf{\(\subset\) is not assumed strict}}: we write \(\subseteq\)
when containment is meant and \(\subsetneq\) when properness is, and use \(\subset\) only
where either would do; and \([n]:=\lbrace1,2,\dots,n\rbrace\) throughout.

\begingroup\small
\begin{center}
\begin{tabular}{@{}ll@{}}
  \toprule
  \multicolumn{2}{@{}l}{\emph{Algebra}}\\
  \(E=\bigwedge k^n\)          & the exterior algebra on \(n\) generators \\
  \(e_\sigma\)                 & the monomial \(e_{i_1}\cdots e_{i_k}\) for \(\sigma=\{i_1<\dots<i_k\}\) \\
  \(F=\bigoplus_i Eg_i\)       & free graded module, \(\deg g_i=d_i\), rank \(r\) \\
  \(H_{F/M}\)                  & Hilbert function of \(F/M\) \\
  \(\inn_>\)                   & initial submodule for a monomial order \(>\) \\
  \(\gin\)                     & generic initial submodule \\
  \(I_\Delta\), \(\IM(\vec\Delta)\) & Stanley--Reisner ideal; indicator module \\
  \addlinespace
  \multicolumn{2}{@{}l}{\emph{Combinatorics}}\\
  \(\mathcal P([n])\), \(\mathcal D_n\) & the subset lattice of \([n]\); its order ideals, the empty one included \\
  \(\Delta\), \(\mathcal S_n\) & simplicial complex on \([n]\) in the strict sense; the set of those \\
  \(S_n\), \(S_n^{+}\)          & \(f\)-vectors of complexes; of non-empty order ideals, \eqref{eq:Splus} \\
  \(f_j(\Delta)\)              & number of \(j\)-element faces \\
  \(\vec\Delta\), \(\vec s\)   & an \(r\)-vector of complexes; its \emph{ff}-vector \\
  \(\partial\)                 & the shadow operator \\
  \(\mu_j=x_j/\binom nj\)      & the LYM ratios \\
  \addlinespace
  \multicolumn{2}{@{}l}{\emph{Geometry}}\\
  \(\vec a=(a_1,\dots,a_n)\)        & the leg vector, of length \(n\); not indexed by a summand \\
  \(\rasimp(\vec a)\)          & right-angle simplex, vertices \(v_j=(a_1,\dots,a_j,0,\dots,0)\) \\
  \(\kozvec(n)\)               & the Kozlov vector, \(\bigl(\binom n1,\dots,\binom nn\bigr)\) \\
  \(\Koz(n)=\conv(S_n)\subseteq\mathbb R^{n+1}\) & the Kozlov simplex \\
  \(\proj(\Koz(n))=\rasimp(\kozvec(n))\) & its image in \(\mathbb R^n\) \\
  \(\vec d=(d_1,\dots,d_r)\)   & the shift vector, sorted; \(\shift^d\) prepends \(d\) zeros \\
  \(d\)                        & a \emph{scalar}: the rank-two case \(\vec d=(0,d)\) \\
  \(m=n-d\)                    & shared coordinates, rank two only \\
  \(N=n+\max\vec d\)           & the ambient dimension \\
  \(\RA(\vec a;\vec d)\)       & the right-angle polytope, \(\sum_i\shift^{d_i}\rasimp(\vec a)\) \\
  \(\Koz(n;\vec d)\) & the Kozlov polytope, \(\sum_i\shift^{d_i}\Koz(n)\) \\
  \(\proj(\Koz(n;\vec d))\) & its image, a translate of \(\RA(\kozvec(n);\vec d)\subseteq\mathbb R^N\), \eqref{eq:koz-as-ra} \\
  \(\proj\)                     & forgets the degree-\(0\) entry, in whatever dimension \\
  \(\extten(\vec a;d)[i][j]\)       & the extremal matrix; \(t=i-d\) \\
  \(v_i\), \(u_j=\shift^d(v_j)\) & vertices of the unshifted and the shifted summand, rank two \\
  \(h(u,\cdot)\), \(F(\varphi,\cdot)\) & support function; exposed face \\
  \addlinespace
  \multicolumn{2}{@{}l}{\emph{Appendix~\ref{sec:app} only}}\\
  \(\varphi_c\), \(\Phi\)      & a functional's \(c\)-th coordinate; its partial sums \\
  \(e_s=\varphi_{d+s}\)        & the functional restricted to the shared block \\
  \(P_{j'}\), \(Q_{t'}\)       & shared-block contributions, \eqref{eq:PQ} \\
  \(\rho_s=a_{d+s}/a_s\)       & ratio by which the two summands weight a shared coordinate \\
  \bottomrule
\end{tabular}
\end{center}
\endgroup

Four collisions are worth flagging, three of them against standard usage and one internal.

An \(f\)-vector here is always the face count of a \emph{simplicial complex}, hence a lattice point of
\(\Koz(n)\); the face numbers of a \emph{polytope} are the other standard use of that phrase, and where
they are needed we write \emph{polytopal face-vector} instead. A bare \(d\) is a scalar and signals rank two,
whereas \(\vec d\) is the shift vector of any rank; \(m=n-d\) inherits that restriction. Two
coordinate systems are in play and the shift operator is linear only in one of them: the
degree-\(0\)-inclusive coordinates of Definition \ref{def-shift} retain the leading \(1\), while
\(\rasimp(\vec a)\) and \(\RA(\vec a;d_1,\dots,d_r)\) drop it; Section \ref{sec-vertices} records the
translation between them.
The leg vector \(\vec a\) is the only vector here indexed by \(1,\dots,n\) rather than by a summand,
but it carries an arrow like every other vector; what distinguishes it is its length, \(n\) rather
than \(r\). The internal collision is the letter \(M\), which denotes both a submodule of \(F\) and
the extremal tensor of Definition \ref{def-extten}; the two never occur in the same argument.
\section{The exterior algebra, and Hilbert functions}
\label{sec-setup}
Throughout, \(k\) is a field. Nothing below requires a hypothesis on \(\operatorname{char}k\); where
\(\gin\) is used, \(k\) is assumed infinite.

The notation is collected in Section \ref{sec-notation}; only one thing about it governs how the
statements below are read. An arrow marks an \(r\)-tuple (\(\vec d\), \(\vec\jmath\),
\(\vec\Delta\)), so that a bare \(d\) is the single scalar shift of the rank-two case
\(\vec d=(0,d)\), and a statement written with \(\vec d\) is one about general rank. Section
\ref{sec-vertices} says at the head of each subsection which of the two regimes it is in.

\begin{definition}
Let \(V=k^n\) have basis \(e_1,\dots,e_n\). The \textbf{exterior algebra} is
\begin{equation}\label{eq:exterior}
  E \;=\; E_n \;=\; \bigwedge V \;=\; \frac{k\langle e_1,\dots,e_n\rangle}
       {\bigl(e_ie_j+e_je_i,\ e_i^2\bigr)}.
\end{equation}
Write \([n]:=\lbrace1,\dots,n\rbrace\), and for a subset
\(\sigma=\lbrace s_1<\cdots<s_j\rbrace\) of \([n]\) put
\begin{equation}\label{eq:exterior-1}
  e_\sigma\;=\;e_{s_1}\wedge\cdots\wedge e_{s_j}.
\end{equation}
\label{def-exterior}
\end{definition}

The \(e_\sigma\) form a \(k\)-basis, \(E\) is \(\mathbb Z_{\ge0}\)-graded with
\(\dim_kE_j=\binom nj\), and \(\dim_kE=2^n\). The algebra \(E\) is graded-commutative, so a one-sided graded ideal
is automatically two-sided, and the canonical identification \(az=(-1)^{\deg a\deg z}za\) makes the
left- and right-module categories agree. We work with left modules throughout, writing
\(e_\sigma g_i\) and never \(g_ie_\sigma\).

\begin{definition}
Fix \(r\ge1\) and degrees \(d_1,\dots,d_r\in\mathbb Z\). Let
\begin{equation}\label{eq:free-module}
  F \;=\; \bigoplus_{i=1}^{r}E(-d_i) \;=\; \bigoplus_{i=1}^{r}Eg_i,\qquad \deg g_i=d_i.
\end{equation}
We say \(F\) has \textbf{rank \(r\)}. For a graded submodule \(M\subseteq F\), the \textbf{Hilbert function} of the
quotient is \(H_{F/M}(j)=\dim_k(F/M)_j\) and the \textbf{Hilbert series} is
\(H_{F/M}(t)=\sum_jH_{F/M}(j)t^j\), a Laurent polynomial supported in degrees \(d_1,\dots,d_r+n\).
Note that the Hilbert function is always that of the \textbf{quotient}: we never use the convention, common
in computer-algebra systems and in Macaulay2 in particular, under which \(\mathrm{Hilb}(I)\) means
\(\mathrm{Hilb}(E/I)\).
\label{def-free-module}
\end{definition}

``Rank'' always counts free generators, never minimal generators of an ideal. Reordering the \(g_i\) is
a relabelling of the basis, so we assume throughout \(d_1\le d_2\le\cdots\le d_r\), and, after
Corollary \ref{cor-normalization}, usually \(d_1=0\).

\begin{remark}
Both \(I\) and \(E/I\) are graded \(E\)-modules, but only \(E/I\) is \emph{cyclic}: it is generated by
\(1+I\), whereas \(I\) is generated as an ideal by its minimal generators, of which there is usually
more than one and none redundant. The distinction is worth flagging because ``graded ideal'' and
``cyclic module'' are easy to conflate; they coincide only when \(I\) is principal.
\label{rem-cyclic}
\end{remark}

The following is standard, and is what makes the combinatorics below describe \emph{all} graded
submodules rather than only the monomial ones. Fix a monomial order \(>\) on the monomials
\(e_\sigma g_i\) of \(F\). For a graded submodule \(M\subseteq F\), let \(\inn_>(M)\) be the
\(k\)-span of the leading monomials of the elements of \(M\); for \(k\) infinite let
\(\gin_>(M)=\inn_>(g\cdot M)\) for \(g\) in the dense open subset of \(\mathrm{GL}_n(k)\) on which
this is constant. Here \(g\) acts linearly on \(V=k^n\), hence on the exterior algebra
\(E=\bigwedge V\) by functoriality, hence diagonally on the free module \(F=\bigoplus_iEg_i\) with
the generators \(g_i\) fixed; \(g\cdot M\) is the image of \(M\) under that action, again a graded
submodule.

\begin{theorem}
\leanmark{\leannone}\relax \cite{AramovaHerzog2000,AramovaHerzogHibi1997,Eisenbud1995,HerzogHibi2011}
Let \(k\) be a field, let \(F\) be as in Definition \ref{def-free-module}, let \(M\subseteq F\) be a
graded submodule and let \(>\) be a monomial order on \(F\).
\begin{enumerate}
\item \(H_{F/M}(t)=H_{F/\inn_>(M)}(t)\). No hypothesis on \(k\) is needed.
\item If in addition \(k\) is infinite --- so that \(\gin_>(M)\) is defined --- then also
\(H_{F/M}(t)=H_{F/\gin_>(M)}(t)\).
\end{enumerate}
\label{thm-invariance}
\end{theorem}

\begin{proof}
(1) In each degree \(j\), choose a \(k\)-basis of \(M_j\) in echelon form with respect to \(>\); its
leading monomials are then pairwise distinct and
span \(\inn_>(M)_j\), so \(\dim_k\inn_>(M)_j=\dim_kM_j\), and therefore \(\dim_k(F/M)_j\) and
\(\dim_k(F/\inn_>(M))_j\) agree.

(2) Each \(g\in\mathrm{GL}_n(k)\) induces a degree-preserving automorphism of \(E\), hence a graded
isomorphism \(M\cong g\cdot M\), so \(H_{F/M}=H_{F/g\cdot M}\); now apply (1) to \(g\cdot M\) for a
\(g\) in the dense open subset defining \(\gin_>\). Infiniteness of \(k\) is used only to know that
subset is non-empty.
\end{proof}

The consequence we use is that \emph{nothing about which Hilbert functions occur is lost by restricting
attention to monomial submodules}. That is the reason the combinatorics below is not a special case.
\section{Rank one: simplicial complexes and the exterior face ring}
\label{sec-rank1}
Take \(r=1\), \(d_1=0\), so \(F=E\) and a graded submodule is a graded ideal \(I\subseteq E\).

\begin{definition}
Write \(\mathcal P([n])\) for the lattice of subsets of \([n]\), and \(\mathcal D_n\) for the set of
\textbf{order ideals} of \(\mathcal P([n])\): the families \(\Delta\subseteq\mathcal P([n])\) with
\(\tau\subseteq\sigma\in\Delta\Rightarrow\tau\in\Delta\). \textbf{\textbf{The empty order ideal
\(\Delta=\emptyset\) belongs to \(\mathcal D_n\)}}, the condition holding vacuously, and so do the
order ideals that omit some singleton.

A \textbf{simplicial complex on \([n]\)} is an order ideal \(\Delta\in\mathcal D_n\) that contains
\(\emptyset\) and every singleton \(\lbrace i\rbrace\), \(i\in[n]\). Its
\textbf{\(f\)-vector} is \(\bigl(f_0(\Delta),f_1(\Delta),\dots,f_n(\Delta)\bigr)\) with
\begin{equation}\label{eq:complex}
  f_j(\Delta)\;=\;\bigl|\lbrace\sigma\in\Delta:|\sigma|=j\rbrace\bigr| .
\end{equation}
Write \(\mathcal S_n\) for the set of simplicial complexes on \([n]\).
\label{def-complex}
\end{definition}

For \(\Delta\in\mathcal S_n\) the two clauses give \(f_0(\Delta)=1\) and \(f_1(\Delta)=n\); for a
general order ideal neither holds, and \(f_0(\emptyset)=0\). We use the \textbf{cardinality} convention for
\(f\)-vectors throughout, matching Kozlov; the dimension convention differs by an index shift.

\textbf{Which convention.} Three mutually inequivalent definitions of ``simplicial complex on the ground set
\([n]\)'' are in use, and every statement below depends on which is taken, so we name them once:
(i) all order ideals of \(\mathcal P([n])\), the empty one included; (ii) the non-empty order ideals,
equivalently those containing \(\emptyset\); (iii) the order ideals containing \(\emptyset\) and
every singleton. \textbf{This paper uses (iii)}, which is Kozlov's convention and the most common one, and
reserves \(\mathcal S_n\) for it; where the wider classes are wanted they are named as what they are,
\(\mathcal D_n\) for (i) and ``non-empty order ideal'' for (ii). Theorem \ref{thm-kozlov} is true under
(iii) and false under (ii) --- see Section \ref{sec-permissive}, which computes the hull in that case as
well --- so a reader importing a different convention gets a false statement rather than a weaker one.

\begin{definition}
For an order ideal \(\Delta\in\mathcal D_n\), the \textbf{Stanley--Reisner ideal} is
\begin{equation}\label{eq:sr-ideal}
  I_\Delta\;=\;\bigl(e_\sigma\ :\ \sigma\notin\Delta\bigr)\ \subseteq\ E,
\end{equation}
and the quotient \(E/I_\Delta\) is the \textbf{exterior face ring}, also called Sköldberg's \textbf{indicator
algebra} \(k\lbrace\Delta\rbrace\).
\label{def-sr-ideal}
\end{definition}

Non-faces are upward closed, so \(I_\Delta\) is a graded ideal.

\begin{proposition}
\leanmark{\leannone}1. The surviving monomials of \(E/I_\Delta\) are exactly the \(e_\sigma\) with \(\sigma\in\Delta\),
   so \(H_{E/I_\Delta}=(f_0(\Delta),f_1(\Delta),\dots,f_n(\Delta))\) with no correction term.
\begin{enumerate}
\item The assignment \(\Delta\mapsto I_\Delta\) is a bijection from \(\mathcal D_n\) to the monomial
ideals of \(E\), the empty order ideal corresponding to the unit ideal \(I_\emptyset=E\).
\item As sets of integer sequences,
\[
     \lbrace H_{E/I}: I\subseteq E\ \text{graded}\rbrace
     \;=\;\lbrace (f_0(\Delta),f_1(\Delta),\dots,f_n(\Delta)): \Delta\in\mathcal D_n\rbrace ,
   \]
where \(f_0(\Delta)=1\) for every \(\Delta\ne\emptyset\) and \(f_0(\emptyset)=0\); the single
sequence with \(f_0=0\) is the zero function, \(H_{E/E}\).
\end{enumerate}
\label{prop-dictionary}
\end{proposition}

\begin{proof}
(1) is immediate from the definition, since \(I_\Delta\) is spanned as a \(k\)-vector space by the
\(e_\sigma\) with \(\sigma\notin\Delta\). For (2), every monomial of \(E\) is squarefree because \(e_i^2\) is a relation, so the monomials outside a monomial ideal \(I\) form an order ideal
\(\Delta\in\mathcal D_n\) with \(I=I_\Delta\); injectivity is (1). (3): ``\(\supseteq\)'' is (2);
``\(\subseteq\)'' follows from Theorem \ref{thm-invariance}, since \(\inn(I)\) is monomial, hence some
\(I_\Delta\).
\end{proof}

\begin{remark}
Three graded \(k\)-algebras attach to the same \(\Delta\), all with underlying graded vector space
spanned by \(\lbrace x_\sigma:\sigma\in\Delta\rbrace\) but with three genuinely different
multiplications: the ordinary Stanley--Reisner ring
\(k[\Delta]=k[x_1,\dots,x_n]/I^{\mathrm{sym}}_\Delta\) (typically infinite-dimensional); the exterior
face ring \(k\lbrace\Delta\rbrace=E/I_\Delta\) (finite-dimensional, skew-commutative); and the
\textbf{artinified Stanley--Reisner ring} \(k[\Delta]/(x_1^2,\dots,x_n^2)\) (finite-dimensional,
commutative). Sköldberg \autocite{Skoldberg1999} proves that the latter two have identical Hilbert functions and, more
strongly, that their module categories are equivalent by an explicit sign-twisting functor commuting
with \(\operatorname{Ext}\).
\label{rem-three-rings}
\end{remark}

For \(a,j\ge1\) there is a unique \(j\)-cascade expansion
\begin{equation}\label{eq:three-rings}
  a=\binom{a_j}{j}+\binom{a_{j-1}}{j-1}+\cdots+\binom{a_\ell}{\ell},
  \qquad a_j>a_{j-1}>\cdots>a_\ell\ge\ell\ge1,
\end{equation}
and one sets
\begin{equation}\label{eq:three-rings-2}
  a^{(j)}\;=\;\binom{a_j}{j+1}+\binom{a_{j-1}}{j}+\cdots+\binom{a_\ell}{\ell+1},
\end{equation}
with the convention \(0^{(j)}=0\).

\begin{theorem}
\leanmark{\leannone}\relax \autocite{ClementsLindstroem1969,Katona1968,Kruskal1963}, in this form
Aramova, Herzog and Hibi~\cite[Theorem 4.1]{AramovaHerzogHibi1997}. A sequence \((1,h_1,\dots,h_n)\) of non-negative integers is the Hilbert function of
\(E/I\) for some graded ideal \(I\subseteq E\) if and only if\footnote{The lower bound is \(0\le h_{j+1}\), not \(0<h_{j+1}\). This matters: the printed statement of Amata--Crupi's Theorem 3.1, which quotes Aramova--Herzog--Hibi, has \(<\) where the source has \(\le\), and as a result formally denies that \((1,3,3,0)\), the Hilbert function of \(E_3/(e_1e_2e_3)\), is a Hilbert function at all.}
\begin{equation}\label{eq:kk}
  h_1\le n \qquad\text{and}\qquad 0\le h_{j+1}\le h_j^{(j)}\quad (1\le j\le n-1).
\end{equation}
\label{thm-kk}
\end{theorem}

\begin{definition}
A Hilbert function \((1,h_1,\dots,h_n)\) of \(E/I\) is \textbf{proper} if \(h_1=n\), equivalently if \(I\)
has no minimal generator of degree \(\le1\), equivalently if the corresponding \(\Delta\) uses the
full vertex set. Write
\begin{equation}\label{eq:proper}
  S_n\;=\;\bigl\lbrace(1,f_1(\Delta),\dots,f_n(\Delta)):\Delta\in\mathcal S_n\bigr\rbrace
   \;\subseteq\;\mathbb Z^{n+1},
\end{equation}
a finite set, every element of which begins \((1,n,\dots)\).
\label{def-proper}
\end{definition}

By Proposition \ref{prop-dictionary} and Definition \ref{def-proper}, \(S_n\) is exactly the set of
proper Hilbert functions. The map \(\Delta\mapsto\) its \(f\)-vector is surjective onto \(S_n\) by
construction, but it is \textbf{not} injective: already for \(n=3\) the nine proper complexes realise only
five distinct \(f\)-vectors, and \((1,3,1,0)\) is the \(f\)-vector of three of them. What the
present section needs is the equality of sets, not a bijection of complexes onto them.
\section{The Kozlov simplex, with its facets}
\label{sec-kozlov}
The single combinatorial input to this section is the local Lubell--Yamamoto--Meshalkin inequality,
universally abbreviated \textbf{local LYM}. The three are Daniel Lubell, Koichi Yamamoto and Lev
Meshalkin, who proved the global form independently between 1954 and 1966 \autocite{Lubell1966,Meshalkin1963,Yamamoto1954}. Only the local, single-level form is needed.

The proof of Theorem \ref{thm-kozlov} given in this section is in essence the second of Kozlov's two
proofs \autocite[Section 5]{Kozlov1997}. He divides the coordinates by the binomial coefficients,
obtains the chain of Lemma \ref{lem-mu} in the rescaled coordinates, and proves the local inequality
of Lemma \ref{lem-lym} by the same double count. What this section adds is bookkeeping: the chain is
written in the original coordinates, and each of its inequalities is identified as the facet
opposite a named vertex (Corollary \ref{cor-facets}). The first version of this paper called that
facet description apparently unrecorded; the inequalities are in Kozlov's paper, though not stated
there as a description of the facets.

\begin{lemma}
\leanmark{\leanfull}Let \(\Delta\in\mathcal D_n\) be an order ideal and \(1\le j\le n-1\). Then
\((n-j)f_j(\Delta)\ge(j+1)f_{j+1}(\Delta)\), equivalently
\begin{equation}\label{eq:lym}
  \frac{f_j(\Delta)}{\binom nj}\;\ge\;\frac{f_{j+1}(\Delta)}{\binom n{j+1}}.
\end{equation}
\label{lem:lym}
\label{lem-lym}
\end{lemma}

\begin{proof}
Count pairs \((\tau,\sigma)\) with \(\tau\subset\sigma\), \(\sigma\in\Delta\), \(|\sigma|=j+1\),
\(|\tau|=j\). Each such \(\sigma\) contributes exactly \(j+1\) pairs, and all of its \(j\)-subsets
lie in \(\Delta\) by downward closure, so the number of pairs is exactly \((j+1)f_{j+1}\). Each
\(\tau\in\Delta\) with \(|\tau|=j\) is contained in exactly \(n-j\) subsets of \([n]\) of size
\(j+1\), of which at most all lie in \(\Delta\), so the number of pairs is at most \((n-j)f_j\). The
normalized form follows by dividing \((n-j)f_j\ge(j+1)f_{j+1}\) through by
\((j+1)\binom n{j+1}\) and using \(\binom n{j+1}=\binom nj\,(n-j)/(j+1)\), so that
\begin{equation}\label{eq:lym-2}
  \frac{f_{j+1}}{\binom n{j+1}}\;=\;\frac{f_{j+1}(j+1)}{\binom nj\,(n-j)}
  \;\le\;\frac{f_j}{\binom nj}.
\end{equation}
\end{proof}

For \(j=1,\dots,n\), write \(\tilde F_j\) for the \(f\)-vector of the complete \(j\)-skeleton on
\([n]\), with the degree-\(0\) entry dropped:\footnote{The notation is Kozlov's, but the
tilde is not doing here what it does there, and the difference is worth one line. In
\cite{Kozlov1997} the tilde separates the simplicial-complex case, \(\tilde F_k(n)\) of his
\S4, from the flag-complex case, \(F_r(n)\) of his \S3; it says nothing about a degree-\(0\)
entry. (In the 1996 technical report that became that paper, both were written \(F_k(n)\), and
the tilde was added to separate them.) His \(f\)-vector is indexed by dimension, so it begins at
\(f_0=\binom n1\) and the empty face never appears; ours is indexed by cardinality, so it begins
at \(f_0=1\). The entry we drop is an artefact of that difference, not of his notation. We keep
his glyph for recognisability.}
\begin{equation}\label{eq:skeleton}
  \tilde F_j\;=\;\Bigl(\tbinom n1,\ \tbinom n2,\ \dots,\ \tbinom nj,\ 0,\ \dots,\ 0\Bigr)\in\mathbb R^n .
\end{equation}

A word on the word \textbf{coordinates}, which this paper uses for two unrelated operations and which is
worth disentangling once. One is the passage between \(\mathbb R^{n+1}\) and \(\mathbb R^n\) that
drops or restores the degree-\(0\) entry --- a projection and its section, written throughout with
the map \(\proj\) of Section \ref{sec-notation}. The other, unrelated to it, is the \textbf{rescaling} below, which
divides the \(j\)-th entry by \(\binom nj\) and does not change the ambient space at all. Where
both could be meant we name the one intended.

The whole section rests on that rescaling, which we isolate first --- as a definition and a lemma
about it --- so that the theorem and its corollary can each be proved in a line.

\begin{definition}
(\emph{LYM ratios.}) For \(x\in\mathbb R^n\) the \textbf{LYM ratios} of \(x\) are
\begin{equation}\label{eq:lym-ratios}
  \mu_j(x)\;=\;\frac{x_j}{\binom nj}\qquad(1\le j\le n),
\end{equation}
and we write \(\mu_j\) for \(\mu_j(x)\) when \(x\) is clear. The map \(x\mapsto\mu(x)\) is a
linear automorphism of \(\mathbb R^n\): it rescales each coordinate by a non-zero factor and moves
nothing between coordinates.
\label{def:lym-ratios}
\label{def-lym-ratios}
\end{definition}

\begin{lemma}
\leanmark{\leanpartial}(\emph{The LYM chain.}) With the LYM ratios \(\mu_j\) of Definition \ref{def-lym-ratios},
\begin{equation}\label{eq:mu}
  \conv\lbrace\tilde F_1,\dots,\tilde F_n\rbrace
  \;=\;\lbrace x\in\mathbb R^n:\ \mu_1=1\ge\mu_2\ge\cdots\ge\mu_n\ge0\rbrace .
\end{equation}
Moreover the points \(\tilde F_1,\dots,\tilde F_n\) are affinely independent, so this set is an
\((n-1)\)-dimensional simplex; and for \(n\ge2\) its facets are cut out by the \(n\) inequalities
\(\mu_j\ge\mu_{j+1}\) for \(1\le j\le n-1\), together with \(\mu_n\ge0\).
\label{lem:mu}
\label{lem-mu}
\end{lemma}

\begin{proof}
Consider a convex combination \(x=\sum_{j=1}^n\lambda_j\tilde F_j\), so that \(\lambda_j\ge0\) and
\(\sum_j\lambda_j=1\). Reading off the \(j\)-th coordinate gives \(\mu_j=\sum_{\ell\ge j}\lambda_\ell\),
whence
\begin{equation}\label{eq:mu-2}
  \lambda_j=\mu_j-\mu_{j+1}\quad(j<n),\qquad \lambda_n=\mu_n,\qquad \mu_1=1 .
\end{equation}
The conditions \(\lambda_j\ge0\) and \(\sum_j\lambda_j=1\) therefore translate exactly into the
displayed chain, which proves the equality of the two sets. The translation is a bijection between
the \(\lambda\) and the \(\mu\) description, so the \(\tilde F_j\) are affinely independent and the
set is a simplex of dimension \(n-1\). Each of the \(n\) inequalities vanishes precisely when one
barycentric coordinate does, namely \(\mu_j\ge\mu_{j+1}\) on \(\lambda_j=0\) and \(\mu_n\ge0\) on
\(\lambda_n=0\), so each cuts out a facet.
\end{proof}

\begin{theorem}
\leanmark{\leanpartial}\relax Kozlov~\cite[Theorem 4.1]{Kozlov1997} For \(n\ge1\),
\begin{equation}\label{eq:kozlov}
  \conv\bigl\lbrace(f_1(\Delta),\dots,f_n(\Delta)):\Delta\in\mathcal S_n\bigr\rbrace
  \;=\;\conv\lbrace\tilde F_1,\dots,\tilde F_n\rbrace,
\end{equation}
an \((n-1)\)-dimensional simplex.
\label{thm:kozlov}
\label{thm-kozlov}
\end{theorem}

\begin{proof}
For the inclusion \(\subseteq\), let \(\Delta\in\mathcal S_n\). Then \(f_1(\Delta)=n\) gives
\(\mu_1=1\), Lemma \ref{lem-lym} gives \(\mu_j\ge\mu_{j+1}\) for every \(j\), and \(\mu_n\ge0\) is
clear; so the \(f\)-vector of \(\Delta\) satisfies the description of Lemma \ref{lem-mu}. For the
inclusion \(\supseteq\), each \(\tilde F_j\) with \(j\ge1\) is itself the \(f\)-vector of the complete
\(j\)-skeleton, which is a member of \(\mathcal S_n\). That the hull is a simplex of dimension
\(n-1\) is part of Lemma \ref{lem-mu}.
\end{proof}

\begin{definition}
The \textbf{Kozlov simplex} is
\begin{equation}\label{eq:kozsimplex}
  \Koz(n)\;:=\;\conv(S_n)\;\subseteq\;\mathbb R^{n+1},
\end{equation}
which lies in the hyperplane \(\lbrace x_0=1\rbrace\) because every element of \(S_n\) does. Its
image \(\proj(\Koz(n))\subseteq\mathbb R^n\) under the projection \(\proj\) of
Section \ref{sec-notation} --- which forgets the degree-\(0\) entry --- is the polytope of
Theorem \ref{thm-kozlov}, and \(\proj\) restricted to that hyperplane is a bijection onto it, so
nothing is lost either way. \textbf{\textbf{The named object always lives in the larger space, and the smaller
one is always written with \(\proj\).}}
\label{def:kozsimplex}
\label{def-kozsimplex}
\end{definition}

\begin{corollary}
\leanmark{\leanpartial\,\compmark}The Kozlov simplex has the description
\begin{equation}\label{eq:facets}
  \proj\bigl(\Koz(n)\bigr)\;=\;\Bigl\lbrace x\in\mathbb R^{n}\;:\;x_1=n,\quad
     \frac{x_1}{\binom n1}\ge\frac{x_2}{\binom n2}\ge\cdots\ge\frac{x_n}{\binom nn}\ge0\Bigr\rbrace ,
\end{equation}
and for \(n\ge2\) these \(n\) inequalities are exactly its facets. The facet opposite \(\tilde F_j\)
is cut out by
\begin{equation}\label{eq:facets-2}
  \frac{x_j}{\binom nj}=\frac{x_{j+1}}{\binom n{j+1}}\qquad(1\le j\le n-1),
\end{equation}
and the facet opposite \(\tilde F_n\) by \(x_n=0\).
\label{cor:facets}
\label{cor-facets}
\end{corollary}

\begin{proof}
Rewriting the chain of Lemma \ref{lem-mu} in the original coordinates gives the displayed description,
since \(\mu_1=1\) is the condition \(x_1=n\). The facet statement is the last sentence of that lemma,
translated the same way.
\end{proof}

\begin{mexample}
At \(n=3\) the description reads \(x_1=3\) together with
\(x_1/3\ge x_2/3\ge x_3/1\ge0\), that is
\begin{equation}\label{eq:facets-ex}
  \proj\bigl(\Koz(3)\bigr)\;=\;\bigl\lbrace x\in\mathbb R^3\;:\;x_1=3,\quad
     x_2\le3,\quad x_2\ge3x_3,\quad x_3\ge0\bigr\rbrace ,
\end{equation}
a triangle in the plane \(\lbrace x_1=3\rbrace\) with vertices
\(\tilde F_1=(3,0,0)\), \(\tilde F_2=(3,3,0)\) and \(\tilde F_3=(3,3,1)\), the \(f\)-vectors of the
three vertices of \([3]\) alone, the full graph on \([3]\), and the full simplex. Its three edges are
the three facets, each opposite the expected vertex: \(x_2=3\) joins \(\tilde F_2\) to
\(\tilde F_3\), \(x_2=3x_3\) joins \(\tilde F_1\) to \(\tilde F_3\), and \(x_3=0\) joins
\(\tilde F_1\) to \(\tilde F_2\). The middle one is the only inequality of the three that mixes two
coordinates, and it is the one the rescaling of Definition \ref{def-lym-ratios} exists to make
readable: as a statement about LYM ratios it is simply \(\mu_2\ge\mu_3\).
\label{ex-facets}
\end{mexample}
\subsection{Properness, and the permissive convention}
\label{sec-permissive}
Theorem \ref{thm-kozlov} depends on convention (iii) of Section \ref{sec-rank1}, and this subsection says
exactly how. Passing to convention (ii), which drops the requirement that \(\Delta\) contain every
singleton but keeps \(\Delta\ne\emptyset\), replaces \(\mu_1=1\) by \(\mu_1\le1\), and the hull
acquires one more vertex. Write
\begin{equation}\label{eq:Splus}
  S_n^{+}\;=\;\bigl\lbrace(f_0(\Delta),f_1(\Delta),\dots,f_n(\Delta)):
    \Delta\in\mathcal D_n,\ \Delta\ne\emptyset\bigr\rbrace\;\subseteq\;\mathbb Z^{n+1}
\end{equation}
for the permissive analogue of \(S_n\), so that \(S_n\subseteq S_n^{+}\).

The theorem is genuinely false under the permissive convention, not merely less sharp. For \(n=2\) the
order ideal \(\lbrace\emptyset,\lbrace1\rbrace\rbrace\) has \(f\)-vector \((1,1,0)\), whose projection
\((1,0)\) is not in \(\conv\lbrace(2,0),(2,1)\rbrace=\proj(\Koz(2))\); so a reader who imports a
different definition of ``simplicial complex on \(n\) vertices'' gets a false statement rather than a
weaker one. Kozlov fixes the strict convention explicitly, his §2 closing with ``\emph{From now on, \(n\)
always denotes the fixed number of vertices in our graph or complex}'', so that under his own
definitions \(\mu_1=1\) identically and both of his proofs are complete as printed. The convention is recorded
here because the statement depends on it, not because anything is missing from the original.

\begin{lemma}
\leanmark{\leanpartial}(\emph{The permissive LYM chain.}) For \(n\ge1\), with the LYM ratios of Definition \ref{def-lym-ratios},
\begin{equation}\label{eq:mu-permissive}
  \conv\lbrace0,\tilde F_1,\dots,\tilde F_n\rbrace
  \;=\;\lbrace x\in\mathbb R^n:\ 1\ge\mu_1\ge\mu_2\ge\cdots\ge\mu_n\ge0\rbrace ,
\end{equation}
and the \(n+1\) points \(0,\tilde F_1,\dots,\tilde F_n\) are affinely independent, so this set is an
\(n\)-simplex, one dimension more than the simplex of Lemma \ref{lem-mu}.
\label{lem:mu-permissive}
\label{lem-mu-permissive}
\end{lemma}

\begin{proof}
Write a point of the hull as \(x=\lambda_0\cdot0+\sum_{j=1}^n\lambda_j\tilde F_j\) with
\(\lambda_p\ge0\) and \(\sum_{p=0}^n\lambda_p=1\). Reading off the \(j\)-th coordinate gives
\(\mu_j=\sum_{\ell\ge j}\lambda_\ell\) exactly as in Lemma \ref{lem-mu}, whence
\begin{equation}\label{eq:mu-permissive-2}
  \lambda_j=\mu_j-\mu_{j+1}\quad(1\le j<n),\qquad \lambda_n=\mu_n,\qquad
  \lambda_0=1-\mu_1 .
\end{equation}
The only change from Lemma \ref{lem-mu} is the last equation: the weight on the new vertex is the slack
in \(\mu_1\le1\). Non-negativity of \(\lambda_1,\dots,\lambda_n\) is the chain
\(\mu_1\ge\cdots\ge\mu_n\ge0\) and non-negativity of \(\lambda_0\) is \(\mu_1\le1\), which proves the
equality of sets. For affine independence: the \(\tilde F_j\) are affinely independent by
Lemma \ref{lem-mu} and their affine hull lies in \(\lbrace\mu_1=1\rbrace\), which does not contain
\(0\); so adjoining \(0\) raises the dimension by one.
\end{proof}

\begin{proposition}
\leanmark{\leanfull}(\emph{Kozlov's theorem, permissive form.}) For \(n\ge1\),
\begin{equation}\label{eq:permissive-hull}
  \proj\bigl(\conv(S_n^{+})\bigr)\;=\;\conv\lbrace0,\tilde F_1,\dots,\tilde F_n\rbrace ,
\end{equation}
an \(n\)-simplex. Equivalently, before projection,
\begin{equation}\label{eq:permissive-hull-incl}
  \conv(S_n^{+})\;=\;\rasimp(\vec b),\qquad
  \vec b=\bigl(\tbinom n0,\tbinom n1,\dots,\tbinom nn\bigr).
\end{equation}
\label{prop:permissive-hull}
\label{prop-permissive-hull}
\end{proposition}

\begin{proof}
For \(\subseteq\): let \(\Delta\in\mathcal D_n\) be non-empty. Then \(f_0(\Delta)=1\), since downward
closure and non-emptiness force \(\emptyset\in\Delta\); \(\mu_1=f_1(\Delta)/n\le1\); Lemma
\ref{lem-lym} (which assumes only that \(\Delta\) is an order ideal, not that it is a complex)
gives \(\mu_j\ge\mu_{j+1}\) throughout; and \(\mu_n\ge0\) is clear. So the projected \(f\)-vector
satisfies the description of Lemma \ref{lem-mu-permissive}. For \(\supseteq\): each \(\tilde F_j\) is
the projected \(f\)-vector of the complete \(j\)-skeleton, which lies in \(\mathcal S_n\subseteq
\mathcal D_n\), and \(0=\proj(e_0)\) is the projected \(f\)-vector of
\(\lbrace\emptyset\rbrace\in\mathcal D_n\); so all \(n+1\) vertices of Lemma
\ref{lem-mu-permissive}'s simplex are attained. The unprojected form is Proposition
\ref{prop-properness-facet} below, which identifies \(\rasimp(\vec b)\) as the hull of those same
\(n+1\) points.
\end{proof}

The next proposition is the sharper way to put all of this: it locates properness rather than merely
noting its effect.

\begin{proposition}
\leanmark{\leanpartial}(\emph{Properness is a facet.}) Let \(n\ge1\) and \(\vec b=\bigl(\tbinom n0,\dots,\tbinom nn\bigr)\), and
form \(\rasimp(\vec b)\subseteq\mathbb R^{n+1}\) by Definition \ref{def-rasimp}, read with legs and
coordinates indexed \(0,1,\dots,n\) since \(\vec b\) has length \(n+1\). Then
\begin{equation}\label{eq:permissive-simplex}
  \rasimp(\vec b)\;=\;\conv\bigl(\lbrace e_0\rbrace\cup\Koz(n)\bigr),
\end{equation}
and \(\Koz(n)\) is the facet of \(\rasimp(\vec b)\) cut out by \(x_1=n\): by properness
itself. The vertex that properness removes is \(e_0\).
\label{prop:properness-facet}
\label{prop-properness-facet}
\end{proposition}

\begin{proof}
By Definition \ref{def-rasimp} the \(n+1\) vertices of \(\rasimp(\vec b)\) are \(e_0\) together with
\begin{equation}\label{eq:permissive-verts}
  \bigl(1,\tbinom n1,\dots,\tbinom nj,0,\dots,0\bigr)\qquad(1\le j\le n),
\end{equation}
and the latter are precisely the vertices of \(\Koz(n)\), which gives
\eqref{eq:permissive-simplex}. Every vertex
satisfies \(x_1\le n\), with equality at all of them except \(e_0\); so \(x_1\le n\) is a valid
inequality whose tight face is the hull of the remaining \(n\) vertices, namely \(\Koz(n)\). That face
has dimension \(n-1\) by Lemma \ref{lem-mu} while \(\rasimp(\vec b)\) has dimension \(n\) by
Lemma \ref{lem-mu-permissive}, so it is a facet.
\end{proof}

Two things are worth stating explicitly, both having been got wrong in drafting. \textbf{\textbf{\(e_0\) is a
standard basis vector, not the origin}}: it is \(\proj(e_0)\) that is the zero vector, which is why the
projected form \eqref{eq:permissive-hull} reads \(\conv\lbrace0,\tilde F_1,\dots,\tilde F_n\rbrace\).
Second, the vertex in question is a vertex of a polytope, not a vertex of a complex; as an \(f\)-vector it
is the one of \(\lbrace\emptyset\rbrace\), the order ideal with the empty face and nothing else. So
properness is not a hypothesis bolted onto Theorem \ref{thm-kozlov}: it is one of the \(n+1\) facet
inequalities of \(\rasimp(\vec b)\), and the Kozlov simplex is the face where it is tight.

Finally, \(S_n^{+}\) is not merely bigger than \(S_n\) but built from smaller copies of it. For
\(0\le k\le n\) let \(\iota_k:\mathbb R^{k+1}\to\mathbb R^{n+1}\) append \(n-k\) zeros.

\begin{proposition}
\leanmark{\leanfull}(\emph{The permissive set is fibred by \(f_1\).}) For \(n\ge0\),
\begin{equation}\label{eq:permissive-fibres}
  S_n^{+}\;=\;\bigsqcup_{k=0}^{n}\iota_k\bigl(S_k\bigr),
\end{equation}
the part with \(f_1=k\) being exactly \(\iota_k(S_k)\). The union is disjoint, so this is a partition
of \(S_n^{+}\) into \(n+1\) blocks, from the single point \(\iota_0(S_0)=\lbrace e_0\rbrace\) at
\(k=0\) up to \(S_n\) itself at \(k=n\).
\label{prop:permissive-fibres}
\label{prop-permissive-fibres}
\end{proposition}

\begin{proof}
Let \(\Delta\in\mathcal D_n\) be non-empty and put
\begin{equation}\label{eq:fibres-support}
  V\;=\;\bigl\lbrace i\in[n]\;:\;\lbrace i\rbrace\in\Delta\bigr\rbrace,
  \qquad
  k\;:=\;\lvert V\rvert\;=\;f_1(\Delta).
\end{equation}
Every \(\sigma\in\Delta\) satisfies \(\sigma\subseteq V\), since downward
closure puts each of its singletons in \(\Delta\); hence \(f_j(\Delta)=0\) for \(j>k\). Regarded as a
family on \(V\), \(\Delta\) is an order ideal containing \(\emptyset\) and every singleton of \(V\),
that is a simplicial complex on \(V\) in the strict sense. Transporting along any bijection
\(V\to[k]\) leaves every \(f_j\) unchanged, so the \(f\)-vector of \(\Delta\) is \(\iota_k\) of an
element of \(S_k\). Conversely, given \(\Delta'\in\mathcal S_k\) and any injection \([k]\hookrightarrow[n]\),
its image is a non-empty order ideal on \([n]\) with \(f_1=k\) and the same \(f\)-vector up to
\(\iota_k\); so \(\iota_k(S_k)\subseteq S_n^{+}\). Finally the blocks are disjoint because an element of
\(\iota_k(S_k)\) has first coordinate \(f_1=k\), which determines \(k\).

For \(k=0\) the argument degenerates correctly: \(V=\emptyset\) forces \(\Delta=\lbrace\emptyset\rbrace\),
\(\mathcal S_0=\lbrace\lbrace\emptyset\rbrace\rbrace\) gives \(S_0=\lbrace(1)\rbrace\), and
\(\iota_0(1)=e_0\).
\end{proof}

Figure \ref{fig:permissive-fibres} draws the whole of \(S_4^{+}\). The layer axis is \(f_1\)
itself, so the fibres of Proposition \ref{prop-permissive-fibres} are separated exactly rather than by
the accident of a projection, and the cone structure of Proposition \ref{prop-properness-facet} is
visible in the same picture: \(e_0\) alone at the bottom, the four vertices of \(\Koz(4)\) in the top
face, and the edges between. Only the top slice needs projecting at all --- the slice \(f_1=k\) has
affine dimension \(k-1\), because \(f_j\) vanishes on it for \(j>k\), so the slices with \(k\le3\)
are already planar and are drawn without distortion.

\begin{figure}[htbp]
\centering
\includegraphics[width=0.86\linewidth]{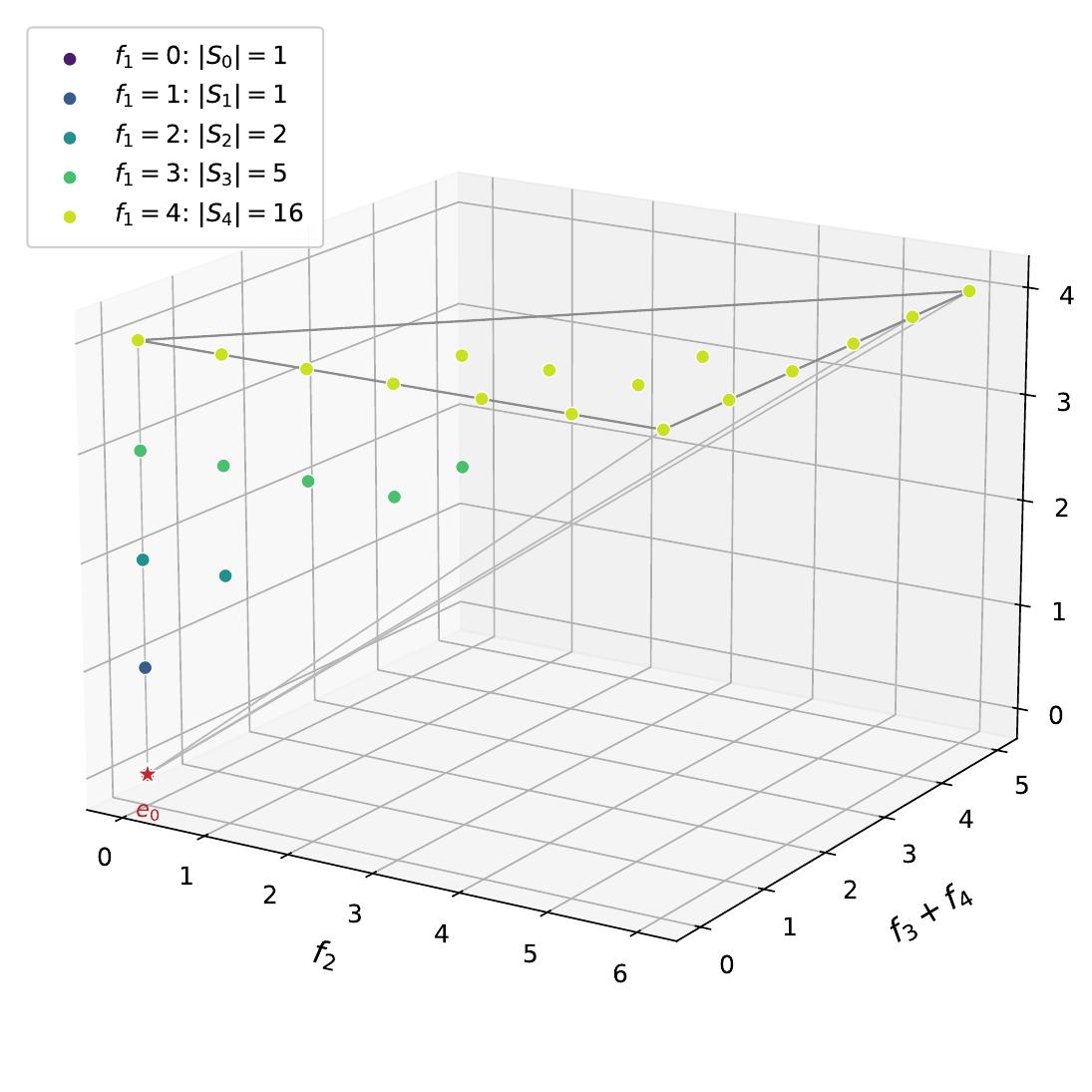}
\caption{The \(25\) points of \(S_4^{+}\), fibred by \(f_1\) as in
Proposition~\ref{prop:permissive-fibres}, inside \(\rasimp(\vec b)\) as in
Proposition~\ref{prop:properness-facet}. Slices \(f_1\le3\) are exact; the top slice, \(S_4\)
itself, is drawn in the projection \((f_2,\,f_3+f_4)\), whose second coordinate counts the faces of
cardinality at least three and which is injective on every slice at once.}
\label{fig:permissive-fibres}
\end{figure}

A picture of the same kind for seven vertices is in Kupreyeva's thesis
Kupreyeva~\cite[Figure 3.1]{Kupreyeva2019}: it draws \(\conv(F_7^3)\), the hull truncated at cardinality
three, in the two coordinates that are our \((f_2,f_3)\), and marks the achievable \(f\)-vectors
apart from the other lattice points of the hull. That distinction is the finer question this paper
does not treat --- which lattice points of the hull are achievable, as against what the hull is ---
and it is the subject of Part III of this series \autocite{SnellmanArticle3}. Two conventions differ there and are worth noting
against Section \ref{sec-rank1}'s list: that thesis admits the empty complex, convention\textasciitilde{}(i), and it
indexes \(f\)-vectors by dimension rather than by cardinality.

All of this was verified exactly at \(n=4\). Of the \(\lvert\mathcal D_4\rvert=168\) order ideals of
\(\mathcal P([4])\), the \(167\) non-empty ones realize \(25\) distinct \(f\)-vectors, whose convex hull
is \(\rasimp(\vec b)\) (five vertices, five facets, one of them \(x_1\le4\)), and the face \(x_1=4\)
has exactly the four vertices of \(\Koz(4)\); the five blocks of \eqref{eq:permissive-fibres} have
sizes \(1,1,2,5,16\), matching \(\lvert S_k\rvert\) for \(k=0,\dots,4\). The empty order ideal is
genuinely excluded and not merely absent: its \(f\)-vector is \(0\), which is not in
\(\rasimp(\vec b)\), so admitting convention (i) in place of (ii) would give a \(26\)-point set whose
hull is a different, five-dimensional polytope.
\section{Higher rank: four nested classes of submodule}
\label{sec-hierarchy}
Now let \(F=\bigoplus_{i=1}^rEg_i\) with \(\deg g_i=d_i\), \(r\ge1\).

\begin{definition}
A graded submodule \(M\subseteq F\) is an \textbf{ideal-direct-sum module}, or \textbf{ids-module}, if
\begin{equation}\label{eq:ids}
  M \;=\; I_1g_1\oplus\cdots\oplus I_rg_r
\end{equation}
for graded ideals \(I_1,\dots,I_r\subseteq E\); equivalently, if
\(M=(M\cap Eg_1)\oplus\cdots\oplus(M\cap Eg_r)\), that is, if \(M\) is already the direct sum of its
own restrictions to the summands of \(F\), with no coupling between generators. Call such an \(M\)
\textbf{proper} if \((I_i)_{\le1}=0\) for every \(i\).
\label{def:ids}
\label{def-ids}
\end{definition}

\begin{definition}
Amata and Crupi~\cite[Definition 2.1]{AmataCrupi2020} \(M\subseteq F\) is a \textbf{monomial submodule} if it is generated
by monomials \(e_\sigma g_i\) of \(F\). It is an \textbf{Amata--Crupi monomial submodule}, or a \textbf{proper}
monomial submodule, if in addition each \(I_i\) has no minimal generator of degree \(\le1\):
equivalently, each associated complex \(\Delta_i\) satisfies \(f_1(\Delta_i)=n\).
\label{def-monomial-submodule}
\end{definition}

\begin{proposition}
\leanmark{\leannone}Let \(n\ge1\), \(r\ge1\) and \(F=\bigoplus_{i=1}^rEg_i\). Then
\begin{enumerate}
\item every monomial submodule of \(F\) is an ids-module, and each of its ideals \(I_i\) is monomial;
\item every ids-module of \(F\) is a graded submodule.
\end{enumerate}
\label{prop-hierarchy}
\end{proposition}

\begin{proof}
(1) A monomial of \(F\) is \(e_\sigma g_i\) for a single \(i\), and
\(e_\tau\cdot(e_\sigma g_i)=(e_\tau e_\sigma)g_i\in Eg_i\): multiplication never moves a term between
generators' slots. So the submodule generated by a set of monomials is the direct sum of the ideals
generated by the monomials appearing at each \(g_i\), and these are monomial ideals.

(2) Each \(I_ig_i\) is graded, and a direct sum of graded submodules is graded.
\end{proof}

The two inclusions become strict under \textbf{different} hypotheses, which is worth separating because the
two conditions have nothing to do with each other: the first is about the field having room for a
non-monomial ideal, the second about there being two generators to couple.

\begin{proposition}
\leanmark{\leannone}Let \(n\ge1\), \(r\ge1\) and \(F=\bigoplus_{i=1}^rEg_i\) with \(\deg g_i=d_i\).
\begin{enumerate}
\item The monomial submodules of \(F\) are a \textbf{proper} subclass of the ids-modules if and only if
\(n\ge2\). For \(n=1\) the two classes coincide, for every \(r\) and every \(\vec d\).
\item The ids-modules of \(F\) are a \textbf{proper} subclass of the graded submodules if and only if
\(\lvert d_i-d_{i'}\rvert\le n\) for some \(i\ne i'\). In particular this requires \(r\ge2\), but
\(r\ge2\) alone does not suffice.
\end{enumerate}
\label{prop-hierarchy-strict}
\end{proposition}

\begin{proof}
(1) Let \(n\ge2\) and put \(I=(e_1+e_2)\subseteq E\), a graded ideal whose degree-\(1\) component is
the line spanned by \(e_1+e_2\) and therefore contains no monomial; so \(I\) is not a monomial ideal,
and \(Ig_1\oplus0\oplus\cdots\oplus0\) is an ids-module that is not a monomial submodule. Conversely
for \(n=1\) we have \(E=k[e]/(e^2)\), whose only graded ideals are \(0\), \((e)\) and \(E\), all
monomial;
so every ids-module has monomial components, and is then a monomial submodule.

(2) Suppose \(d_i\le d_{i'}\) with \(\delta:=d_{i'}-d_i\le n\). Choose \(\sigma\subseteq[n]\) with
\(\lvert\sigma\rvert=\delta\) and put \(v=e_\sigma g_i+g_{i'}\), which is homogeneous of degree
\(d_{i'}\), and \(M=Ev\). Looking at the \(g_{i'}\)-component, \(fv=0\) forces \(f=0\), so
\(\operatorname{Ann}(v)=0\) and \(\dim_kM=\dim_kE=2^n\). The same computation gives
\(M\cap Eg_i=0\), while \(fv\in Eg_{i'}\) holds exactly when \(fe_\sigma=0\), so
\(M\cap Eg_{i'}=\operatorname{Ann}(e_\sigma)g_{i'}\), of dimension \(2^n-2^{n-\delta}\), namely the
monomials \(e_\tau\) with \(\tau\cap\sigma\ne\emptyset\). As \(M\) meets no other summand,
\begin{equation}\label{eq:hierarchy-strict}
  \sum_{p}\dim_k(M\cap Eg_p)\;=\;2^n-2^{n-\delta}\;<\;2^n\;=\;\dim_kM ,
\end{equation}
so \(M\ne\bigoplus_p(M\cap Eg_p)\) and \(M\) is not an ids-module.

Conversely suppose \(\lvert d_i-d_{i'}\rvert>n\) for all \(i\ne i'\). The component \(Eg_i\) is
supported in degrees \([d_i,d_i+n]\), so no two components share a degree, and in each degree \(j\)
at most one is non-zero. Hence any graded \(M\) has \(M_j\subseteq Eg_{i(j)}\) for a single index
\(i(j)\), so \(M_j\subseteq M\cap Eg_{i(j)}\) and \(M=\sum_jM_j\subseteq\bigoplus_p(M\cap Eg_p)\subseteq M\).
\end{proof}

\begin{remark}
The gap condition in Proposition \ref{prop-hierarchy-strict}(2) is not an artefact. Widely separated
generator degrees make \(F\) split degreewise, and then there is simply nothing for a graded
submodule to couple: the classes coincide. The analogous threshold
reappears geometrically in Proposition \ref{prop-regimes-disjoint}, where non-overlapping windows turn
the Kozlov polytope into a Cartesian product --- but the two are off by one, and in the informative
direction. The classes coincide only for \(\lvert d_i-d_{i'}\rvert\ge n+1\), whereas the windows come
apart already at \(d\ge n\); at \(\lvert d_i-d_{i'}\rvert=n\) the polytope is a Cartesian product
while the algebra still couples, the construction of
Proposition \ref{prop-hierarchy-strict}(2) working at \(\delta=n\) with \(\lvert\sigma\rvert=n\). The
discrepancy is exactly the degree-\(0\) coordinate the geometric windows drop. Example \ref{ex-nonids} is the overlapping case with
\(\delta=1\); the case \(\delta=0\) is \(M=E\cdot(g_1+g_2)\) of Remark \ref{rem-ac-reduction}.
\label{rem-hierarchy-gaps}
\end{remark}

\begin{mexample}
Take \(n=3\), \(r=2\), \((d_1,d_2)=(0,1)\), and
\begin{equation}\label{eq:nonids}
  v\;=\;e_1e_2\,g_1+e_1\,g_2,\qquad M\;=\;E\cdot v .
\end{equation}
Both terms have degree \(2\) in \(F\) (namely \(2+0\) and \(1+1\)), so \(v\) is homogeneous.
Computing in Macaulay2: \(H_{F/M}=(1,4,5,2,0)\), while \(M\cap Eg_1=0\) and
\(M\cap Eg_2=(e_1e_2)g_2\), so the largest ids-module contained in \(M\) is
\(0\oplus(e_1e_2)g_2\), whose quotient has Hilbert function \((1,4,6,3,0)\ne(1,4,5,2,0)\). Hence
\(M\ne(M\cap Eg_1)\oplus(M\cap Eg_2)\), and \(M\) is not an ids-module.
\label{ex-nonids}
\end{mexample}

The mechanism is worth naming: \emph{every} non-zero element of \(M\) has a non-zero \(g_2\)-component, so
no non-zero element of \(M\) lies inside a single summand of \(F\). Such an \(M\) can produce a
Hilbert function that is not a shifted sum of \(f\)-vectors at all, which is exactly why
Theorem \ref{thm-main} is stated for ids-modules and not for graded submodules.

\begin{remark}
Amata and Crupi~\cite{AmataCrupi2020KK,AmataCrupi2020} For \emph{any} graded \(M\subseteq F\), \(\inn_>(M)\) is a monomial
submodule, and by Theorem \ref{thm-invariance}, \(H_{F/M}=H_{F/\inn_>(M)}\). So every graded submodule's
Hilbert function is realized by a monomial submodule --- but not necessarily by a \textbf{proper} one, and
this is the sharp point.

Take \(n=2\), \(r=2\), \(d_1=d_2=0\), \(M=E\cdot(g_1+g_2)\). Then \(H_{F/M}=(1,2,1)\), and indeed
\(\inn(M)=0\cdot g_1\oplus E\cdot g_2\), so one generator survives entirely and the other is killed
entirely, so \(\Delta_1\) is the full complex and \(\Delta_2\) is the void complex, with
\(f_1(\Delta_2)=0\ne n\). No proper monomial submodule has Hilbert function \((1,2,1)\) here, since a
proper one would force \(H_{F/N}(1)=f_1(\Delta_1)+f_1(\Delta_2)=4\). So Example \ref{ex-nonids} is not an
artefact of a bad choice: properness genuinely restricts the achievable set, and the restriction
cannot be recovered by choosing a cleverer term order.
\label{rem-ac-reduction}
\end{remark}

Inside the ids class, by contrast, nothing goes wrong.

\begin{proposition}
\leanmark{\leannone}Let \(M=I_1g_1\oplus\cdots\oplus I_rg_r\) be an ids-module and let \(>\) be \emph{any} monomial order on
\(F\). Then
\begin{enumerate}
\item The initial submodule splits along the same decomposition,
\[
     \inn_>(M)=\inn_>(I_1)g_1\oplus\cdots\oplus\inn_>(I_r)g_r ;
   \]
in particular \(\inn_>(M)\) is a monomial submodule, and an ids-module with the expected
components.
\item \(\inn_>(M)\) is proper if and only if \(M\) is; that is, \(\inn_>(M)\) is an Amata--Crupi
monomial submodule exactly when \(M\) is a proper ids-module.
\end{enumerate}
\label{prop-in-ids}
\end{proposition}

\begin{proof}
(1) Every monomial of \(F\) lies in a single component \(Eg_i\). Take \(0\ne z\in M\) and write
\(z=\sum_i f_ig_i\) with \(f_i\in I_i\); the monomials occurring in \(z\) are exactly those occurring
in the \(f_i\), each tagged by its own generator, so \(\inn_>(z)=\inn_>(f_j)g_j\) for whichever
\(j\) maximizes \(\inn_>(f_i)g_i\) among the non-zero components. Hence
\(\inn_>(M)\subseteq\bigoplus_i\inn_>(I_i)g_i\). Conversely \(f_ig_i\in M\) for every \(f_i\in I_i\)
and every \(i\), so \(\inn_>(f_i)g_i\in\inn_>(M)\), which gives the reverse inclusion. No
compatibility between \(>\) and the direct-sum decomposition is needed: a position-over-term order, a
term-over-position order, or anything else will do, because the argument only compares monomials, and
every monomial already carries its component with it.

(2) By Theorem \ref{thm-invariance} applied in rank one, \(H_{E/I_i}=H_{E/\inn_>(I_i)}\) for each \(i\).
Properness of a rank-one quotient is a condition on its Hilbert function alone --- it says
\(H_{E/I_i}(1)=n\), so \(I_i\) is proper if and only if \(\inn_>(I_i)\) is. Now apply (1)
componentwise.
\end{proof}

This sharpens Remark \ref{rem-ac-reduction} in the one place where sharpening is possible. For a general
graded \(M\), the reduction says only that \(\inn_>(M)\) is \textbf{some} monomial submodule, and
Remark \ref{rem-ac-reduction}'s example shows properness can be destroyed outright. Inside the ids class
neither happens. So the loss of control is located precisely at the ids boundary, not at the monomial
one --- a second reason, independent of Example \ref{ex-nonids}, to state Theorem \ref{thm-main} where we do.
A consequence worth recording: in Theorem \ref{thm-main} one may replace ``proper monomial submodule'' by
``proper ids-module'' throughout without changing the set whose hull is being taken. We keep the
monomial phrasing because it is Amata and Crupi's, and because it is the one their software
implements.
\section{The combinatorial gadget: \(r\)-vectors of complexes and their \emph{ff}-vectors}
\label{sec-combinatorial}
\begin{definition}
Let \(n\ge1\) and \(r\ge1\). An \textbf{\(r\)-vector of simplicial complexes on
\([n]\)} is a tuple \(\vec\Delta=(\Delta_1,\dots,\Delta_r)\) with each \(\Delta_i\in\mathcal S_n\), chosen
\textbf{independently}: no relation whatever is imposed between them. Fix a shift vector
\(\vec d=(d_1,\dots,d_r)\) with \(0=d_1\le\cdots\le d_r\), and put \(N=n+d_r\). Extend
\(f_j(\Delta_i)\) by \(0\) for \(j<0\) and for \(j>n\). The \textbf{\emph{ff}-vector} of \(\vec\Delta\) with
respect to \(\vec d\) is
\begin{equation}\label{eq:ffvector}
  \ff(\vec\Delta)_j\;=\;\sum_{i=1}^{r}f_{j-d_i}(\Delta_i),\qquad j=0,1,\dots,N .
\end{equation}
\label{def:ffvector}
\label{def-ffvector}
\end{definition}

\begin{mexample}
Let \(n=3\), \(r=3\) and \(\vec d=(0,1,1)\), so \(N=4\); nothing forbids the tie \(d_2=d_3\). Take
\(\Delta_1\) the full simplex on \([3]\), \(\Delta_2\) its boundary (every subset of size at most
\(2\)) and \(\Delta_3\) the three isolated vertices, with \(f\)-vectors
\((1,3,3,1)\), \((1,3,3,0)\) and \((1,3,0,0)\). Each \(j\) sums one entry from each complex, read at
its own offset \(j-d_i\):
\begin{equation}\label{eq:ffvector-ex}
  \ff(\vec\Delta)_2\;=\;f_2(\Delta_1)+f_1(\Delta_2)+f_1(\Delta_3)\;=\;3+3+3\;=\;9,
\end{equation}
and altogether \(\ff(\vec\Delta)=(1,5,9,4,0)\). The two ends are where the extension by \(0\) does its
work: \(\ff_0=1\) because \(f_{-1}(\Delta_2)=f_{-1}(\Delta_3)=0\), and \(\ff_4=0\) because
\(f_4(\Delta_1)=0\) and \(f_3\) of the other two vanishes.

Two checks are worth having in hand. The total \(1+5+9+4+0=19\) is the total number of faces,
\(8+7+4\), since every face of every \(\Delta_i\) is counted exactly once. And
\(\ff_0=\lvert\lbrace i:d_i=0\rbrace\rvert\) always, each \(\Delta_i\in\mathcal S_n\) having
\(f_0=1\): the \(0\)-th entry counts the generators sitting in degree \(0\) and nothing else. That
last identity is the combinatorial shadow of the translation term in \eqref{eq:koz-as-ra} --- the
entry \(\proj\) forgets is exactly the one those generators contribute.
\label{ex-ffvector}
\end{mexample}

The doubled letter is deliberate: it marks an invariant of the \emph{tuple}, keeping it visually distinct
from the \(f\)-vector of a single complex.

\begin{definition}
For \(\vec\Delta\) as above, put
\begin{equation}\label{eq:indicator-module}
  M(\vec\Delta)\;=\;\bigl(I_{\Delta_1}\bigr)g_1\oplus\cdots\oplus\bigl(I_{\Delta_r}\bigr)g_r,
\end{equation}
which is an Amata--Crupi monomial submodule of \(F\), and call the quotient
\(\IM(\vec\Delta):=F/M(\vec\Delta)\) the \textbf{indicator module} of \(\vec\Delta\).
\label{def-indicator-module}
\end{definition}

For \(r=1\), \(d_1=0\) this is Sköldberg's indicator algebra \(k\lbrace\Delta\rbrace\) exactly.

\begin{proposition}
\leanmark{\leannone\,\compmark}Let \(M=I_1g_1\oplus\cdots\oplus I_rg_r\) be an ids-module, with \(\Delta_i\) the downward-closed
family attached to \(\inn(I_i)\). Then
\begin{equation}\label{eq:splitting}
  F/M\;\cong\;\bigoplus_{i=1}^{r}(E/I_i)(-d_i),
  \qquad H_{F/M}(t)\;=\;\sum_{i=1}^{r}t^{d_i}H_{E/I_i}(t),
\end{equation}
and coordinatewise
\begin{equation}\label{eq:splitting-2}
  H_{F/M}(j)\;=\;\sum_{i=1}^{r}H_{E/I_i}(j-d_i)\;=\;\sum_{i=1}^{r}f_{j-d_i}(\Delta_i).
\end{equation}
In particular \(H_{\IM(\vec\Delta)}=\ff(\vec\Delta)\).
\label{prop:splitting}
\label{prop-splitting}
\end{proposition}

\begin{proof}
Because \(M\) is a direct sum, we have \(F_j=\bigoplus_i(E_{j-d_i})g_i\) and
\(M_j=\bigoplus_i(I_i)_{j-d_i}g_i\) in every degree, so the quotient splits degreewise and the
isomorphism is the evident one. The last
equality is Proposition \ref{prop-dictionary}(1) applied to each summand, together with
Theorem \ref{thm-invariance} to replace \(I_i\) by \(\inn(I_i)\).
\end{proof}

\begin{definition}
For \(d\ge0\) let \(\shift^d:\mathbb R^{n+1}\to\mathbb R^{N+1}\) prepend \(d\) zeros to the whole
vector --- degree-\(0\) entry included --- and pad with zeros on the right.
\label{def-shift}
\end{definition}

This is a \emph{linear} map, namely the coordinate inclusion \(e_j\mapsto e_{j+d}\). Linearity, rather
than the affine ``translation'' one might reach for first, is what makes Lemma \ref{lem-convex}(B)
applicable and what makes Theorem \ref{thm-main} come out with no leftover translation vector.

\begin{proposition}
\leanmark{\leanpartial}With \(S_n\) as in Definition \ref{def-proper} and \(d_1=0\le d_2\le\cdots\le d_r\),
\begin{equation}\label{eq:sumset}
  \bigl\lbrace H_{F/M}: M\ \text{a proper monomial submodule of }F\bigr\rbrace
  \;=\;\bigl\lbrace\ff(\vec\Delta):\vec\Delta\ \text{an }r\text{-vector on }[n]\bigr\rbrace,
\end{equation}
and both sets are equal to
\begin{equation}\label{eq:sumset-2}
  S_n+\shift^{d_2}(S_n)+\cdots+\shift^{d_r}(S_n),
\end{equation}
an ordinary Minkowski sum of finite subsets of \(\mathbb Z^{N+1}\), the sum being over the \emph{multiset}
\(\lbrace d_1,\dots,d_r\rbrace\) so that repeated degrees contribute repeated summands.
\label{prop:sumset}
\label{prop-sumset}
\end{proposition}

\begin{proof}
The first equality is Proposition \ref{prop-splitting} together with the bijection
\(\Delta\leftrightarrow I_\Delta\) of Proposition \ref{prop-dictionary}(2): a proper monomial submodule is
precisely an \(M(\vec\Delta)\), and its quotient's Hilbert function is \(\ff(\vec\Delta)\).

For the second, let \(s_i\in S_n\) be the \(f\)-vector of \(\Delta_i\). By construction
\(\shift^{d}(s)_j=f_{j-d}(\Delta)\) for \emph{every} \(j\in\lbrace0,\dots,N\rbrace\), since both sides are
already zero outside \(d\le j\le d+n\), by the convention \(f_j=0\) off \([0,n]\). Hence
\(\ff(\vec\Delta)_j=\sum_i\shift^{d_i}(s_i)_j\) for every \(j\), i.e.
\(\ff(\vec\Delta)=\sum_i\shift^{d_i}(s_i)\), one term drawn from each \(\shift^{d_i}(S_n)\). Since
the \(\Delta_i\) range independently and without constraint over \(\mathcal S_n\), the tuple
\((s_1,\dots,s_r)\) ranges over all of \(S_n\times\cdots\times S_n\), which is exactly the definition
of the Minkowski sum.
\end{proof}

The proof is immediate, deliberately. The content is not in any hard combinatorics but in choosing
the defining formula for \(\ff\) so that independence of the \(\Delta_i\) becomes, coordinatewise, a
sumset. Nothing here uses Kruskal--Katona or Kozlov, or any property of \(S_n\) beyond its being some
fixed finite subset of \(\mathbb Z^{n+1}\) all of whose elements begin \((1,n)\).
\section{The main theorem}
\label{sec-main}
Two definitions first. Both are used in the statement of the theorem, and both are about the
\textbf{shape} of the summands rather than about the algebra, so it is worth having them named before the
sum is formed.

\begin{definition}
Let \(\vec a=(a_1,\dots,a_n)\) be a vector of positive integers, called a \textbf{leg vector}. The
\textbf{right-angle simplex} \(\rasimp(\vec a)\) is the convex hull of the \(n\) points
\begin{equation}\label{eq:rasimp}
  v_j=(a_1,\dots,a_j,0,\dots,0),\qquad j=1,\dots,n .
\end{equation}
For \(0\le d_1\le\cdots\le d_r\) the \textbf{right-angle polytope} is the Minkowski sum of shifted copies
of it,
\begin{equation}\label{eq:rasimp-1}
  \RA(\vec a;d_1,\dots,d_r)\;=\;\sum_{i=1}^{r}\shift^{d_i}\bigl(\rasimp(\vec a)\bigr).
\end{equation}
\label{def-rasimp}
\end{definition}

The name records the shape: consecutive vertices differ in one coordinate, so the edges out of
\(v_1\) run along the coordinate axes and \(\rasimp(\vec a)\) is the simplex on a right-angled
corner with legs \(a_1,\dots,a_n\). The leg vector fixes the metric shape; the shift vector
\(\vec d\) fixes how the copies overlap.

\begin{definition}
The \textbf{Kozlov vector} is the leg vector of binomial coefficients
\begin{equation}\label{eq:kozvec}
  \kozvec(n)\;=\;\Bigl(\tbinom n1,\tbinom n2,\dots,\tbinom nn\Bigr),
\end{equation}
of length \(n\), not \(n+1\): a leg vector records the increments of a staircase, and
\(\rasimp\) of Definition \ref{def-rasimp} builds from it a simplex in \(\mathbb R^n\), so that
\(\proj(\Koz(n))=\rasimp(\kozvec(n))\). There is therefore no degree-\(0\) companion of
\(\kozvec\) to define: the projection is applied to the polytope, never to the leg vector.
\label{def:kozvec}
\label{def-kozvec}
\end{definition}

Since \(v_j=\tilde F_j\) for this leg vector, the Kozlov simplex of Section \ref{sec-kozlov} projects
to the right-angle simplex on it, and the Kozlov polytope to a translate of the corresponding
right-angle polytope:
\begin{equation}\label{eq:koz-as-ra}
  \proj\bigl(\Koz(n)\bigr)=\rasimp\bigl(\kozvec(n)\bigr),\qquad
  \proj\bigl(\Koz(n;d_1,\dots,d_r)\bigr)=\RA\bigl(\kozvec(n);d_1,\dots,d_r\bigr)
     +\sum_{i\,:\,d_i\ge1}e_{d_i} .
\end{equation}
The translation is unavoidable and is not a normalization one may drop: \(\proj\) forgets the
coordinate in degree \(0\) only, whereas \(\shift^{d_i}\) has already carried the \(i\)-th summand's
degree-\(0\) entry, which is \(1\), up to coordinate \(d_i\), where it survives the projection.
The sum therefore runs over the indices with \(d_i\ge1\), not over \(i\ge2\), a distinction that
matters exactly when two or more generators tie at degree \(0\); Remark \ref{rem-linearity} and
Section \ref{sec-vertices} say the same thing from the other side. At \(r=1\), \(d_1=0\) the term is
empty, which is the first equality. Taking instead the \textbf{all-ones leg vector}
\(\onesvec(n):=(1,\dots,1)\) of length \(n\) gives a
combinatorially simpler comparison object, used throughout Section \ref{sec-vertices}.

\begin{lemma}
\leanmark{\leannone}(A) For \(U,V\subseteq\mathbb R^N\), \(\conv(U+V)=\conv(U)+\conv(V)\); by induction and
associativity of Minkowski addition the same holds for any finite number of summands.
(B) For linear \(L:\mathbb R^N\to\mathbb R^{N'}\) and \(U\subseteq\mathbb R^N\),
\(\conv(L(U))=L(\conv(U))\).
\label{lem-convex}
\end{lemma}

\begin{proof}
For the inclusion \(\subseteq\) in (A), observe first that \(\conv(U)+\conv(V)\) is convex: given
\(x_1+y_1\) and \(x_2+y_2\) in it and \(\lambda\in[0,1]\),
\begin{equation}\label{eq:convex}
  \lambda(x_1+y_1)+(1-\lambda)(x_2+y_2)
  =\bigl(\lambda x_1+(1-\lambda)x_2\bigr)+\bigl(\lambda y_1+(1-\lambda)y_2\bigr).
\end{equation}
It contains \(U+V\), and a convex hull is the smallest convex superset. (A, \(\supseteq\)) If \(x=\sum_p\lambda_pu_p\) and \(y=\sum_q\mu_qv_q\) are convex
combinations then \(x+y=\sum_{p,q}\lambda_p\mu_q(u_p+v_q)\) with \(\sum_{p,q}\lambda_p\mu_q=1\) and
all coefficients \(\ge0\), a convex combination of points of \(U+V\).
For (B), the identity \(L(\sum_p\lambda_pu_p)=\sum_p\lambda_pL(u_p)\) gives both inclusions at
once.
\end{proof}

\begin{theorem}
\leanmark{\leanpartial\,\compmark}Let \(n\ge1\), \(r\ge1\), and \(0=d_1\le d_2\le\cdots\le d_r\), and let
\(F=\bigoplus_{i=1}^rEg_i\), \(\deg g_i=d_i\). Then, in the degree-\(0\)-inclusive coordinates of
Definition \ref{def-shift},
\begin{equation}\label{eq:main}
  \conv\bigl\lbrace H_{F/M}: M\ \text{proper monomial}\bigr\rbrace
  \;=\;\sum_{i=1}^{r}\shift^{d_i}\bigl(\Koz(n)\bigr)\;=\;\Koz(n;d_1,\dots,d_r),
\end{equation}
a Minkowski sum of \(r\) shifted copies of one and the same Kozlov simplex, one copy per generator
degree. Equivalently, by Proposition \ref{prop-sumset}, \(\Koz(n;d_1,\dots,d_r)\) is the convex hull of
the set of \emph{ff}-vectors of \(r\)-vectors of simplicial complexes on \([n]\).
\label{thm:main}
\label{thm-main}
\end{theorem}

\begin{proof}
By Proposition \ref{prop-sumset} the achievable set is \(\sum_i\shift^{d_i}(S_n)\). By
Lemma \ref{lem-convex}(A) applied \(r-1\) times and then Lemma \ref{lem-convex}(B) applied to each linear map
\(\shift^{d_i}\),
\begin{equation}\label{eq:main-2}
  \conv\Bigl(\sum_{i=1}^r\shift^{d_i}(S_n)\Bigr)
  =\sum_{i=1}^r\conv\bigl(\shift^{d_i}(S_n)\bigr)
  =\sum_{i=1}^r\shift^{d_i}\bigl(\conv(S_n)\bigr),
\end{equation}
and \(\conv(S_n)=\Koz(n)\) by Theorem \ref{thm-kozlov} --- every element of \(S_n\) has degree-\(0\)
entry \(1\), so the hull lies in \(\lbrace x_0=1\rbrace\) and is the degree-\(0\)-inclusive copy of
the Kozlov simplex.
\end{proof}

The proof cites nothing beyond the definition of a convex hull and elementary linear algebra: step
one is Proposition \ref{prop-splitting} plus the definition of a monomial submodule, step two is
Proposition \ref{prop-sumset}, step three is Lemma \ref{lem-convex}, and step four is Theorem \ref{thm-kozlov},
itself proved above from Lemma \ref{lem-lym}. In particular the Gröbner-theoretic reduction of
Remark \ref{rem-ac-reduction} is \emph{not} used: it answers the different and broader question of which
Hilbert functions occur for arbitrary graded \(M\).

\begin{remark}
We call \(\Koz(n;d_1,\dots,d_r)\) the \textbf{Kozlov polytope} and reserve \textbf{Kozlov simplex} for
\(\Koz(n)=\Koz(n;0)\). The two are worth keeping apart by name because for \(r\ge2\) the sum is
essentially never a simplex: already \(\Koz(3;0,1)\) of Example \ref{ex-worked} is \(3\)-dimensional with
\(7\) vertices. A simplex is exactly what one loses by passing to rank \(r\), and it is the reason
Section \ref{sec-vertices} exists at all.
\label{rem-name}
\end{remark}

\begin{corollary}
\leanmark{\leannone}Let \(c\ge0\) and \(d_1\ge0\). Writing \(\mathcal H(d_1,\dots,d_r)\) for the set of proper Hilbert
functions,
\begin{equation}\label{eq:normalization}
  \mathcal H(d_1+c,\dots,d_r+c)=\shift^c\bigl(\mathcal H(d_1,\dots,d_r)\bigr),
\end{equation}
\begin{equation}\label{eq:normalization-2}
  \Koz(n;d_1+c,\dots,d_r+c)=\shift^{c}\bigl(\Koz(n;d_1,\dots,d_r)\bigr),
\end{equation}
so Theorem \ref{thm-main} for \((d_1,\dots,d_r)\) is equivalent to Theorem \ref{thm-main} for
\((d_1+c,\dots,d_r+c)\), and one may normalize \(d_1=0\).
\label{cor-normalization}
\end{corollary}

\begin{proof}
The module \(F(-c)\) is the same module with its degrees relabelled by \(j\mapsto j+c\).
Submodules correspond and Hilbert functions translate, which gives the first identity. For the second, \(\shift^c\) is linear,
hence distributes over Minkowski sums, and \(\shift^c\circ\shift^{d}=\shift^{c+d}\); both facts are
independent of the number of summands, so nothing here is \(r\)-specific. Finally
\(\conv\circ\shift^c=\shift^c\circ\conv\) by Lemma \ref{lem-convex}(B).
\end{proof}

\begin{remark}
The linearity of \(\shift\) is exactly what makes this work for every \(r\). A \emph{translation} \(t\)
would satisfy \(t(A+B)=t(A)+B\), not \(t(A)+t(B)\), and the \(r\)-fold version would pick up a factor
\(r\). This is why the theorem is stated in degree-\(0\)-inclusive coordinates, in which the
operation is genuinely linear, and it is the reason the alternative coordinates of
Section \ref{sec-vertices}, which drop each summand's leading \(1\), need an explicit translation
term to compare with.
\label{rem-linearity}
\end{remark}

\begin{proposition}
\leanmark{\leanpartial}(\emph{Edge cases.}) Theorem \ref{thm-main} behaves as follows at its boundary.
\begin{enumerate}
\item For \(r=1\): then \(d_1=0\), \(F=E\), \(M=I_1\), and the statement is
\(\conv(S_n)=\Koz(n)\), which is Kozlov's theorem, with no index shift and no leftover
translation.
\item For tied degrees \(d_i=d_{i'}\): the conclusion is unchanged. Nothing in the proof compares
\(d_i\) to \(d_{i'}\), and Proposition \ref{prop-sumset}'s sumset is over a multiset.
\item For \(n=1\): \(\mathcal S_1\) has a single element, so \(S_1=\lbrace(1,1)\rbrace\) and
\(\Koz(1)\) is a point; both sides reduce to the single point \(\sum_i\shift^{d_i}(1,1)\).
\item For \(n=0\) the statement is excluded, since the chain of identifications it rests on breaks
rather than degenerating: the vertex list \(\tilde F_1,\dots,\tilde F_n\) of
Theorem \ref{thm-kozlov} is empty at \(n=0\), so that theorem has no content there and
Lemma \ref{lem-mu}'s simplex is not defined. (Both sides of Theorem \ref{thm-main} are in fact a
point, \(\mathcal S_0=\lbrace\lbrace\emptyset\rbrace\rbrace\) giving \(S_0=\lbrace(1)\rbrace\); it
is the proof, not the statement, that has nothing to work with.)
\item Dropping properness replaces \(\Koz(n)\) by
\(\conv\lbrace0,\tilde F_1,\dots,\tilde F_n\rbrace\) in every summand.
\end{enumerate}
\label{prop:edge-cases}
\label{prop-edge-cases}
\end{proposition}

\begin{proof}
(1)--(4) are read off the statement, given Kozlov's theorem for (1) and the description of
\(\mathcal S_n\) for (3) and (4); (5) is Proposition \ref{prop-permissive-hull} applied summand by summand.
\end{proof}

Ties at the \textbf{initial} degree are worth a word, because they are exactly where the exact
characterization of Part V \autocite{SnellmanArticle5} has been misstated in the literature.

\begin{mexample}
Let \(n=3\), \(r=2\), \(d=(0,1)\). The simplex \(\proj(\Koz(3))\) has vertices \((3,0,0)\),
\((3,3,0)\), \((3,3,1)\).
Take \(\Delta_1\) the full simplex on \([3]\) (so \(I_{\Delta_1}=0\), \(f(\Delta_1)=(1,3,3,1)\)) and
\(\Delta_2\) the three isolated vertices (so \(I_{\Delta_2}=(e_1e_2,e_1e_3,e_2e_3)\),
\(f(\Delta_2)=(1,3,0,0)\)). Then \(\ff(\vec\Delta)=(1,4,6,1,0)\); for instance
\(\ff_2=f_2(\Delta_1)+f_1(\Delta_2)=3+3=6\). Building
\(M(\vec\Delta)=I_{\Delta_1}g_1\oplus I_{\Delta_2}g_2\) in Macaulay2 gives
\(H_{\IM(\vec\Delta)}=(1,4,6,1,0)\), as Proposition \ref{prop-splitting} requires.

The polytope
\begin{equation}\label{eq:worked}
  \Koz(3;0,1)\;=\;\Koz(3)+\shift^1\bigl(\Koz(3)\bigr)\;\subseteq\;\mathbb R^5 ,
\end{equation}
is \(3\)-dimensional, not \(4\),
with \(7\) vertices:

\begin{equation}\label{eq:worked-2}
  (1,4,3,0,0),\quad (1,4,3,3,0),\quad (1,4,3,3,1),\quad (1,4,6,0,0),
\end{equation}
\begin{equation}\label{eq:worked-3}
  (1,4,6,3,1),\quad (1,4,6,4,0),\quad (1,4,6,4,1).
\end{equation}
Note that \(\ff(\vec\Delta)=(1,4,6,1,0)\) is \emph{not} a vertex. It is not interior either: it lies on
the boundary, being tight on the two facet inequalities \(x_2\le6\) and \(x_4\ge0\): its
\(H(2)=6\) is the largest value the sum admits in that degree and its \(H(4)=0\) the smallest. What
it exhibits is an achievable lattice point that is not a vertex, which is all the distinction between
the two questions needs.

Figure \ref{fig:summands-and-sum} draws the same construction at \(n=3\) and \(\vec d=(0,1)\) for two
other leg vectors, the all-ones one and \(\vec a=(1,2,1)\), rather than for the Kozlov vector
\(\kozvec(3)=(3,3,1)\) of this example. Both have polytopal face-vector \((1,7,11,6,1)\): seven
vertices, eleven edges, six
facets. They are therefore combinatorially the \emph{same} polytope and differ only in shape --- as does
the Kozlov sum of this example, which is why the figure loses nothing by not drawing it. That is not
an accident of the particular leg vectors but forced at this size: \(m=n-d=2\) here, and
Section \ref{sec-vertices} shows that no positive leg vector can produce a different face lattice
until \(m=3\).
\label{ex-worked}
\end{mexample}

Each summand is a triangle, a two-dimensional right-angle simplex, and the two sit in different
coordinate planes, overlapping in the single coordinate \(x_3\); the Minkowski sum of two
triangles meeting in this way is the three-dimensional body in the right-hand column. Comparing
the rows of the figure, the leg vector stretches the second leg of each triangle, and that is the
whole of the difference between them.

\begin{figure}[htbp]
\centering
\setlength{\tabcolsep}{2pt}
\begin{tabular}{@{}c@{\ }c@{\ }c@{\ }c@{\ }c@{}}
{\footnotesize $\rasimp(\vec a)$} & &
{\footnotesize $\shift^1\rasimp(\vec a)$} & &
{\footnotesize $\RA(\vec a;0,1)=\rasimp(\vec a)+\shift^1\rasimp(\vec a)$}\\[-2pt]
$\vcenter{\hbox{\includegraphics[width=0.27\linewidth]{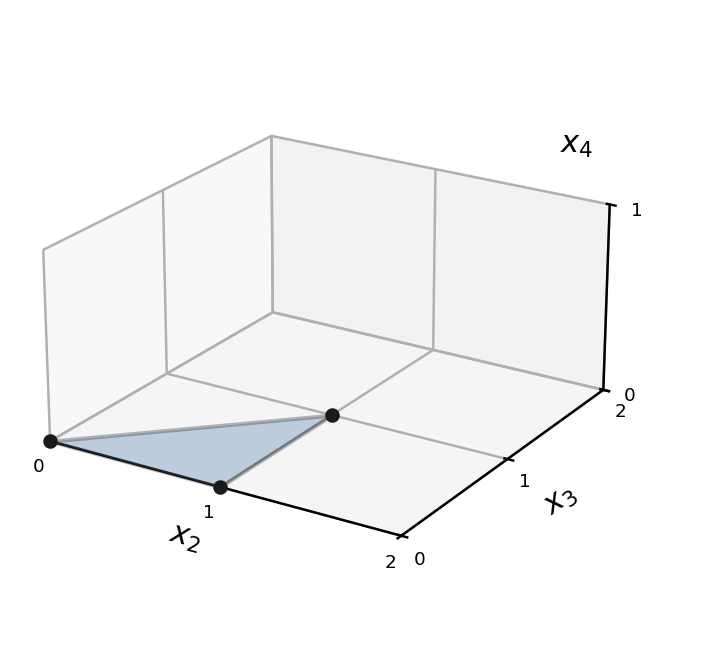}}}$ & $+$ &
$\vcenter{\hbox{\includegraphics[width=0.27\linewidth]{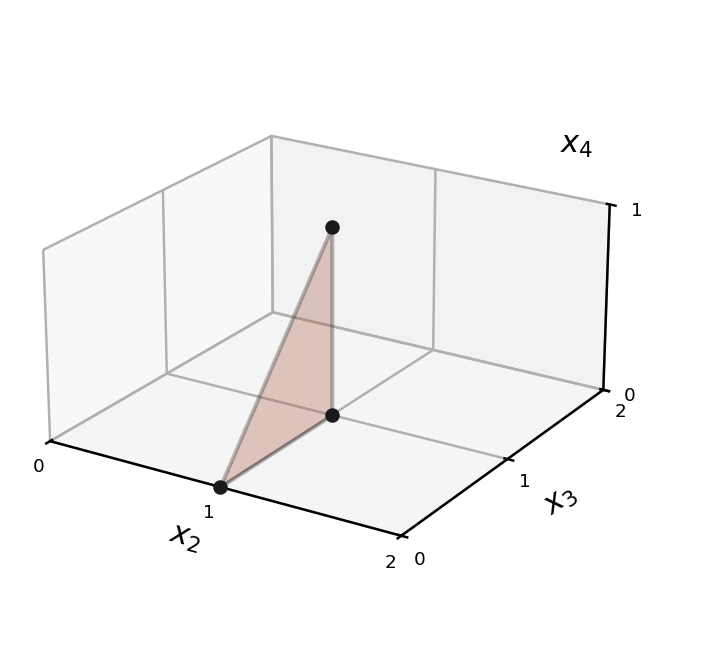}}}$ & $=$ &
$\vcenter{\hbox{\includegraphics[width=0.27\linewidth]{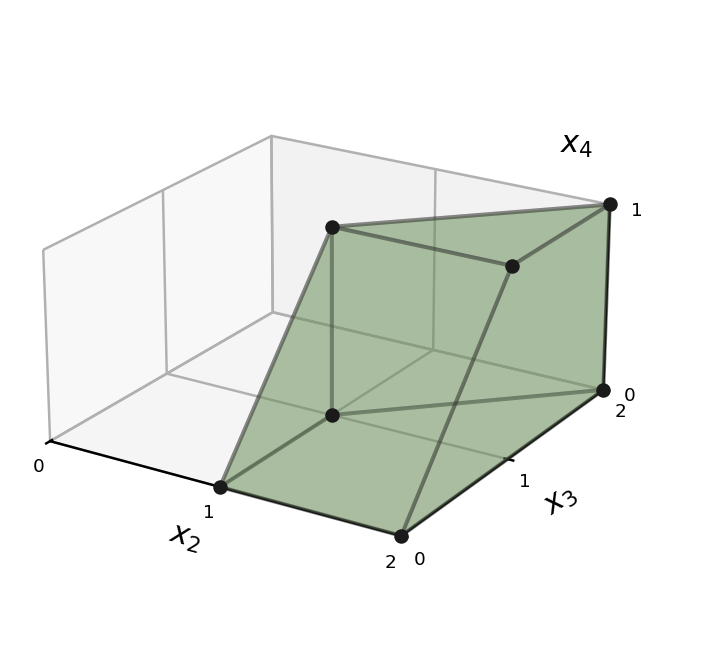}}}$\\[-8pt]
\multicolumn{5}{c}{\footnotesize $\vec a=(1,1,1)$, the all-ones vector}\\[6pt]
$\vcenter{\hbox{\includegraphics[width=0.27\linewidth]{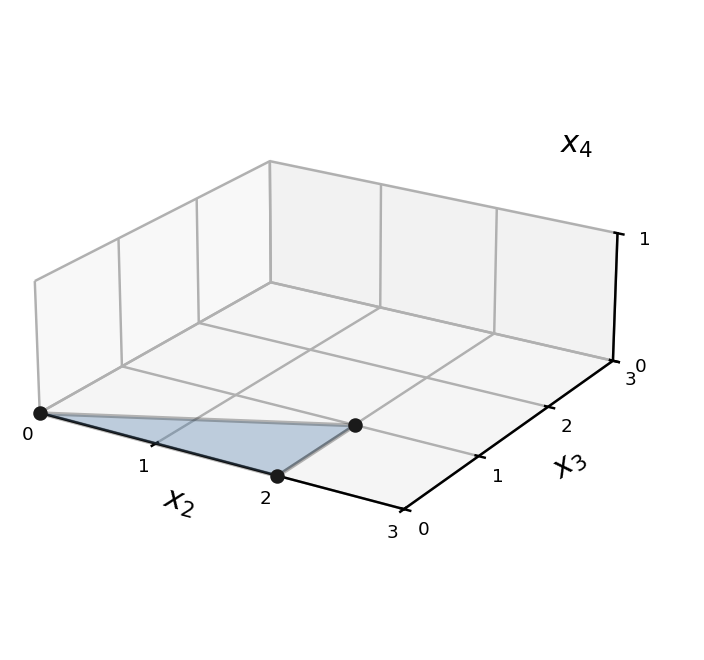}}}$ & $+$ &
$\vcenter{\hbox{\includegraphics[width=0.27\linewidth]{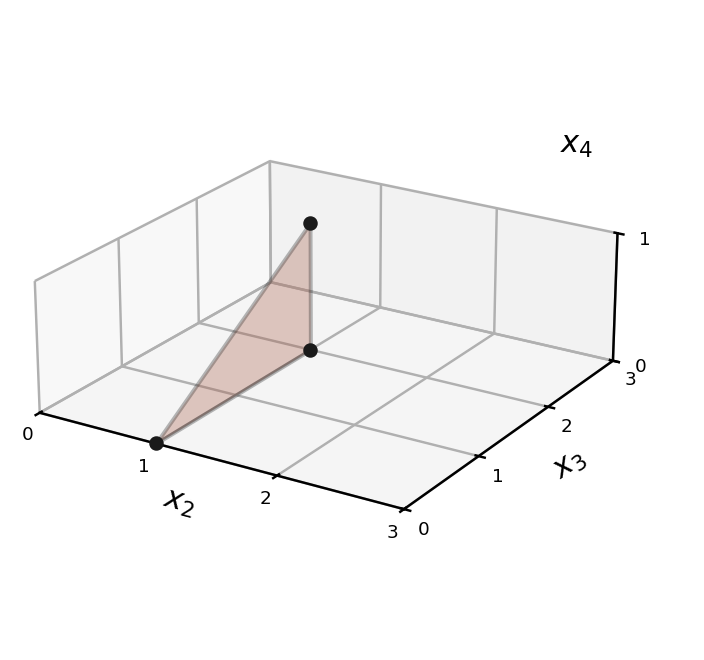}}}$ & $=$ &
$\vcenter{\hbox{\includegraphics[width=0.27\linewidth]{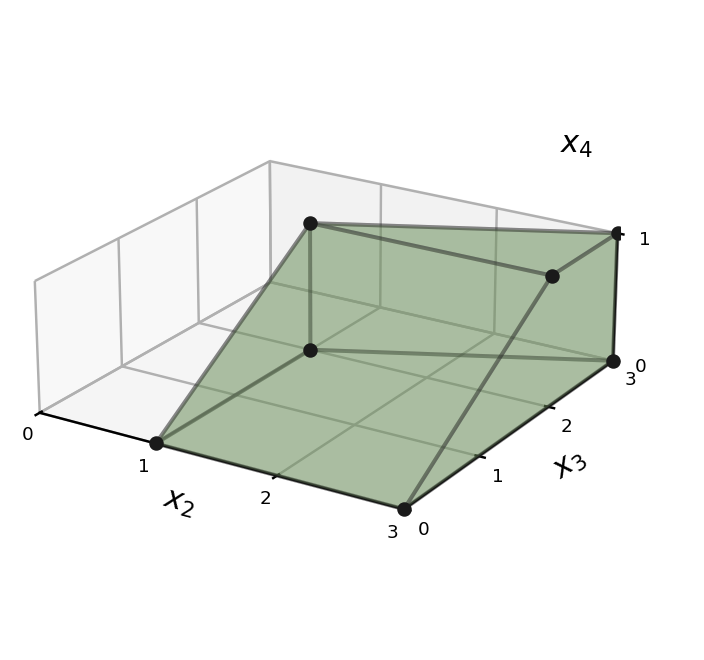}}}$\\[-8pt]
\multicolumn{5}{c}{\footnotesize $\vec a=(1,2,1)$, log-concave}
\end{tabular}
\caption{Summands and their Minkowski sum, $n=3$ and $\vec d=(0,1)$, in the coordinates
$(x_2,x_3,x_4)$ after dropping the constant $x_1$. Within each row the three panels share one
viewpoint and one scale.}
\label{fig:summands-and-sum}
\end{figure}

The drop in dimension is structural, and the general statement is easy.

\begin{proposition}
\leanmark{\leannone\,\compmark}Let \(n\ge1\), \(r\ge1\) and \(0\le d_1\le\cdots\le d_r\). Then
\begin{equation}\label{eq:dimension}
  \dim\Koz(n;d_1,\dots,d_r)\;=\;\Bigl|\;\bigcup_{i=1}^{r}\lbrace d_i+2,\;d_i+3,\;\dots,\;d_i+n\rbrace\;\Bigr| .
\end{equation}
This equals \(r(n-1)\) exactly when the \(r\) index ranges are pairwise disjoint, i.e. when
\(d_{i+1}-d_i\ge n-1\) for all \(i\), and is strictly smaller otherwise.
\label{prop:dimension}
\label{prop-dimension}
\end{proposition}

\begin{proof}
Every element of \(S_n\) has its first two coordinates constant, \((1,n)\), so the direction space
--- the linear span of the differences of its points --- of
\(\Koz(n)\) is \(\operatorname{span}\lbrace e_2,\dots,e_n\rbrace\), of dimension \(n-1\), and that of
\(\shift^{d}(\Koz(n))\) is \(\operatorname{span}\lbrace e_{d+2},\dots,e_{d+n}\rbrace\). The direction
space of a Minkowski sum is the sum of the summands' direction spaces.
\end{proof}

For \(n=3\), \(d=(0,1)\) this gives \(|\lbrace2,3\rbrace\cup\lbrace3,4\rbrace|=3\), as computed. So a Minkowski sum of shifted Kozlov simplices is rarely
full-dimensional in the space \(\mathbb R^N\) it is presented in.
\section{The vertex and facet structure of the Kozlov polytope at rank two}
\label{sec-vertices}
Theorem \ref{thm-main} describes the polytope as a Minkowski sum. It does not say which points are its
vertices, and that is a genuinely separate question: a vertex of a Minkowski sum is always a sum of
vertices of the summands, but most such sums are not vertices. This section records what is known.
\textbf{\textbf{Where everything in this section lives.}} The object under study is the right-angle polytope of
Definition \ref{def-rasimp},
\begin{equation}\label{eq:dimension-2}
  \RA(\vec a;d_1,\dots,d_r)\;\subseteq\;\mathbb R^N,
  \qquad N\;=\;n+d_r ,
\end{equation}
\(d_r\) being the largest shift, since \(\shift^{d_i}\) prepends \(d_i\) zeros to a vector of length
\(n\), so the \(i\)-th summand lies in \(\mathbb R^{n+d_i}\) and the sum in the largest of these.
At rank two, with \(\vec d=(0,d)\), this is \(N=n+d\).

These are the coordinates \textbf{without} the degree-\(0\) entry: every element of \(S_n\) has
its first coordinate equal to \(1\), and the projection that forgets that constant entry turns each
summand into a right-angle simplex. Theorem \ref{thm-main} is stated in the degree-\(0\)-inclusive
coordinates of Definition \ref{def-shift}; carrying it across that projection costs one translation,
by \(\sum_{i:\,d_i\ge1}e_{d_i}\), the sum running over the indices with \(d_i\ge1\) and not
simply over \(i\ge2\), a distinction that matters exactly when two or more generators tie at
degree \(0\).

Every sum of one vertex from each summand is a candidate vertex of \(\RA(\vec a;d_1,\dots,d_r)\);
which of them actually are vertices is recorded by the following tensor, and is the subject of the
section.

Throughout this section \textbf{vertex} means extreme point: a point of a polytope \(P\) that is not an
interior point of any segment with both endpoints in \(P\), equivalently a zero-dimensional face of
\(P\). Lemma \ref{lem-vertexdecomp}(1), in Appendix \ref{sec-app-criterion}, converts this into the
working criterion the proofs actually use, namely
unique maximization of a linear functional.

\begin{definition}
Let \(n\ge1\), \(r\ge1\), let \(\vec a\) be a positive leg vector of length \(n\) and let
\(0\le d_1\le\cdots\le d_r\). The \textbf{extremal tensor} \(\extten(\vec a;\vec d)\) of
\(\RA(\vec a;\vec d)\) is the \(r\)-dimensional \(0/1\) tensor, with each index running over
\(1,\dots,n\), defined by
\begin{equation}\label{eq:extten}
\begin{split}
  \extten(\vec a;\vec d)[j_1,\dots,j_r]=1 \iff{}
    &\shift^{d_1}(v_{j_1})+\cdots+\shift^{d_r}(v_{j_r})\\
    &\quad\text{is a vertex of }\RA(\vec a;d_1,\dots,d_r).
\end{split}
\end{equation}
For \(r=2\) and \(\vec d=(0,d)\) we call it the \textbf{extremal matrix} and write \(\extten(\vec a;d)\),
with \(i\) indexing rows (the unshifted summand) and \(j\) columns (the shifted one).
Both arguments are carried in the notation because both matter: the shift vector fixes which
coordinates the summands share, and the leg vector decides the answer inside that overlap ---
Section \ref{sec-vertices} is in large part the study of that dependence. Since \(\vec a\) has length
\(n\), writing it also records \(n\).
\label{def-extten}
\end{definition}

\textbf{Indices run from \(1\), on every axis and in every statement of this section}, matching the
labelling \(v_1,\dots,v_n\) of Definition \ref{def-rasimp}.

The indexing deserves a word, because it is a feature of this family rather than of Minkowski sums
in general. For an arbitrary sum \(P_1+\dots+P_r\) of polytopes there is no canonical way to label
the vertices of the summands, so the analogous \(0/1\) tensor is defined only up to independent
permutation of each axis, and to pin it down one has to supply an ordering of each \(P_i\)'s
vertices. Here no such choice is needed: \(\rasimp(\vec a)\) comes with a natural total order on its
vertices, \(v_1,\dots,v_n\) by increasing support, and every summand is a shifted copy of the same
simplex, so the same order serves for all of them. That is what makes \(\extten(\vec a;d)\) a well-defined tensor
with meaningful indices, and it is why statements such as ``the entry at \(j=t+1\) vanishes'' have
content.

The leg vector that matters is the Kozlov vector \(\kozvec(n)\) of Definition \ref{def-kozvec}, for
which \(\proj\bigl(\Koz(n;d_1,\dots,d_r)\bigr)\) is \(\RA(\vec a;d_1,\dots,d_r)\) up to the
translation of \eqref{eq:koz-as-ra}; it is the
case the rest of the paper is about. We also study the all-ones leg vector \(\onesvec(n)\), which has
a theorem of its own; and a good deal of what
follows holds for a \textbf{general} positive leg vector: of the six propositions of
Appendix \ref{sec-app-cases} only two ask anything of \(\vec a\) beyond positivity, and what they ask
is that the ratios \(a_{d+s}/a_s\) decrease.

\textbf{The proofs are in Appendix \ref{sec-app}.} Every statement in this section is a statement about the
extremal tensor, and every proof of one is a case analysis on the index pair, carried out by
exhibiting a single linear functional or by deriving a contradiction from one. The analysis is long
and its bookkeeping is of a quite different character from the statements it establishes, so it is
collected in Appendix \ref{sec-app} --- together with the auxiliary lemmas it needs and the notation
those lemmas use, none of which is needed to read what follows. What is left here is the
definitions, the statements, and the examples and comparisons that can be read without the case
analysis.
\subsection{The three regimes}
\label{sec:orgbe6fba5}

Fix \(r=2\) and \(\vec d=(0,d)\), so that \(\RA(\vec a;0,d)\subseteq\mathbb R^N\) with
\(N=n+d\), and let \(m=n-d\) be the number of coordinates the two summands share, so that
\(m\le0\) means they share none. Every statement in this subsection is a rank-two
one, and \(\vec a\) is an arbitrary positive leg vector of length \(n\) unless said otherwise.

The three regimes are separated because they need genuinely different hypotheses, and because only
the third depends on \(\vec a\) at all. We take them one at a time.

\begin{proposition}
\leanmark{\leannone\,\compmark}(\emph{Tied shifts.}) Let \(n\ge1\) and \(\vec a\) be a positive leg vector. Then
\(\RA(\vec a;0,0)=2\rasimp(\vec a)\), a dilation of the right-angle simplex, and
\(\extten(\vec a;0)\) is the \(n\times n\) identity matrix.
\label{prop:regimes-tie}
\label{prop-regimes-tie}
\end{proposition}

\begin{proposition}
\leanmark{\leannone\,\compmark}(\emph{Disjoint windows.}) Let \(n\ge1\), \(d\ge n\) and \(\vec a\) be a positive leg vector. Then the
two summands are supported on disjoint sets of coordinates,
\(\RA(\vec a;0,d)=\rasimp(\vec a)\times\shift^d(\rasimp(\vec a))\) is their Cartesian (prism)
product, and \(\extten(\vec a;d)\) is the all-ones matrix: all \(n^2\) vertex sums are vertices.
\label{prop:regimes-disjoint}
\label{prop-regimes-disjoint}
\end{proposition}

\begin{proposition}
\leanmark{\leannone\,\compmark}(\emph{One shared coordinate.}) Let \(n\ge2\), \(d=n-1\) and \(\vec a\) be a positive leg vector. Then
\(\extten(\vec a;d)\) is the all-ones matrix, \(\RA(\vec a;0,n-1)\) has \(n^2\) vertices, and it is
the image of \(\rasimp(\vec a)\times\rasimp(\vec a)\) under an affine isomorphism.
\label{prop:regimes-boundary}
\label{prop-regimes-boundary}
\end{proposition}

All three are proved in Appendix \ref{sec-app-regimes}. Each argument exhibits one functional; only
the third needs a word about which of the two summands is constant on the shared coordinate, and it
is the shifted one.

\begin{remark}
The remaining regime is \(1\le d\le n-2\), equivalently \(m\ge2\), and there the answer \emph{can} depend
on \(\vec a\): for \(m\ge3\), Theorem \ref{thm-ones} and Theorem \ref{thm-kozlov-matrix} give different
matrices. At \(m=2\) they still agree --- see the paragraph after Figure \ref{fig:summands-and-sum},
which locates the onset at \(m=3\) exactly.
This is an observation about the pair of theorems rather than a theorem itself, which is why it is
recorded here and not as a third part of a proposition.
\label{rem-regimes-dependence}
\end{remark}

Figure \ref{fig:summands-and-sum} is worth pausing on, because it shows less than one might hope. Rank
two with \(n=3\) and \(d=1\) is the \emph{only} case in which \(\RA(\vec a;0,d)\) is three-dimensional:
the dimension is
\(\lvert\lbrace2,\dots,n\rbrace\cup\lbrace d+2,\dots,d+n\rbrace\rvert\), and for it to equal \(3\)
one needs \(d+n=4\). That forces \(m=2\), and at \(m\le2\) the two leg vectors give the same face
lattice. So no picture in three dimensions can display the dependence on \(\vec a\); the two bodies in
its right-hand column differ metrically and not combinatorially. The dependence begins at \(m=3\)
(below), which needs \(\dim\RA(\vec a;0,d)\ge4\), so this is why the evidence below is numerical
rather than pictorial.

At \(d=n-2\) the \emph{argument} of Proposition \ref{prop-regimes-boundary} breaks: two coordinates are
shared rather than one, the second being \(a_2\) or \(0\) depending on which vertex of the summand
\(\rasimp(\vec a)\) is taken, so the shared
block starts discriminating on both sides at once. That is a statement about the proof and not about
the answer. At \(d=n-2\) one has \(m=2\), and the paragraph above records that the two leg vectors
still give the same face lattice there: \textbf{dependence on \(\vec a\) begins one step later, at
\(m=3\), that is at \(d=n-3\)}, where \(\binom{m-1}{2}\) first becomes positive. Once it begins
it is real, and not a subtler invariant. For \(n=6\) and \(d=1\) the two polytopes are both \(6\)-dimensional but have
different numbers of faces in \emph{every} dimension: \(16\) vertices, \(50\) edges, then \(75\), \(65\),
\(33\), and \(9\) facets for \(\vec a=\onesvec(n)\), against \(22\), \(76\), \(120\), \(104\), \(50\),
and \(12\) for the Kozlov vector. So the two are not the same polytope differently coordinatised;
they differ as combinatorial objects.

The numbers just listed are those of the \emph{polytope \(\RA(\vec a;0,d)\) itself}, its polytopal
face-vector in the sense of Section \ref{sec-notation}, and not \(f\)-vectors of complexes: an
unrelated invariant that happens to share the classical name.

The last entries are worth keeping in view: \(9\) facets against \(12=2n\). The Kozlov polytope
attains \(2n\) and the all-ones one falls short, which is exactly the redundancy phenomenon
Theorem \ref{conj-kozlov-hrep2} and Theorem \ref{conj-ones-hrep} describe in Section \ref{sec-hrep2}.

The smallest case in which the two differ at all is \(n=4\), \(d=1\), where \(m=3\) and the
discrepancy \(\binom{m-1}{2}\) first equals \(1\). There \(\RA(\vec a;0,d)\) is four-dimensional,
so it is seen
through a Schlegel diagram (Figure \ref{fig:schlegel}): projecting from a point just outside one
facet onto that facet's hyperplane, which draws the boundary complex faithfully in three dimensions.

\begin{figure}[htbp]
\centering
\begin{minipage}[t]{0.47\linewidth}\centering
  \includegraphics[width=\linewidth]{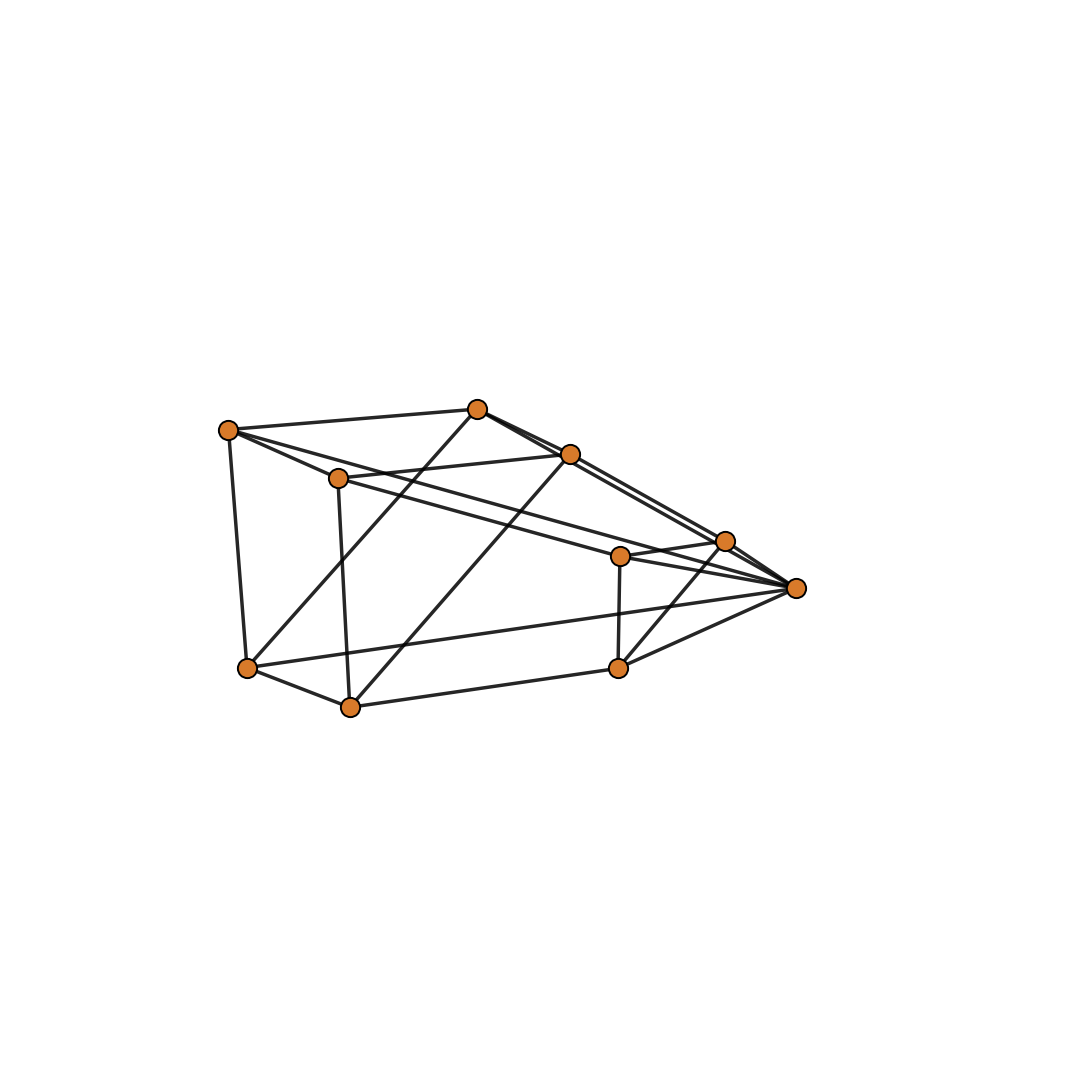}\\[-4pt]
  {\footnotesize $a=(1,1,1,1)$: $10$ vertices, $21$ edges}
\end{minipage}\hfill
\begin{minipage}[t]{0.47\linewidth}\centering
  \includegraphics[width=\linewidth]{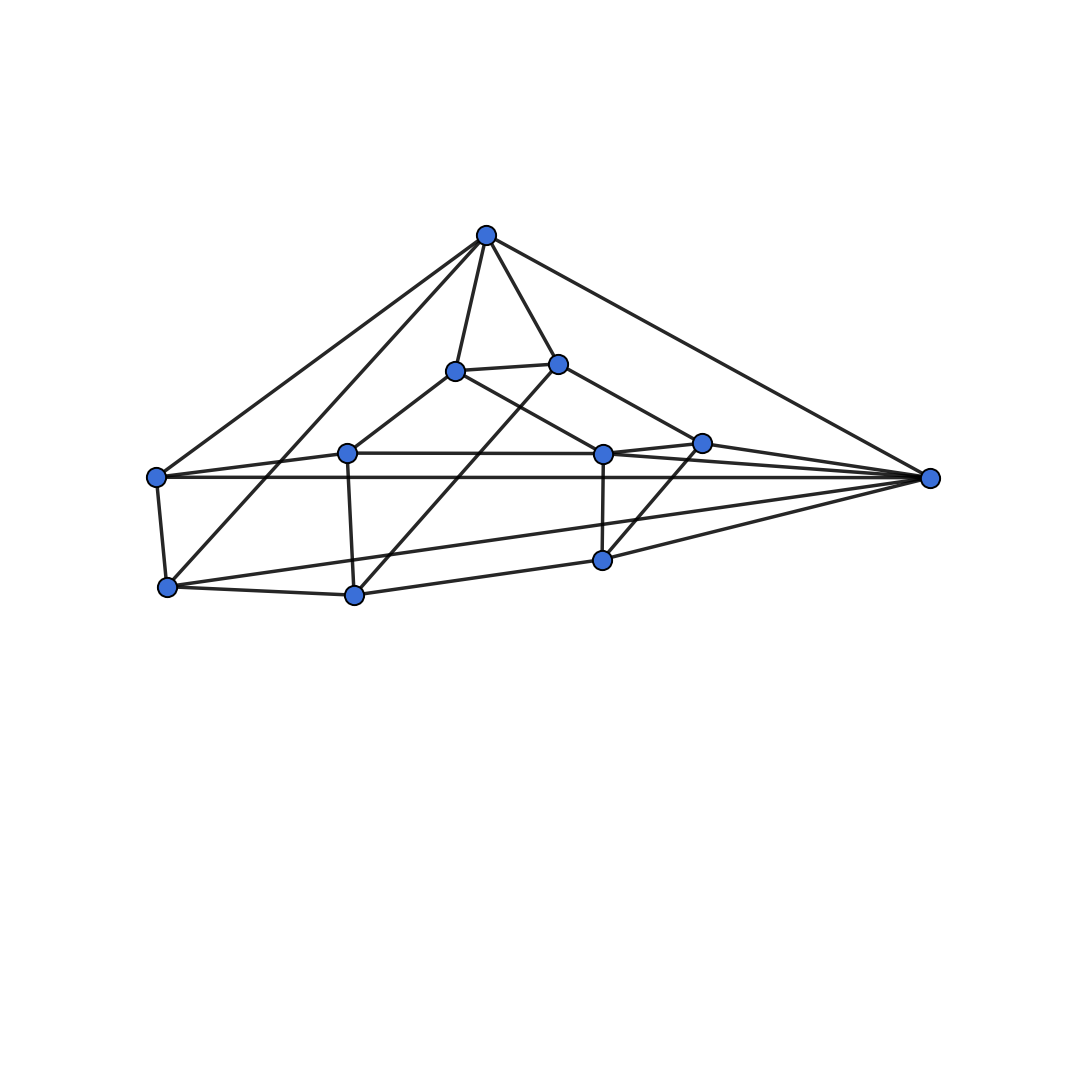}\\[-4pt]
  {\footnotesize $a=(1,3,3,1)$: $11$ vertices, $24$ edges}
\end{minipage}
\caption{Schlegel diagrams of $\RA(\vec a;0,1)=\rasimp(\vec a)+\shift^1(\rasimp(\vec a))$ for $n=4$, the smallest case where
the leg vector changes the combinatorics. The polytopal face-vectors are $(1,10,21,18,7,1)$ and
$(1,11,24,21,8,1)$, differing by $(1,3,3,1)$ in every dimension: one extra vertex, carrying three
extra edges, three extra $2$-faces and one extra facet. Cross-sections locate the change at the two
extreme coordinates --- a triangular prism acquiring a split vertex --- while the middle coordinates
give identical sections for both.}
\label{fig:schlegel}
\end{figure}

\clearpage
\subsection{The all-ones case at rank two}
\label{sec:org50870ac}

\begin{theorem}
\leanmark{\leanfull\,\compmark}\regionbox{\regdiag \regones{(0,0.7) rectangle (1,1)} \regones{(0.7,0) rectangle (1,1)} \regones{(0,0.7)--(0.09,0.7)--(0.7,0.09)--(0.7,0)--(0.61,0)--(0,0.61)--cycle} \node[font=\tiny] at (0.35,0.35) {$I_m$};}Let \(n\ge2\), \(\vec a=\onesvec(n)\), \(r=2\) and \(\vec d=(0,d)\) with
\(1\le d\le n-1\); put \(m=n-d\), and for indices \(1\le i,j\le n\) write \(t=i-d\). Then
\begin{equation}\label{eq:ones}
  \extten(\vec a;d)[i][j]=0 \iff 1\le t\le m,\quad 1\le j\le m,\quad t\ne j .
\end{equation}
Inside the \(m\times m\) overlap block only the diagonal survives, so that block is the identity
matrix \(I_m\); outside it every entry is one.
Consequently \(\RA(\vec a;0,d)\) has exactly \(n^2-m(m-1)\) vertices.
\label{thm:ones}
\label{thm-ones}
\end{theorem}

In this schematic and the ones like it below, the frame is the whole \(n\times n\) extremal
matrix, rows downward and columns rightward; the dashed rules cut off the top \(d\) rows and the
right \(d\) columns, leaving the \(m\times m\) overlap block at the lower left, and the drawn
diagonal is \(j=t\). \textbf{\textbf{Blue marks entries the statement shows are \(1\), red entries it shows are
\(0\).}} The proportions are schematic, not to scale for any particular \(n\) and \(d\).

The hypothesis \(1\le d\le n-1\) is what makes the statement non-degenerate, and it forces
\(n\ge2\); the two excluded values of \(d\) are Proposition \ref{prop-regimes-tie} (\(d=0\), the
identity matrix) and Proposition \ref{prop-regimes-disjoint} (\(d\ge n\), the all-ones matrix), and at
\(d=n-1\) the theorem's zero set is empty, agreeing with Proposition \ref{prop-regimes-boundary}.

The proof is in Appendix \ref{sec-app-ones}. It turns the criterion into a question about a small
digraph on the coordinates, whose acyclicity decides the entry; the example that works that digraph
out by hand is there too, since it is about the mechanism rather than about the matrix.

The shape is easiest to see in an example. For \(n=5\) the matrices \(\extten(\vec 1;d)\), with \(i\)
indexing rows and \(j\) columns, are
\[
d=1:\;\begin{pmatrix}1&1&1&1&1\\1&0&0&0&1\\0&1&0&0&1\\0&0&1&0&1\\0&0&0&1&1\end{pmatrix} \qquad d=2:\;\begin{pmatrix}1&1&1&1&1\\1&1&1&1&1\\1&0&0&1&1\\0&1&0&1&1\\0&0&1&1&1\end{pmatrix}
\]
\[
d=3:\;\begin{pmatrix}1&1&1&1&1\\1&1&1&1&1\\1&1&1&1&1\\1&0&1&1&1\\0&1&1&1&1\end{pmatrix} \qquad d=4:\;\begin{pmatrix}1&1&1&1&1\\1&1&1&1&1\\1&1&1&1&1\\1&1&1&1&1\\1&1&1&1&1\end{pmatrix}
\]
The \(m\times m\) block of zeros-off-the-diagonal sits in the lower left and shrinks as \(d\) grows,
vanishing at \(d=n-1\) where \(\extten(\vec a;d)\) is all ones and \(\RA(\vec a;0,d)\) is the image of
a Cartesian product under an affine isomorphism (Proposition \ref{prop-regimes-boundary}; it is a
Cartesian product on the nose only from \(d\ge n\)).

The orientation of Definition \ref{def-extten}, with \(i\) indexing the unshifted summand, is forced
and not conventional: transposing it makes the statement false. See Section \ref{sec-methods}.

Contrast Theorem \ref{thm-kozlov-matrix}. Both theorems restrict their zeros to the same block
\(1\le t\le m\), \(1\le j\le m\), and inside it both put zeros exactly where \(j<t\); they differ
only above the diagonal, where the all-ones matrix is zero throughout \(j>t\) while the Kozlov
matrix is zero only at \(j=t+1\). So the Kozlov matrix is the one displayed above with the region
\(j>t+1\) filled back in, a difference of \(\binom{m-1}{2}\) entries. The two statements are
directly comparable in this way only because both use the index base fixed in
Definition \ref{def-extten}.

The corresponding facet description was Conjecture \ref{conj:ones-hrep} of the first version of this
paper. It is now a theorem; the proof is in Part IV, not here.

\begin{theorem}
\leanmark{\leanpartial\,\compmark}\relax \autocite{SnellmanArticle4}
Let \(n\ge2\), \(1\le d\le n-1\) and \(N=n+d\). For \(\vec a=\onesvec(n)\), the
polytope \(\RA(\vec a;0,d)\)
is cut out in \(\mathbb R^N\) by a single equation together with one capped, weakly decreasing chain
of inequalities:
\[
\begin{array}{ll}
 x_1=1 & \text{(equation)}\\[2pt]
 1=x_1\ge x_2\ge\cdots\ge x_d & \text{(first block)}\\
 x_{d+1}\le x_d+1,\quad x_{d+1}\ge1 & \text{(entry to the overlap: the \(+1\) bump)}\\
 x_{d+1}\ge x_{d+2}\ge\cdots\ge x_n & \text{(overlap block)}\\
 x_n\ge x_{n+1},\quad x_{n+1}\le1 & \text{(transition)}\\
 x_{n+1}\ge x_{n+2}\ge\cdots\ge x_N\ge0 & \text{(second block)}
\end{array}
\]
\label{conj:ones-hrep}
\label{conj-ones-hrep}
\end{theorem}

The chain runs weakly downward from \(x_2\) to \(x_N\), floored at \(0\) and capped at \(1\) at both
ends, except that the first overlap coordinate is allowed one unit higher. The facet count is
\(n+d+2\) when \(m\ge2\), and \(n+d+1\) when \(m=1\).

The evidence available when the statement was a conjecture was exact polyhedral equality against
the Minkowski sum itself (Section \ref{sec-methods}). The proof in Part IV obtains it as a corollary
of the description for a general weakly log-concave leg vector.

\begin{mexample}
Take \(n=5\), \(d=2\), so \(N=7\) and \(m=3\). The sum \(\RA(\vec a;0,d)\) is \(6\)-dimensional
with \(19\)
vertices, and it has \(9\) facets --- one fewer than the \(2n=10\) that the Kozlov vector attains,
which is the redundancy at issue. The facets below were computed
by exact rational polyhedral arithmetic from the vertex list, independently of the general
description. Written out,
and with the equation \(x_1=1\) understood, they are
\[
\begin{array}{lll}
 x_2\le 1, & x_2-x_3\ge-1, & x_3\ge 1,\\
 x_3-x_4\ge 0, & x_4-x_5\ge 0, & x_5-x_6\ge 0,\\
 x_6-x_7\ge 0, & x_6\le 1, & x_7\ge 0 .
\end{array}
\]
The chain \(x_3\ge x_4\ge x_5\ge x_6\ge x_7\ge0\) runs unbroken through the overlap, which is
exactly what makes one of the two windows' chain inequalities redundant here.
\label{ex-ones-facets}
\end{mexample}
\subsection{The Kozlov-vector case at rank two}
\label{sec:orgc373e84}

\textbf{Standing hypotheses and notation for this subsection.} Recall that the leg vector is
\(\vec a=(a_1,\dots,a_n)\), of length \(n\). Throughout, \(n\ge2\), \(r=2\),
\(\vec d=(0,d)\) with \(1\le d\le n-1\), and \(m=n-d\ge1\). The leg vector \(\vec a\) is positive of
length \(n\), and is the Kozlov vector \(\kozvec(n)\) except where stated; the propositions of
Appendix \ref{sec-app-cases} record in each case how much of that is actually used. The
ambient space is \(\mathbb R^N\) with \(N=n+d\), and its coordinates split into three blocks that
partition \([1,N]\):
\begin{equation}\label{eq:ones-facets}
  \underbrace{[1,d]}_{\text{first summand only}}\quad
  \underbrace{[d+1,n]}_{\text{shared, }m\text{ coordinates}}\quad
  \underbrace{[n+1,N]}_{\text{second summand only}} .
\end{equation}
We write \(v_i\) for the \(i\)-th vertex of the first (unshifted) summand and
\(u_j=\shift^d(v_j)\) for the \(j\)-th vertex of the second, and \(t=i-d\), so that
\(1\le t\le m\) says exactly that \(v_i\) reaches into the shared block without leaving it.

\begin{theorem}
\leanmark{\leanfull\,\compmark}\regionbox{\regdiag \regones{(0,0.7) rectangle (1,1)} \regones{(0.7,0) rectangle (1,1)} \regones{(0,0.7)--(0.09,0.7)--(0.7,0.09)--(0.7,0)--(0.61,0)--(0,0.61)--cycle} \regones{(0.21,0.7)--(0.7,0.7)--(0.7,0.21)--cycle} \regzeros{(0,0)--(0.61,0)--(0,0.61)--cycle} \regzeros{(0.09,0.7)--(0.21,0.7)--(0.7,0.21)--(0.7,0.09)--cycle}}Let \(n\ge2\), let \(\vec a=\kozvec(n)=\bigl(\binom n1,\dots,\binom nn\bigr)\) be the Kozlov vector,
let \(r=2\) and \(\vec d=(0,d)\) with \(1\le d\le n-1\), and put \(m=n-d\). Then for indices
\(1\le i,j\le n\), writing \(t=i-d\),
\begin{equation}\label{eq:kozlov-matrix}
  \extten(\vec a;d)[i][j]=0 \iff 1\le t\le m,\quad 1\le j\le m,\quad\text{and}\quad
   (j<t\ \text{or}\ j=t+1),
\end{equation}
and \(\extten(\vec a;d)[i][j]=1\) otherwise. Consequently the vertices of
\(\RA(\vec a;0,d)=\rasimp(\vec a)+\shift^d(\rasimp(\vec a))\) are exactly the points
\(v_i+u_j\) with \(\extten(\vec a;d)[i][j]=1\); these are pairwise distinct, and
\begin{equation}\label{eq:vertex-count-kozlov}
  \#\operatorname{vert}\bigl(\RA(\vec a;0,d)\bigr)\;=\;n^2-\frac{(m-1)(m+2)}{2}.
\end{equation}
\label{thm:kozlov-matrix}
\label{thm-kozlov-matrix}
\end{theorem}

The proof is in Appendix \ref{sec-app-kozlov}. The index pairs are classified by six propositions,
stated and proved in Appendix \ref{sec-app-cases}: each covers one region of the matrix, and together
they leave no pair unaccounted for, and only two of them use anything about \(\vec a\) beyond
positivity, as noted at the head of this section.

The Kozlov matrices \(\extten(\kozvec(5);d)\), set beside the all-ones matrices
displayed after Theorem \ref{thm-ones}, are
\[
d=1:\;\begin{pmatrix}1&1&1&1&1\\1&0&1&1&1\\0&1&0&1&1\\0&0&1&0&1\\0&0&0&1&1\end{pmatrix} \qquad d=2:\;\begin{pmatrix}1&1&1&1&1\\1&1&1&1&1\\1&0&1&1&1\\0&1&0&1&1\\0&0&1&1&1\end{pmatrix}
\]
\[
d=3:\;\begin{pmatrix}1&1&1&1&1\\1&1&1&1&1\\1&1&1&1&1\\1&0&1&1&1\\0&1&1&1&1\end{pmatrix} \qquad d=4:\;\begin{pmatrix}1&1&1&1&1\\1&1&1&1&1\\1&1&1&1&1\\1&1&1&1&1\\1&1&1&1&1\end{pmatrix}
\]
Comparing with the all-ones matrices makes the difference visible entry by entry. Outside the block
\(1\le t\le m\), \(1\le j\le m\) both matrices are all ones, and inside it they agree except in the
region \(j>t+1\), which is zero for the all-ones vector and one here. At \(d=1\) they
differ in three entries, namely \((i,j)=(2,3)\), \((2,4)\) and \((3,4)\); at \(d=2\) in the single
entry \((3,3)\); and at \(d=3\) not at all. These counts are \(\binom{m-1}{2}=3\), \(1\) and \(0\),
as they must be.

The vertex count settles the vertex set of the rank-two Kozlov polytope outright. Ranks beyond the
second are not settled here; they are settled in Part IV \autocite{SnellmanArticle4}.

Two checks on \eqref{eq:vertex-count-kozlov}. It agrees with direct polyhedral computation
(Section \ref{sec-methods}); and at \(d=n-1\) it gives \(m=1\) and \(n^2-0=n^2\), the
one-shared-coordinate regime of Proposition \ref{prop-regimes-boundary}, as it must. Contrast the
all-ones count \(n^2-(n-d)(n-d-1)\) of Theorem \ref{thm-ones}: the two differ, which is the
\(\vec a\)-dependence of Section \ref{sec-vertices} in its sharpest form.
\subsection{The facet description at rank two}
\label{sec-hrep2}
Section \ref{sec-vertices} settles the vertices at rank two for both leg vectors. The facets are known
for the Kozlov simplex itself (Corollary \ref{cor-facets}) and for the all-ones sum
(Theorem \ref{conj-ones-hrep}); the remaining case is the Kozlov sum, and it has the same shape. It was
Conjecture \ref{conj:kozlov-hrep2} of the first version of this paper and is likewise proved in Part IV.

Two features make the statement short. The offsets cost nothing: the support function of a Minkowski
sum is the sum of the support functions, so the supporting value of \(\RA(\vec a;0,d)\) in a
direction \(u\) is
\(h(u,\rasimp(\vec a))+h(u,\shift^d(\rasimp(\vec a)))\), a sum of maxima over \(n\) known vertices. The
candidate directions are few: each summand contributes the \(n\) inequalities of Corollary
\ref{cor-facets}, to which exactly two more must be added.

\begin{theorem}
\leanmark{\leanpartial\,\compmark}\relax \autocite{SnellmanArticle4}
Let \(\vec a=\kozvec(n)\) be the Kozlov vector, \(r=2\), \(\vec d=(0,d)\) with \(1\le d\le n-1\),
\(N=n+d\) and \(m=n-d\), and put \(t=j-d\). Then \(\RA(\vec a;0,d)\) is cut out in \(\mathbb R^N\) by
the equation \(x_1=n\), the first chain with its offsets,
\begin{equation}\label{eq:kozlov-hrep2}
  \frac{x_j}{\binom nj}-\frac{x_{j+1}}{\binom n{j+1}}\ \ge\ -c_j,\qquad
  c_j=\begin{cases}
    0, & 1\le j\le d-1,\\[2pt]
    \dfrac{\binom n1}{\binom n{d+1}}, & j=d,\\[6pt]
    \dfrac{\binom n{t+1}}{\binom n{j+1}}-\dfrac{\binom nt}{\binom nj}, & d+1\le j\le n-1,
  \end{cases}
\end{equation}
the second chain, which needs no offset,
\begin{equation}\label{eq:kozlov-hrep2-second}
  \frac{x_{d+j}}{\binom nj}\ \ge\ \frac{x_{d+j+1}}{\binom n{j+1}}\qquad(1\le j\le n-1),
\end{equation}
and the three remaining inequalities
\begin{equation}\label{eq:kozlov-hrep2-extras}
  x_{d+1}\ \ge\ \tbinom n1=n,\qquad x_{n+1}\ \le\ \tbinom n{m+1},\qquad x_N\ \ge\ 0 .
\end{equation}
The three cases of \(c_j\) are one formula, the last, read with \(\binom n0\) replaced by \(0\) and
\(\binom ns=0\) for \(s<0\): the leg vector has no entry in position \(0\) or below. Exactly one of
these \(2n+1\) inequalities is redundant, the second chain at \(j=1\), so
\(\RA(\vec a;0,d)\) has
\(2n\) facets.
\label{conj:kozlov-hrep2}
\label{conj-kozlov-hrep2}
\end{theorem}

\begin{remark}
(\emph{Where the offsets come from, and why only one chain has them.}) Each of the \(2n+1\) inequalities is
\(\langle u,x\rangle\le h(u,\rasimp(\vec a))+h(u,\shift^d(\rasimp(\vec a)))\) for its own direction
\(u\), and in every case \emph{the summand the inequality belongs to contributes \(0\)}: a chain link of
Corollary \ref{cor-facets} is supported at value \(0\) on the simplex it cuts. So the whole offset is
the \emph{other} summand's supporting value, and the asymmetry between the two chains is an asymmetry
between the two summands' positions.

For the first chain, the shifted copy occupies coordinates \(d+1,\dots,N\), so it contributes nothing
while \(j+1\le d\), whence \(c_j=0\) for \(j\le d-1\); and from \(j=d\) on it contributes the
staircase entry that has entered the window, beginning with its leading \(\binom n1=n\). In terms of
the ratios \(\rho_j=\binom n{d+j}\big/\binom nj\) of Lemma \ref{lem-binomial-ratio}, the offset is
\(c_j=1/\rho_{t+1}-1/\rho_t\) with \(1/\rho_0\) read as \(0\), which is positive because \(\rho\) is
strictly decreasing.

For the second chain the unshifted copy is the intruder, and it contributes nothing at all: the
supporting value is \(\max\lbrace0,\ \binom n{s+1}/\binom n{j+1}-\binom ns/\binom nj\rbrace\) with
\(s=d+j>j\), and that maximum is \(0\) exactly because \(\rho_{j+1}\le\rho_j\), again
Lemma \ref{lem-binomial-ratio}(2). The two chains therefore differ not by accident of presentation but
because the same monotonicity that makes one offset positive makes the other vanish.
\label{rem-hrep2-offsets}
\end{remark}

\begin{mexample}
Continuing Example \ref{ex-ones-facets} with the same \(n=5\), \(d=2\) but \(\vec a\) the Kozlov vector
\((5,10,10,5,1)\): now \(\RA(\vec a;0,d)\) is again \(6\)-dimensional but has \(20\) vertices
and \(10=2n\)
facets, one more than the all-ones sum. With \(x_1=5\) understood they are
\[
\begin{array}{lll}
 x_2\le 10, & x_2-x_3\ge-5, & x_3\ge 5,\\
 x_3-2x_4\ge-15, & x_4-5x_5\ge-40, & x_4-x_5\ge 0,\\
 x_5-2x_6\ge 0, & x_6-5x_7\ge 0, & x_6\le 5,\\
 x_7\ge 0 . & &
\end{array}
\]
The two \(\mu\)-chains are visible in the coefficients \(x_3-2x_4\), \(x_5-2x_6\),
\(x_4-5x_5\), \(x_6-5x_7\) are the ratios \(\binom5j/\binom5{j+1}\) of Corollary \ref{cor-facets}, one
chain from each window --- together with the two extra inequalities \(x_3\ge5\) and \(x_2\le10\)
that belong to neither summand. The extra facet relative to the all-ones case is \(x_4-x_5\ge0\).
\label{ex-kozlov-facets}
\end{mexample}

The description was checked against the Minkowski sum itself, with the redundant member the same one
each time (Section \ref{sec-methods}). Validity needs no
argument: the supporting value of a Minkowski sum in a direction is the sum of the summands'
supporting values in that direction, which is the additivity recorded at the head of this
subsection. The substance is completeness: that every point satisfying these inequalities really
decomposes as a point of \(\rasimp(\vec a)\) plus a point of \(\shift^d(\rasimp(\vec a))\). That is
what Part IV proves \autocite{SnellmanArticle4}, for every positive, weakly log-concave leg vector.

One remark on the shape of the statement. The two extra inequalities are not inherited from either
summand and cannot be, since they are exactly where the naive expectation that a Minkowski sum's
facets come from its summands' facets breaks down --- and it breaks down already at \(n=3\). The
first version of this paper added that a decomposition cannot be produced by truncating \(x\)
against the second summand's maximum profile; that remark was mistaken and has been removed.
\section{Methods and verification}
\label{sec-methods}
Every theorem of this paper has a written proof, except Theorems \ref{conj-ones-hrep} and
\ref{conj-kozlov-hrep2}, which are proved in Part IV. Many statements were in addition checked by
exact computation, some of them before they were proved, and a substantial part of the paper is
proved in Lean. This section says what was computed and how, and what the marks in the margin
mean: a disc for the Lean status of a result, and a square, \compmark{}, for a result checked by
exact computation. Two systems were used for the computations: \textbf{SageMath} 10.9 for exact rational polyhedral
computation (vertex and facet enumeration, Minkowski sums, equality of polytopes) and \textbf{Macaulay2}
for the module-theoretic statements, there with Amata and Crupi's \texttt{ExteriorIdeals} and
\texttt{ExteriorModules} packages \autocite{AmataCrupi2018ExteriorIdeals,AmataCrupi2021,Macaulay2,SageMath}.

Throughout, \emph{checked} means checked along a route independent of the formula being tested: a vertex
description against a facet description, a Minkowski sum against direct enumeration of the
achievable set, a closed form against brute-force construction. A check that merely re-evaluates
the same expression is not recorded here. File names below refer to the directory
\texttt{verification-code/} of the repository named in Section \ref{sec-methods-anc}.

\begingroup\sloppy
\begin{itemize}
\item \textbf{The Kozlov simplex.} Corollary \ref{cor-facets}: the polytope built from the \(n\) vertices and the
polytope cut out by the equation and the \(n\) inequalities, checked equal as polyhedra for
\(n=1,\dots,9\), each with \(n\) vertices and \(n\) facets in dimension \(n-1\); the degenerate
\(n=1\) is a point (\texttt{kozlov\_facets\_verification.sage}). The same script checks the dimension
formula of Proposition \ref{prop-dimension} in \(40\) configurations, \(r=2,3,4\) and \(n=2,\dots,5\).
\item \textbf{The extremal matrices.} Theorem \ref{thm-ones}: every entry against direct polyhedral computation
of the vertex set, \(n=3,\dots,9\) and \(d=0,\dots,n\), \(2296\) entries; the transposed reading
disagrees on \(15\) of the pairs \((n,d)\), which is what fixes the orientation
(\texttt{allones\_rank2\_closed\_form\_proof\_check.sage}). Theorems \ref{thm-ones} and \ref{thm-kozlov-matrix}
as restated, for both leg vectors, \(n=2,\dots,7\) and \(1\le d\le n-1\), and the vertex count
\eqref{eq:vertex-count-kozlov} for \(n=3,\dots,8\), \(27\) cases
(\texttt{verify\_section8\_restated.sage}).
\item \textbf{The proofs of the six case propositions.} For Propositions \ref{prop-row-below}--\ref{prop-gapK} and
the three regime propositions, a different and stronger check: the functional that each proof
\emph{specifies} is built and tested to expose the intended pair, \(2604\) checks
(\texttt{verify\_section8\_proofs.py}). The check of the statements cannot see a proof that is wrong
while its statement is right; this one caught such a proof.
\item \textbf{The facet descriptions at rank two.} Theorem \ref{conj-ones-hrep}: exact polyhedral equality against
the Minkowski sum for every \(n=2,\dots,10\) and \(1\le d\le n-1\), \(45\) cases
(\texttt{rank2\_hrep\_direct\_candidate.sage}). Theorem \ref{conj-kozlov-hrep2}, with the closed-form offsets
\(c_j\): the same for \(n=2,\dots,9\), \(36\) cases, exactly one inequality redundant in each;
every \(c_j\) also against the support value measured over the vertex sums; the redundant
inequality is the second chain at \(j=1\), checked separately for \(n=3,\dots,7\)
(\texttt{kozlov\_hrep\_offsets\_closed\_form.sage}, \texttt{rank2\_hrep\_direct\_candidate.sage}). These checks were
made when the two statements were conjectures.
\item \textbf{Modules.} Proposition \ref{prop-splitting}, both parts: componentwise identity and agreement of the
least generator degree, in Macaulay2, at \(n=4\), \(r=2\), \(\vec d=(0,1)\), with non-monomial
components in the proper and in the non-proper case (\texttt{ids\_initial\_submodule\_check.m2}).
Theorem \ref{thm-main} and the scope of Definition \ref{def-ids}: random monomial submodules sampled
and their Hilbert functions located inside the predicted polytope, and, as a negative control,
deliberately non-monomial submodules landing outside it
(\texttt{rankr\_minkowski\_random\_sampling\_m2.py}, \texttt{rankr\_nonmonomial\_negative\_control\_m2.py}).
\end{itemize}
\endgroup
\subsection{Formal verification}
\label{sec-methods-lean}
Much of this paper is also proved in the proof assistant Lean 4, on top of its mathematical
library Mathlib: the combinatorial spine of Sections \ref{sec-rank1}--\ref{sec-main}, with the
permissive variant of Section \ref{sec-permissive}, and the vertex structure of Section
\ref{sec-vertices} and Appendix \ref{sec-app}. The printed proofs and the formal ones are
deliberately the same proofs: where the formalisation found a shorter route, the text was
rewritten to take it. A mark in the margin beside a result records how far the formal proof goes:
\begin{itemize}
\item \leanfull{} proved in Lean, with no \texttt{sorry}, from the three axioms that Mathlib uses throughout
(\texttt{propext}, \texttt{Classical.choice}, \texttt{Quot.sound}) and nothing else;
\item \leanpartial{} part of the statement proved as for \leanfull{}: typically the description by
inequalities without the clause that these are the facets, or the combinatorial form of a
statement without its module-theoretic form;
\item \leannone{} not formalised.
\end{itemize}
A second mark, \compmark{}, beside a result says that it was also checked by exact computation, as
described in the list above; it can stand beside any of the three. Definitions, remarks and
examples carry no mark.

The file \texttt{lean/CORRESPONDENCE-PART-I.md} of the repository pairs each marked result with the Lean
declarations that prove it, gives the file and line of each, says what a partial mark leaves
out, and names the scripts behind each square; \texttt{lean/README.md} there has the fuller account, with the definitions. Two general points
belong here. The development works throughout in the coordinates that keep the degree-\(0\) entry,
so the \(\proj\) forms of the statements, and the translation \eqref{eq:koz-as-ra}, are not
formalised. And Theorem \ref{thm-main} is formalised as the combinatorial identity: the dictionary
between Hilbert functions and \emph{ff}-vectors, the first equality of Proposition \ref{prop-sumset}, is
outside the formal part. Theorems \ref{conj-ones-hrep} and \ref{conj-kozlov-hrep2} are marked for the
formal proofs that belong to Part IV. The part of the development that belongs to this paper is
about \(3\,300\) lines of Lean in twelve files, built with Lean \texttt{v4.34.0-rc2}.
\subsection{Availability}
\label{sec-methods-anc}
The scripts named above, the recorded output of the runs behind the counts quoted, and the Lean
development are available in three places.

\begin{itemize}
\item \raggedright \textbf{The repository of the series}, at
\url{https://gitlab.liu.se/jansn19/exterior-convex}, with a mirror at
\url{https://github.com/JaTaSn/exterior-convex}. It holds the sources of the articles,
the Lean development (directory \texttt{lean}), the scripts and their recorded output (directory
\texttt{verification-code}), and a reference certificate for each article. It is kept up to date as the
series grows, so its Lean development now covers later parts as well.
\item \textbf{The Zenodo archive of the series} \autocite{SnellmanZenodoSeries}, DOI
\href{https://doi.org/10.5281/zenodo.24288404}{10.5281/zenodo.24288404}: a fixed, citable snapshot of
the Lean development and the verification code for Parts I--V. It predates
\texttt{lean/CORRESPONDENCE-PART-I.md}, which is in the repository only.
\item \textbf{The ancillary files of the first version of this paper}, arXiv:2609.23605v1, directory \texttt{anc}:
the scripts, run logs and Lean development exactly as they stood when that version was submitted.
This version carries no ancillary files of its own.
\end{itemize}

\appendix
\section{Proofs for the rank-two vertex structure}
\label{sec-app}
\label{sec:app}
This appendix carries the proofs of everything stated in Section \ref{sec-vertices}, together with
the auxiliary lemmas and the notation those proofs need and nothing else does. It is ordered to
match that section: the criterion first, then the three regimes, then the all-ones case, then the
notation and the case analysis for the Kozlov vector, and finally the two results those cases
assemble into. Nothing here is used anywhere else in the paper.
\subsection{The vertex criterion}
\label{sec-app-criterion}
Everything in Section \ref{sec-vertices} rests on the following two lemmas, which replace the word ``vertex'' by a
condition that can be checked by exhibiting one linear functional. They are stated separately
because they do different work: the first is about a single polytope presented by its generators,
the second about how such presentations add.

\begin{lemma}
\leanmark{\leanfull}Let \(S\subseteq\mathbb R^N\) be a finite non-empty set and \(P=\conv(S)\). For \(x\in\mathbb R^N\)
the following are equivalent:
\begin{enumerate}
\item \(x\) is a vertex of \(P\);
\item \(x\in S\) and \(x\notin\conv(S\setminus\lbrace x\rbrace)\);
\item there is a linear functional \(\varphi\) with \(\langle\varphi,x\rangle>\langle\varphi,y\rangle\)
for every \(y\in S\setminus\lbrace x\rbrace\).
\end{enumerate}
\label{lem:exposed-vertex}
\label{lem-exposed-vertex}
\end{lemma}

\begin{proof}
(3) \(\Rightarrow\) (1). Every \(y\in P\) is a convex combination \(\sum_p\lambda_ps_p\) of elements
of \(S\), so \(\langle\varphi,y\rangle\le\langle\varphi,x\rangle\) with equality only if
\(\lambda_p=0\) for every \(s_p\ne x\), that is only if \(y=x\). So \(\lbrace x\rbrace\) is the face
of \(P\) exposed by \(\varphi\), and an exposed point is a vertex.

(1) \(\Rightarrow\) (2). A vertex of \(\conv(S)\) lies in \(S\): writing it as a convex combination
of elements of \(S\), extremality forces all the weight onto one of them; and if
\(x\in\conv(S\setminus\lbrace x\rbrace)\) then \(x\) is a convex combination of points of \(P\) other
than \(x\), hence interior to a segment in \(P\), contradicting extremality.

(2) \(\Rightarrow\) (3). If \(S=\lbrace x\rbrace\) take \(\varphi=0\), the condition being vacuous.
Otherwise \(C=\conv(S\setminus\lbrace x\rbrace)\) is a non-empty compact convex set not containing
\(x\), so a hyperplane separates them strictly: there are \(\varphi\) and \(u\) with
\(\langle\varphi,y\rangle<u<\langle\varphi,x\rangle\) for all \(y\in C\), and
\(S\setminus\lbrace x\rbrace\subseteq C\).
\end{proof}

\begin{lemma}
\leanmark{\leanpartial}Let \(P_1,\dots,P_r\subseteq\mathbb R^N\) be polytopes, \(P_p=\conv(V_p)\) with each \(V_p\)
finite and non-empty, let \(x_p\in P_p\) for each \(p\), and put \(x=\sum_px_p\).
\begin{enumerate}
\item \(x\) is a vertex of \(\sum_pP_p\) if and only if there is a \textbf{single} linear functional
\(\varphi\) for which each \(x_p\) is the \emph{unique} maximizer of \(\varphi\) on \(P_p\).
\item \emph{(Uniqueness of the decomposition.)} If \(x\) is a vertex of \(\sum_pP_p\) and also
\(x=\sum_px_p'\) with \(x_p'\in P_p\), then \(x_p'=x_p\) for every \(p\).
\end{enumerate}
\label{lem:vertexdecomp}
\label{lem-vertexdecomp}
\end{lemma}

\begin{proof}
Note first that \(\sum_pP_p=\conv\bigl(\sum_pV_p\bigr)\) by Lemma \ref{lem-convex}(A), and that
\(\sum_pV_p\) is finite, so Lemma \ref{lem-exposed-vertex} applies to it. We prove (2) first, because
(1) uses it.

(2) Fix \(q\) and set
\begin{equation}\label{eq:vertexdecomp}
  y=x_q'+\sum_{p\ne q}x_p,\qquad z=x_q+\sum_{p\ne q}x_p' ,
\end{equation}
both of which lie in \(\sum_pP_p\). Their sum is \(\sum_p(x_p+x_p')=2x\), so
\(x=\tfrac12y+\tfrac12z\) exhibits \(x\) as the midpoint of a segment with endpoints in
\(\sum_pP_p\). As \(x\) is a vertex of \(\sum_pP_p\) it is not interior to a segment with
\emph{distinct} endpoints there, so \(y=z\); and \(y+z=2x\) then gives \(y=x\), that is
\(x_q'+\sum_{p\ne q}x_p=\sum_px_p\), whence \(x_q'=x_q\). Since \(q\) was arbitrary, the
decomposition is unique.

(1) (\(\Leftarrow\)) Suppose each \(x_p\) is the unique maximizer of \(\varphi\) on \(P_p\). Any point
of \(\sum_pP_p\) is \(\sum_py_p\) with \(y_p\in P_p\), and
\(\langle\varphi,\sum_py_p\rangle=\sum_p\langle\varphi,y_p\rangle\le\sum_p\langle\varphi,x_p\rangle\),
with equality only if \(y_p=x_p\) for every \(p\). So \(\varphi\) is uniquely maximized on the sum
at \(x\), which is therefore exposed, and hence a vertex of \(\sum_pP_p\).

(\(\Rightarrow\)) Let \(x\) be a vertex of \(\sum_pP_p\). By Lemma \ref{lem-exposed-vertex} there is a \(\varphi\) with
\(\langle\varphi,x\rangle>\langle\varphi,y\rangle\) for every \(y\in\sum_pV_p\) other than \(x\), and
moreover \(x\in\sum_pV_p\), say \(x=\sum_pw_p\) with \(w_p\in V_p\). By part (2) applied to the two
decompositions \((x_p)\) and \((w_p)\) we get \(x_p=w_p\), so \textbf{each \(x_p\) is itself a member of
\(V_p\)} --- which is what makes the following exchange available.

Fix \(q\) and suppose some \(v\in V_q\) with \(v\ne x_q\) had
\(\langle\varphi,v\rangle\ge\langle\varphi,x_q\rangle\). Exchanging it into the decomposition gives
\(y=v+\sum_{p\ne q}x_p\in\sum_pV_p\) with \(y\ne x\) and
\(\langle\varphi,y\rangle\ge\langle\varphi,x\rangle\), a contradiction. So \(x_q\) is the strict
maximizer of \(\varphi\) over the finite set \(V_q\), and hence over \(P_q=\conv(V_q)\), a maximum
over a hull being attained only at convex combinations of maximizers.
\end{proof}

\begin{remark}
Neither part uses support functions, and neither imports the fact that a vertex of a polytope is
exposed --- Lemma \ref{lem-exposed-vertex} supplies what is needed of it for hulls of finite sets, and
is proved here rather than cited. Both are usually how this lemma is proved --- the face of \(\sum P_p\) exposed by
\(\varphi\) is \(\sum_pF(\varphi,P_p)\), which follows from Ziegler~\cite[Proposition 7.12]{Ziegler1995},
the common-refinement description of the normal fan of a Minkowski sum --- and both
are avoidable here because every polytope in this paper comes to us as the convex hull of an
explicit finite set, which is exactly the hypothesis Lemma \ref{lem-exposed-vertex} wants. The
exchange argument in (1) and the midpoint argument in (2) are then elementary. We spell this out
because it removes the only citation Section \ref{sec-vertices} would otherwise need.
\label{rem-no-support-functions}
\end{remark}

Combining the two: each summand \(\shift^{d_p}(\rasimp(\vec a))\) is the convex hull of its \(n\)
explicit vertices, so Lemma \ref{lem-exposed-vertex} turns ``unique maximizer on the summand'' into a
comparison against those \(n\) points only, and Lemma \ref{lem-vertexdecomp}(1) assembles the summands.
Together they give: \(\extten(\vec a;\vec d)[j_1,\dots,j_r]=1\) if and only if there is a
functional \(\varphi\) with
\begin{equation}\label{eq:no-support-functions}
  \bigl\langle\varphi,\shift^{d_p}(v_{j_p})\bigr\rangle
  \;>\;\bigl\langle\varphi,\shift^{d_p}(v_{k})\bigr\rangle
  \qquad\text{for every }p\text{ and every }k\ne j_p .
\end{equation}
\textbf{This is a finite system of strict linear inequalities in \(\varphi\), and nothing else in
Section \ref{sec-vertices} or in this appendix uses convexity.} Every proof below either exhibits
such a \(\varphi\) or derives a
contradiction from one.
\subsection{Proofs for the three regimes}
\label{sec-app-regimes}
\begin{proof}[Proof of Proposition~\ref{prop:regimes-tie}]
Both summands are \(\rasimp(\vec a)\), and \(P+P=2P\) for a convex \(P\). For the extremal matrix,
a linear functional attains its maximum on \(\rasimp(\vec a)\) at a unique vertex \(v_i\) of
\(\rasimp(\vec a)\) or at
several; by Lemma \ref{lem-vertexdecomp}(1) the sum \(v_i+v_j\) is a vertex of \(\RA(\vec a;0,0)\)
exactly when some \(\varphi\)
has \(v_i\) as its unique maximizer \emph{and} \(v_j\) as its unique maximizer, which forces \(i=j\).
Conversely each \(2v_i\) is a vertex of \(\RA(\vec a;0,0)\), since a functional exposing \(v_i\) on
\(\rasimp(\vec a)\) exposes it on both summands at once.
\end{proof}

\begin{proof}[Proof of Proposition~\ref{prop:regimes-disjoint}]
The first summand is supported in coordinates \([1,n]\) and the second in \([d+1,d+n]\), which are
disjoint precisely when \(d\ge n\). Given \(i\) and \(j\), choose \(\varphi'\) exposing \(v_i\) on
\(\rasimp(\vec a)\) and \(\varphi''\) exposing \(v_j\) on \(\shift^d(\rasimp(\vec a))\), and let
\(\varphi\) agree with \(\varphi'\) on \([1,n]\) and with \(\varphi''\) on \([d+1,d+n]\); the two
prescriptions do not conflict. Then \(\varphi\) uniquely maximizes on each summand, and
Lemma \ref{lem-vertexdecomp}(1) applies.
\end{proof}

\begin{proof}[Proof of Proposition~\ref{prop:regimes-boundary}]
Exactly one coordinate is shared, namely \(c=n\). Every vertex \(v_k\) of \(\rasimp(\vec a)\) has
first coordinate \(a_1\), so \(\rasimp(\vec a)\subseteq\lbrace x_1=a_1\rbrace\); applying
\(\shift^{n-1}\), which carries coordinate \(1\) to coordinate \(n\), gives
\(\shift^{n-1}(\rasimp(\vec a))\subseteq\lbrace x_n=a_1\rbrace\). \textbf{It is the shifted summand, not
the unshifted one, that is constant on the shared coordinate.} A contribution that is the same at
every vertex of a summand cannot affect which of them maximizes a functional, so \(\varphi_n\) may be
prescribed freely by the first summand and the disjoint-support argument of
Proposition \ref{prop-regimes-disjoint} goes through verbatim. Subtracting the constant \(a_1e_n\) from
the second factor identifies the sum with the Cartesian product.
\end{proof}
\subsection{The all-ones extremal matrix}
\label{sec-app-ones}
\begin{proof}[Proof of Theorem~\ref{thm:ones}]
There are exactly two summands, and everything below is written out for each of them separately
rather than carried through an index.

One piece of notation. A \emph{functional} \(\varphi\) is an element of the dual of \(\mathbb R^N\);
\(\varphi_c\) denotes its \(c\)-th coordinate, so \(\langle\varphi,x\rangle=\sum_c\varphi_cx_c\)
is the ordinary pairing. Write \(\Phi(c)=\varphi_1+\dots+\varphi_c\) for its partial sums, and
\begin{equation}\label{eq:no-support-functions-2}
  W_1=[1,n],\qquad W_2=[d+1,d+n],\qquad m_1=i,\qquad m_2=d+j ,
\end{equation}
so \(W_1\) and \(m_1\) belong to the unshifted summand and \(W_2\), \(m_2\) to the shifted one.

For \(\vec a=\onesvec(n)\) and any \(\varphi\), the two summands give
\begin{equation}\label{eq:no-support-functions-3}
  \bigl\langle\varphi,v_k\bigr\rangle=\sum_{c=1}^{k}\varphi_c=\Phi(k)-\Phi(0),
  \qquad
  \bigl\langle\varphi,\shift^{d}(v_k)\bigr\rangle=\sum_{c=d+1}^{d+k}\varphi_c=\Phi(d+k)-\Phi(d),
\end{equation}
with \(\Phi(0)=0\); in each case the offset (\(\Phi(0)\) for the first summand, \(\Phi(d)\) for
the second) does not depend on \(k\). Since \(\varphi\mapsto\Phi\) is a bijection onto real
sequences, Lemma \ref{lem-vertexdecomp}(1) says \(\extten(\vec a;d)[i][j]=1\) exactly when some
sequence \(\Phi\) satisfies both
\begin{equation}\label{eq:no-support-functions-4}
  \Phi(m_1)>\Phi(x)\quad\text{for every }x\in W_1\setminus\lbrace m_1\rbrace ,
\end{equation}
\begin{equation}\label{eq:no-support-functions-5}
  \Phi(m_2)>\Phi(x)\quad\text{for every }x\in W_2\setminus\lbrace m_2\rbrace .
\end{equation}

We call the elements of a digraph its \textbf{nodes} throughout, keeping the word \emph{vertex} for polytopes
alone. A system of strict inequalities \(X(u)>X(v)\), one per edge of a finite digraph, is solvable if
and only if that digraph is acyclic: a cycle gives \(X(u)>\dots>X(u)\), and conversely a topological
order provides a solution. So \(\extten(\vec a;d)[i][j]=1\) if and only if the digraph \(D\) is acyclic, where
\(D\) carries an edge \(m_1\to x\) for every \(x\in W_1\setminus\lbrace m_1\rbrace\) and an edge
\(m_2\to x\) for every \(x\in W_2\setminus\lbrace m_2\rbrace\).

That makes \(D\) very restricted: only \(m_1\) and \(m_2\) have outgoing
edges. Every directed cycle contains a simple one, each of whose nodes needs an outgoing edge, so
a simple cycle can visit only \(m_1\) and \(m_2\); and \(D\) has no loops. Hence \(D\) has a cycle
precisely when
\begin{equation}\label{eq:no-support-functions-6}
  m_1\ne m_2,\qquad m_2\in W_1,\qquad m_1\in W_2 .
\end{equation}
Unwinding, with \(1\le i\le n\) and \(1\le j\le n\): the second condition is \(d+j\le n\), that is
\(j\le m\); the third is \(i\ge d+1\), that is \(t\ge1\), the upper bound \(i\le d+n\) being
automatic; and the first is \(i\ne d+j\), that is \(t\ne j\). Since \(t\le m\) and \(j\ge1\) always
hold, the set of pairs with a cycle is exactly \(\lbrace 1\le t\le m,\ 1\le j\le m,\ t\ne j\rbrace\).
The vertex count follows, the zero set being an \(m\times m\) block minus its diagonal, distinct
surviving pairs giving distinct vertices of \(\RA(\vec a;0,d)\) by Lemma \ref{lem-vertexdecomp}(2).
\end{proof}

The proof's own machinery reproduces the two excluded values, which is a check on the set-up rather
than a separate argument: \(d=0\) gives \(W_1=W_2\), so a cycle exists exactly when \(m_1\ne m_2\),
and the matrix is the identity; \(d\ge n\) gives disjoint windows, no cycle is possible, and the
matrix is all ones. These agree with Propositions \ref{prop-regimes-tie} and
\ref{prop-regimes-disjoint}, which prove them without the all-ones hypothesis.

\begin{remark}
At rank two the digraph is a convenience rather than a necessity. Since only \(m_1\) and \(m_2\)
carry outgoing edges, the whole of it can be replaced by a direct construction: if \(m_1=m_2\) take
\(\Phi\) to be the indicator of that common point; if \(m_1\ne m_2\) and \(m_2\notin W_1\) take
\(\Phi(m_2)=2\), \(\Phi(m_1)=1\) and \(\Phi=0\) elsewhere, and symmetrically if \(m_1\notin W_2\);
while in the remaining case the two inequalities \(\Phi(m_1)>\Phi(m_2)\) and \(\Phi(m_2)>\Phi(m_1)\)
are directly contradictory. That version of the argument uses nothing about the windows at all ---
not that they are intervals, not that they have the same length --- which the digraph formulation
obscures.

We keep the digraph because it is the formulation that survives to higher rank: at rank \(r\) the
cycles can be long, acyclicity is the right criterion, and the rank-two collapse above is exactly
the statement that every cycle is a \(2\)-cycle. The reader who only wants Theorem \ref{thm-ones}
may take the three cases instead.
\label{rem:digraph-collapses}
\label{rem-digraph-collapses}
\end{remark}

\begin{mexample}
Take \(n=4\), \(d=1\), so \(N=5\), \(W_1=\lbrace1,2,3,4\rbrace\), \(W_2=\lbrace2,3,4,5\rbrace\) and
\(m=3\). Fix \(j=1\) and vary \(i\).

For \(i=2\) we get \(t=1=j\), so \(m_1=2\) and \(m_2=d+j=2\) coincide. Only the node \(2\) has
outgoing edges, and it has no incoming edge, so \(D\) is acyclic and the vertex sum \(v_i+u_j\) is a
vertex of \(\RA(\vec a;0,d)\).
For \(i=3\) we get \(t=2\ne j\), so \(m_1=3\) and \(m_2=2\) are distinct, each lies in the other's
window, and the two edges \(3\to2\) and \(2\to3\) close a cycle: the vertex sum is not a vertex of
\(\RA(\vec a;0,d)\).

The two digraphs are shown in Figure \ref{fig:digraph}.
\label{ex-digraph}
\end{mexample}

\begin{figure}[H]
\centering
\begin{tikzpicture}[shorten >=1pt,node distance=13mm,
  every node/.style={circle,draw,minimum size=6mm,inner sep=0pt,font=\small}]
\begin{scope}
  \node (a1) at (0,0) {1}; \node (a2) at (1.3,0) {2}; \node (a3) at (2.6,0) {3};
  \node (a4) at (3.9,0) {4}; \node (a5) at (5.2,0) {5};
  \draw[->] (a2) to[bend right=25] (a1);
  \draw[->] (a2) to[bend left=25] (a3);
  \draw[->] (a2) to[bend left=32] (a4);
  \draw[->] (a2) to[bend left=38] (a5);
  \node[draw=none,shape=rectangle,inner sep=1pt,font=\footnotesize] at (2.6,-1.15) {$i=2$, $t=j=1$: $m_1=m_2=2$, acyclic \textit{(the sum is a vertex)}};
\end{scope}
\begin{scope}[yshift=-30mm]
  \node (b1) at (0,0) {1}; \node (b2) at (1.3,0) {2}; \node (b3) at (2.6,0) {3};
  \node (b4) at (3.9,0) {4}; \node (b5) at (5.2,0) {5};
  \draw[->] (b3) to[bend right=25] (b1);
  \draw[->,very thick] (b3) to[bend right=25] (b2);
  \draw[->,very thick] (b2) to[bend right=25] (b3);
  \draw[->] (b3) to[bend left=25] (b4);
  \draw[->] (b2) to[bend left=32] (b4);
  \draw[->] (b2) to[bend left=38] (b5);
  \node[draw=none,shape=rectangle,inner sep=1pt,font=\footnotesize] at (2.6,-1.15) {$i=3$, $t=2\ne j$: $m_1=3$, $m_2=2$, the $2$-cycle in bold \textit{(the sum is not a vertex)}};
\end{scope}
\end{tikzpicture}
\caption{Digraphs associated to inequality chains, their elements called \emph{nodes} to keep \emph{vertex} for polytopes. $D$ is acyclic exactly when the corresponding vertex sum is a vertex of $\RA(\vec a;0,1)$.}
\label{fig:digraph}
\end{figure}
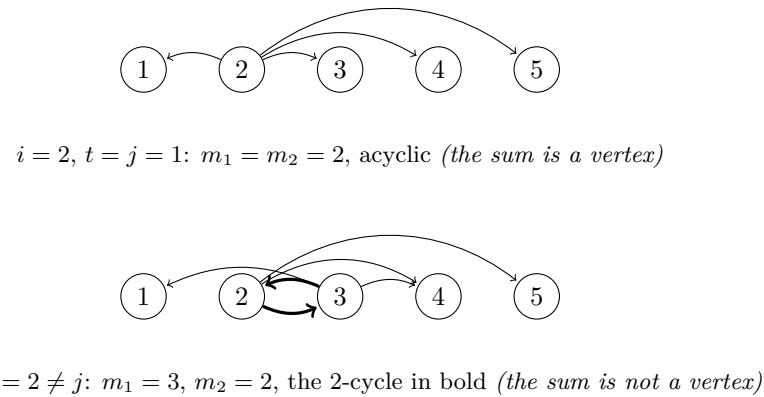
\subsection{Notation for the Kozlov-vector case}
\label{sec-app-koznot}
The hypotheses and the three coordinate blocks are those fixed in Section \ref{sec-vertices}; what
follows is the additional notation used only in the proofs below.

For a functional \(\varphi\) we abbreviate its restriction to the shared block by
\(e_s=\varphi_{d+s}\) for \(1\le s\le m\), and set
\begin{equation}\label{eq:PQ}
  Q_{t'}=\sum_{s\le t'}e_sa_{d+s}\quad(0\le t'\le m),\qquad
  P_{j'}=\sum_{s\le j'}e_sa_{s}\quad(0\le j'\le m).
\end{equation}
\(Q\) is the shared-block contribution to the \textbf{first} summand, whose coordinate \(d+s\) carries the
value \(a_{d+s}\); \(P\) is the shared-block contribution to the \textbf{second}, whose coordinate \(d+s\)
carries \(a_s\). Finally
\begin{equation}\label{eq:digraph}
  b_s=a_s,\qquad c_s=a_{d+s},\qquad \rho_s=c_s/b_s=a_{d+s}/a_s\qquad(1\le s\le m),
\end{equation}
so that \(\rho_s\) is the ratio by which the two summands weight the same shared coordinate. All of
\(P\), \(Q\), \(b\), \(c\) and \(\rho\) are used in the proofs below and are defined here once, so
that no proof depends on notation introduced after it.

\begin{lemma}
\leanmark{\leanfull}Let \(n\ge1\), \(d\ge1\), \(m=n-d\), and let \(\vec a=\kozvec(n)\) be the Kozlov vector, so that
\(\rho_j=\binom n{d+j}\big/\binom nj\) for \(1\le j\le n\), with the convention
\(\binom ns=0\) for \(s>n\). Then
\begin{enumerate}
\item \(\rho_j>0\) if and only if \(j\le m\), and \(\rho_j=0\) for \(m<j\le n\);
\item \(\rho_j>\rho_{j+1}\) for every \(1\le j\le m\).
\end{enumerate}
In particular \(\rho_1>\rho_2>\cdots>\rho_m>\rho_{m+1}=0\), so \(\rho\) is strictly decreasing
exactly on \(1,\dots,m+1\) and constant thereafter.
\label{lem:binomial-ratio}
\label{lem-binomial-ratio}
\end{lemma}

\begin{proof}
(1) is immediate: \(\binom n{d+j}>0\) exactly when \(d+j\le n\), that is \(j\le m\), and
\(\binom nj>0\) throughout \(1\le j\le n\) since \(j\le n\).

(2) Fix \(1\le j\le m\), so \(d+j\le n\) and hence \(\rho_j>0\); also \(j+1\le m+1\le n\) because
\(d\ge1\), so \(\rho_{j+1}\) is defined. Using \(\binom n{s+1}=\binom ns(n-s)/(s+1)\) twice,
\(\rho_{j+1}<\rho_j\) is equivalent to \((n-d-j)(j+1)<(d+j+1)(n-j)\), and
\begin{equation}\label{eq:binomial-ratio}
  (d+j+1)(n-j)-(n-d-j)(j+1)\;=\;d(n+1)\;>\;0 .
\end{equation}
Both cross-multiplications are by the positive quantities \(\binom nj\) and \(\binom n{j+1}\).
\end{proof}

The restriction \(j\le m\) in part (2) is not cosmetic: for \(j>m\) both \(\rho_j\) and
\(\rho_{j+1}\) vanish and the decrease is no longer strict. Every use below stays in the range
\(1\le j\le m\).
\subsection{The six cases}
\label{sec-app-cases}
\textbf{The short cases.} Five of the six regions are settled by short arguments, given here as separate
propositions so that each carries its own hypotheses and each says what it needs of \(\vec a\).

Two remarks before the first of them, both about which coordinates each summand can see. The
vertices \(v_{i'}\) of the first summand are supported in \([1,n]\) and never meet \([n+1,N]\); the
vertices \(u_{j'}\) of the second are supported in \([d+1,N]\) and never meet \([1,d]\). So
\(\varphi\) restricted to \([1,d]\) is invisible to the second summand and \(\varphi\) restricted to
\([n+1,N]\) is invisible to the first, and it is only on the shared block that a choice has to serve
both. Second, it is convenient to record the value of a vertex directly, for \(1\le i',j'\le n\):
\begin{align}
  \langle\varphi,v_{i'}\rangle &= \sum_{k\le i'}\varphi_ka_k, \label{eq:vertval-first}\\
  \langle\varphi,u_{j'}\rangle &= \sum_{s\le j'}\varphi_{d+s}a_s, \label{eq:vertval-second}
\end{align}
\eqref{eq:vertval-second} because \(u_{j'}\) carries the value \(a_s\) at coordinate \(d+s\). For
\(i'=d+t'\), \eqref{eq:vertval-first} reads \(\sum_{k\le d}\varphi_ka_k+Q_{t'}\), and for
\(j'\le m\), \eqref{eq:vertval-second} is \(P_{j'}\).

\begin{lemma}
\leanmark{\leannone}(\emph{Free shaping of one summand.}) Let \(w_1,\dots,w_n\) be positive reals, let \(1\le p\le n\), and
let \(\psi_1\in\mathbb R\) be prescribed arbitrarily. Put \(\psi_s=+1\) for \(2\le s\le p\) and
\(\psi_s=-1\) for \(p<s\le n\), and \(\Sigma_{p'}=\sum_{s\le p'}\psi_sw_s\). Then \(\Sigma_p>\Sigma_{p'}\) for
every \(p'\in[1,n]\) with \(p'\ne p\).
\label{lem-shape}
\end{lemma}

\begin{proof}
The term \(\psi_1w_1\) occurs in every \(\Sigma_{p'}\) with \(p'\ge1\) and so cannot affect which of
them is largest; that is the only role of the hypothesis that \(\psi_1\) is prescribed. Above index
\(1\) the increments \(\psi_sw_s\) are \(+w_s>0\) for \(s\le p\) and \(-w_s<0\) for \(s>p\), so
\(\Sigma_{p'}\) is strictly increasing on \([1,p]\) and strictly decreasing on \([p,n]\).
\end{proof}

The entries outside the overlap block split into two propositions rather than one, because the two
cases are proved differently: for \(i\le d\) it is the \textbf{first} summand that must be pinned using a
coordinate the second summand also sees, and for \(j>m\) it is the second summand that has private
room to spare.

\begin{proposition}[]
\leanmark{\leanfull\,\compmark}\regionbox{\regones{(0,0.7) rectangle (1,1)}}(\emph{Rows that stop inside the first private block.}) Let \(\vec a\) be any positive leg vector and let
\(1\le i\le d\), so that \(t\le0\). Then \(\extten(\vec a;d)[i][j]=1\) for every \(1\le j\le n\).
\label{prop:row-below}
\label{prop-row-below}
\end{proposition}

\begin{proof}
Let \(M>0\) and define \(\varphi\) by
\begin{equation}\label{eq:row-below}
  \varphi_k=+1\ (k\le i),\qquad \varphi_k=-1\ (i<k\le d),\qquad \varphi_{d+1}=-M,
\end{equation}
with \(\varphi_{d+2},\dots,\varphi_{d+n}\) given by Lemma \ref{lem-shape} applied with \(w_s=a_s\),
\(p=j\) and the prescribed value \(\psi_1=-M\). By that lemma \(u_j\) is the unique maximizer on the
second summand, \textbf{whatever \(M\) is} --- this is the point of prescribing \(\psi_1\) rather than
choosing it.

For the first summand, \(\langle\varphi,v_{i'}\rangle=\sum_{k\le i'}\varphi_ka_k\) increases strictly
up to \(i'=i\) and decreases strictly on \(i\le i'\le d\), so \(v_i\) beats every \(v_{i'}\) with
\(i'\le d\), and \(\langle\varphi,v_d\rangle\le\langle\varphi,v_i\rangle\). For \(i'>d\),
\(\langle\varphi,v_{i'}\rangle=\langle\varphi,v_d\rangle+Q_{i'-d}\) with
\begin{equation}\label{eq:row-below-2}
  Q_{t'}\;=\;-Ma_{d+1}+\sum_{s=2}^{t'}\varphi_{d+s}a_{d+s}\qquad(t'\ge1),
\end{equation}
whose second term does not involve \(M\) and takes finitely many values. Choosing \(M\) larger than
\(\max_{1\le t'\le m}\bigl(\sum_{s=2}^{t'}\varphi_{d+s}a_{d+s}\bigr)\big/a_{d+1}\) makes \(Q_{t'}<0\)
for every \(t'\ge1\), whence
\(\langle\varphi,v_{i'}\rangle<\langle\varphi,v_d\rangle\le\langle\varphi,v_i\rangle\). Apply
Lemma \ref{lem-vertexdecomp}(1).
\end{proof}

\begin{remark}
The case \(i=d\) is the one that forces this argument. For \(i<d\) the coordinates \(i+1,\dots,d\)
give the first summand a private penalty which can be made as large as one likes, and any choice at
all on the shared block is then harmless; at \(i=d\) that penalty is an empty sum, and the first
summand has to be pinned using a coordinate the second summand also sees. The device is to spend the
\textbf{first} shared coordinate on it: \(-Ma_{d+1}\) appears in every \(Q_{t'}\) with \(t'\ge1\) and so
pushes all of them below zero, while \(-Ma_1\) appears in every \(\langle\varphi,u_{j'}\rangle\) with
\(j'\ge1\) and so disturbs nothing on the second summand.
\label{rem-row-below}
\end{remark}

\begin{proposition}[]
\leanmark{\leanfull\,\compmark}\regionbox{\regones{(0.7,0) rectangle (1,1)}}(\emph{Columns that reach the second private block.}) Let \(\vec a\) be any positive leg vector and let
\(m<j\le n\). Then \(\extten(\vec a;d)[i][j]=1\) for every \(1\le i\le n\).
\label{prop:col-beyond}
\label{prop-col-beyond}
\end{proposition}

\begin{proof}
Here the second summand has room of its own: \(u_j\) meets \([n+1,N]\) precisely because \(j>m\).
Choose \(\varphi\) on \([1,n]\) by Lemma \ref{lem-shape} with \(w_k=a_k\), \(p=i\) and \(\psi_1=+1\),
which makes \(v_i\) the unique maximizer on the first summand; note this uses only coordinates the
first summand sees, and fixes \(\varphi\) on the shared block as a side effect. On the remaining
coordinates put \(\varphi_{d+s}=+\Gamma\) for \(m<s\le j\) and \(\varphi_{d+s}=-\Gamma\) for
\(j<s\le n\), with \(\Gamma>0\); these are the coordinates \(n+1,\dots,N\), invisible to the first
summand, so the first summand is unaffected. Then
\begin{equation}\label{eq:col-beyond}
  \langle\varphi,u_{j'}\rangle=P_{\min(j',m)}+\Gamma\Bigl(\sum_{s=m+1}^{\min(j',j)}a_s
    -\sum_{s=j+1}^{j'}a_s\Bigr),
\end{equation}
where the last sum is empty unless \(j'>j\). The quantities \(P_{j'}\) do not involve \(\Gamma\),
while the coefficient of \(\Gamma\) is strictly largest at \(j'=j\): it is
\(\sum_{s=m+1}^{j}a_s>0\) there, strictly smaller for \(j'<j\) and strictly smaller again for
\(j'>j\), by positivity of \(\vec a\). So for \(\Gamma\) large enough \(u_j\) is the unique
maximizer, and Lemma \ref{lem-vertexdecomp}(1) applies.
\end{proof}

\begin{proposition}[]
\leanmark{\leanfull\,\compmark}\regionbox{\regdiag \regones{(0,0.7)--(0.09,0.7)--(0.7,0.09)--(0.7,0)--(0.61,0)--(0,0.61)--cycle}}(\emph{The diagonal.}) Let \(\vec a\) be any positive leg vector and let \(1\le t\le m\). Then
\(\extten(\vec a;d)[d+t][t]=1\); that is, \(\extten(\vec a;d)[i][j]=1\) for \(i=d+t\) and
\(j=t\).
\label{prop:diagonal}
\label{prop-diagonal}
\end{proposition}

\begin{proof}
Take the single \textbf{threshold functional} with cut at \(p=d+t\),
\begin{equation}\label{eq:diagonal}
  \varphi_c=+1\ (c\le d+t),\qquad \varphi_c=-1\ (c>d+t)\qquad(1\le c\le N),
\end{equation}
which is \(+1\) on \([1,d]\), \(+1\) on the shared block up to \(s=t\), \(-1\) after that, and
\(-1\) on \([n+1,N]\), the three blocks of the standing notation written as one formula.

One computation now serves both summands. The value of \(\varphi\) at consecutive vertices of the
summand shifted by \(d'\) differs by \(a_{k+1}\varphi_{d'+k+1}\), a positive multiple of a single
coefficient; so that value strictly increases while \(d'+k+1\le d+t\) and strictly decreases after,
and is therefore uniquely maximized at the vertex index \(q\) with \(d'+q=d+t\). Applying this
with \(d'=0\), \(q=d+t\) and with \(d'=d\), \(q=t\), the same threshold \(p=d+t\) both times,
exposes \(v_{d+t}\) and \(u_t\) simultaneously. Apply Lemma \ref{lem-vertexdecomp}(1).
\end{proof}

\begin{proposition}[]
\leanmark{\leanfull\,\compmark}\regionbox{\regdiag \regzeros{(0.09,0.7)--(0.21,0.7)--(0.7,0.21)--(0.7,0.09)--cycle} \regzeros{(0,0.61)--(0,0.49)--(0.49,0)--(0.61,0)--cycle}}(\emph{The two adjacent diagonals.}) Let \(\vec a\) be any positive leg vector, and let \(1\le t\le m\)
and \(1\le j\le m\) with \(j=t+1\) or \(j=t-1\). Then \(\extten(\vec a;d)[i][j]=0\) for
\(i=d+t\).
\label{prop:adjacent}
\label{prop-adjacent}
\end{proposition}

\begin{proof}
The point is that \textbf{both summands cross the same coordinate}, and need its coefficient to have
opposite signs; no partial sums are involved.

Suppose \(j=t+1\le m\) and the entry were one, witnessed by \(\varphi\). Passing from \(u_t\) to
\(u_{t+1}\) on the second summand adds the single term \(a_{t+1}\varphi_{d+t+1}\), and exposing
\(u_{t+1}\) requires that to be positive, so \(\varphi_{d+t+1}>0\). Passing from \(v_{d+t}\) to
\(v_{d+t+1}\) on the first summand adds \(a_{d+t+1}\varphi_{d+t+1}\), \textbf{the same coefficient},
because coordinate \(d+t+1\) is where both summands grow at this step, and exposing \(v_{d+t}\)
requires that to be negative, so \(\varphi_{d+t+1}<0\). This is a contradiction. Both competitor
vertices exist: \(u_t\) because \(t\ge1\), and \(v_{d+t+1}\) because \(j=t+1\le m\) gives
\(d+t+1\le n\), which is precisely what the hypothesis \(j\le m\) buys.

The case \(j=t-1\ge1\) is the mirror image, across the coordinate \(d+t\) instead.
\end{proof}

\textbf{The hypothesis \(1\le j\le m\) is what makes this true}. If \(j=t+1\) exceeds \(m\) then the coordinate \(d+j\) lies outside the
shared block, \(e_{t+1}\) plays no role at all, and Proposition \ref{prop-col-beyond} gives the entry
value \(1\), not \(0\). This happens exactly once in each matrix, at \(t=m\), that is at
\((i,j)=(n,m+1)\); for \(n=5\), \(d=1\) it is the entry \((5,5)\), which the Kozlov matrices
displayed in Section \ref{sec-vertices}
duly show as a one. Without the hypothesis the two propositions would contradict each other there.

\begin{proposition}[]
\leanmark{\leanfull\,\compmark}\regionbox{\regdiag \regzeros{(0,0)--(0.61,0)--(0,0.61)--cycle}}(\emph{Below the diagonal.}) Let \(1\le t\le m\) and \(1\le j<t\), and suppose \(\vec a\) is positive with
\(\rho_1>\rho_2>\cdots>\rho_t\), which by Lemma \ref{lem-binomial-ratio} holds for the Kozlov vector.
Then \(\extten(\vec a;d)[i][j]=0\) for \(i=d+t\); note \(i>d+j\), so this is the region strictly
below the diagonal of the overlap block.
\label{prop:below-diagonal}
\label{prop-below-diagonal}
\end{proposition}

\begin{proof}
Suppose the entry were one, witnessed by \(\varphi\), and write \(K=t-j\ge1\) and
\(B_s=\sum_{l=j+1}^{j+s}e_lb_l\) for \(1\le s\le K\). Exposing \(u_j\) on the second summand forces
\(P_{j+s}<P_j\) for every \(s\), that is \(B_s<0\); exposing \(v_i\) on the first forces
\(Q_t>Q_j\), that is \(\sum_{l=j+1}^{t}e_lc_l>0\). Summation by parts turns the second sum into
\begin{equation}\label{eq:below-diagonal}
  \sum_{l=j+1}^{t}e_lc_l\;=\;\sum_{s=1}^{K}e_{j+s}b_{j+s}\rho_{j+s}
  \;=\;B_K\rho_{t}+\sum_{s<K}B_s\bigl(\rho_{j+s}-\rho_{j+s+1}\bigr),
\end{equation}
in which every \(B_s\) is negative, \(\rho_t>0\), and every difference
\(\rho_{j+s}-\rho_{j+s+1}\) is positive by hypothesis. So the whole sum is negative, contradicting
the first summand's requirement.
\end{proof}

The remaining entries are those with \(1\le t\le m\), \(1\le j\le m\) and gap \(K:=j-t\ge2\). This is
the substantive case, and we give it in full.

\begin{proposition}[]
\leanmark{\leanfull\,\compmark}\regionbox{\regdiag \regones{(0.21,0.7)--(0.7,0.7)--(0.7,0.21)--cycle}}(\emph{Gap at least two.}) Let \(n\ge2\), \(\vec d=(0,d)\) with \(1\le d\le n-1\), \(m=n-d\), and let
\(\vec a\) be a positive leg vector. Let \(1\le t\le m\) and \(1\le j\le m\) with \(K:=j-t\ge2\),
and suppose
\begin{equation}\label{eq:gapK}
  \rho_{t+1}\;>\;\rho_{t+K}\;=\;\rho_j ,
\end{equation}
which by Lemma \ref{lem-binomial-ratio} holds for the Kozlov vector. Then
\(\extten(\vec a;d)[i][j]=1\) for \(i=d+t\); that is, \(v_{d+t}+u_j\) is a vertex of
\(\RA(\vec a;0,d)\).
\label{prop:gapK}
\label{prop-gapK}
\end{proposition}

\begin{proof}
Exhibit the functional. Any \(\delta\) with \(0<\delta<\rho_{t+1}/\rho_j-1\) will do, and such a
\(\delta\) exists exactly because \(\rho_{t+1}>\rho_j>0\); we fix the midpoint
\begin{equation}\label{eq:gapK-2}
  \delta=\frac12\Bigl(\frac{\rho_{t+1}}{\rho_j}-1\Bigr),
  \qquad\text{so that}\qquad (1+\delta)\rho_j=\frac{\rho_{t+1}+\rho_j}{2}<\rho_{t+1},
\end{equation}
which makes the one inequality the choice of \(\delta\) is for immediate. Define
\(\varphi\in(\mathbb R^N)^*\) by
\begin{equation}\label{eq:gapK-3}
  \varphi_k=1\ (1\le k\le d),\qquad \varphi_{n+s}=-1\ (1\le s\le d),\qquad
  \varphi_{d+s}=e_s\ (1\le s\le m),
\end{equation}
where the shared block is
\begin{align}
  e_s &= 1 & &(s\le t), \label{eq:gapK-e1}\\
  e_{t+1} &= \frac{-1}{a_{t+1}}, & & \label{eq:gapK-e2}\\
  e_{t+l} &= 0 & &(1<l<K), \label{eq:gapK-e3}\\
  e_{j}=e_{t+K} &= \frac{1+\delta}{a_{j}}, & & \label{eq:gapK-e4}\\
  e_s &= -1 & &(j<s\le m). \label{eq:gapK-e5}
\end{align}
\textbf{No largeness is required anywhere}: the constants \(\pm1\) on the two private blocks are used as
written, and the verification below never asks for more. Put \(z_l=e_{t+l}a_{t+l}\), so that
\(z_1=-1\), \(z_K=1+\delta\) and \(z_l=0\) otherwise, and note \(e_{t+l}a_{d+t+l}=z_l\rho_{t+l}\).
Write \(A=\sum_{k\le d}a_k\) for the first summand's private-block total.

By Lemma \ref{lem-vertexdecomp}(1) it suffices to show that \(\varphi\) uniquely maximizes on each summand,
at \(v_i\) and at \(u_j\) respectively. \textbf{The three index blocks do not interact}, which is what
makes the verification immediate: \(v_{i'}\) never meets \([n+1,N]\), and \(u_{j'}\) never meets
\([1,d]\).

For the first summand. A vertex \(v_{i'}\) with \(i'\le d\) has value \(\sum_{k\le i'}a_k\le A\); one
with \(i'>d\) has value \(A+Q_{i'-d}\). So it is enough that \(Q_t\) be the strict maximum of
\(Q_0,\dots,Q_m\), together with \(Q_t>0\), and the latter is \(Q_t>Q_0\). If \(t'\le t\) then
\(Q_{t'}=\sum_{s\le t'}a_{d+s}\) is strictly increasing in \(t'\), so \(Q_t\) beats all of them,
including \(Q_0=0\). For \(t'=t+l\) with \(1\le l\le K\),
\begin{equation}\label{eq:gapK-Qstep}
  Q_{t'}-Q_t\;=\;\sum_{l'\le l}z_{l'}\rho_{t+l'}
  \;=\;\begin{cases}
    -\rho_{t+1}<0, & l<K,\\[2pt]
    -\rho_{t+1}+(1+\delta)\rho_j=\tfrac12(\rho_j-\rho_{t+1})<0, & l=K,
  \end{cases}
\end{equation}
the second line by \eqref{eq:gapK-2}. For \(t'>j\),
\begin{equation}\label{eq:gapK-Qtail}
  Q_{t'}\;=\;Q_j-\sum_{s=j+1}^{t'}a_{d+s}\;<\;Q_j\;<\;Q_t .
\end{equation}

For the second summand. A vertex \(u_{j'}\) with \(j'\le m\) has value \(P_{j'}\); one with
\(j'>m\) has value \(P_m-\sum_{s=m+1}^{j'}a_s\). The three ranges of \(j'\) below \(m\) give
\begin{align}
  P_{j'} &= \sum_{s\le j'}a_s \;\le\; P_t \;<\; P_t+\delta \;=\; P_j
    & &(j'\le t), \label{eq:gapK-P1}\\
  P_{j'} &= P_t-1, \quad\text{so}\quad P_j-P_{j'}=1+\delta>0
    & &(j'=t+l,\ 1\le l<K), \label{eq:gapK-P2}\\
  P_{j'} &= P_j-\sum_{s=j+1}^{j'}a_s \;<\; P_j
    & &(j<j'\le m), \label{eq:gapK-P3}
\end{align}
using \(P_j=P_t+\sum_lz_l=P_t+\delta\) in the first line. Hence \(P_j\) is the strict maximum of \(P_1,\dots,P_m\); in
particular \(P_m\le P_j\), so for \(j'>m\) the value \(P_m-\sum_{s=m+1}^{j'}a_s<P_m\le P_j\), the
first inequality being strict because \(\vec a\) is positive. This is the step at which one expects
to need a large penalty on the private block, and does not: \(u_j\) already beats \(u_m\) weakly, and
any strictly positive penalty then settles it.
\end{proof}

\begin{mexample}
The construction run on numbers, the smallest case with a gap: \(n=4\), \(d=1\),
\(\vec a=\kozvec(4)=(4,6,4,1)\), so \(m=3\) and \(N=5\); take \(t=1\) and \(j=3\), whence \(K=2\).
The ratios are \(\rho=(3/2,\,2/3,\,1/4,\,0)\), so the hypothesis \eqref{eq:gapK} reads
\(\rho_2=2/3>1/4=\rho_3\) and
\begin{equation}\label{eq:gapK-ex}
  \delta=\tfrac12\bigl(\tfrac{\rho_2}{\rho_3}-1\bigr)=\tfrac56,
  \qquad (1+\delta)\rho_3=\tfrac{11}{24}<\tfrac23=\rho_2 .
\end{equation}
The shared block is \(e=(1,\,-1/6,\,11/24)\) by
\eqref{eq:gapK-e1}--\eqref{eq:gapK-e4}, with no \(e_s=-1\) tail since \(j=m\), so
\begin{equation}\label{eq:gapK-ex-2}
  \varphi=\bigl(1,\ 1,\ -\tfrac16,\ \tfrac{11}{24},\ -1\bigr),
\end{equation}
the outer \(1\) and \(-1\) being the two private blocks, here one coordinate each. Its values are
\begin{equation}\label{eq:gapK-ex-3}
  \langle\varphi,v_{i'}\rangle=\bigl(4,\ \mathbf{10},\ \tfrac{28}{3},\ \tfrac{235}{24}\bigr),
  \qquad
  \langle\varphi,u_{j'}\rangle=\bigl(4,\ 3,\ \mathbf{\tfrac{29}{6}},\ \tfrac{23}{6}\bigr),
\end{equation}
maximized uniquely at \(v_2=v_{d+t}\) and at \(u_3=u_j\) as claimed, so
\(v_2+u_3=(4,10,6,4,0)\) is a vertex of \(\RA(\kozvec(4);0,1)\), which it is, that polytope having
\(11\) vertices in dimension \(4\).

The two bookkeeping quantities of the proof come out as \(P=(4,\,3,\,29/6)\) and
\(Q=(6,\,16/3,\,139/24)\). One sees both mechanisms directly: \(P_j-P_t=29/6-4=\delta\), the whole
margin by which \(u_j\) wins being \(\delta\) itself; and \(Q_j-Q_t=139/24-6=-5/24\), which is
\(\tfrac12(\rho_3-\rho_2)\) exactly as in \eqref{eq:gapK-Qstep}. The second is the step the choice of
\(\delta\) is for: at \(\delta\) any larger than \(\rho_2/\rho_3-1=5/3\) the sign would flip and
\(v_2\) would lose to \(v_4\).
\label{ex-gapK}
\end{mexample}

\begin{remark}
Three things the proof makes visible.

First, it needs \textbf{only} \(\rho_{t+1}>\rho_j\) at the two ends of the gap. The full strict monotonicity
of Lemma \ref{lem-binomial-ratio} is convenient but not required, so the proposition holds for any
positive leg vector with that one property, not just for the Kozlov vector.

Second, it explains the \(\vec a\)-dependence of Section \ref{sec-vertices} exactly: for
\(\vec a=\onesvec(n)\) one has \(\rho_s\equiv1\), so \(\rho_{t+1}/\rho_j-1=0\), no \(\delta\) exists,
and the construction is unavailable --- which is precisely why the all-ones extremal matrix is a
different matrix, and by Theorem \ref{thm-ones} genuinely has a zero at every such entry.

Third, the functional is completely explicit, with no parameter left to be chosen large. An earlier
version carried a factor \(\Lambda\) on the private blocks and asked for it to be large; the
verification above shows that \(\Lambda=1\) always works, because \(u_j\) already weakly beats
\(u_m\) before any penalty is applied.

The construction was checked against a direct computation of the exposed pair; see
Section \ref{sec-methods}.
\label{rem-gapK}
\end{remark}
\subsection{The Kozlov extremal matrix and the vertex count}
\label{sec-app-kozlov}
\begin{proof}[Proof of Theorem~\ref{thm:kozlov-matrix}]
The six propositions classify every pair \((i,j)\) with \(1\le i,j\le n\), and no pair is left
unaccounted for. If \(i\le d\), Proposition \ref{prop-row-below} gives the entry \(1\); if \(j>m\),
Proposition \ref{prop-col-beyond} does;
and both disjuncts of the displayed condition fail because \(t\le0\) in the first case and \(j>m\) in
the second. Otherwise \(1\le t\le m\) and \(1\le j\le m\), and exactly one of four cases holds:
\(j<t\), by Proposition \ref{prop-below-diagonal}, entry \(0\); \(j=t\), by
Proposition \ref{prop-diagonal}, entry \(1\); \(j=t+1\), by Proposition \ref{prop-adjacent} --- whose
hypothesis \(j\le m\) is met --- entry \(0\); and \(j\ge t+2\), by Proposition \ref{prop-gapK}, entry
\(1\). The hypotheses on \(\vec a\) of Propositions \ref{prop-below-diagonal} and \ref{prop-gapK} hold for
the Kozlov vector by Lemma \ref{lem-binomial-ratio}.

Note that \(t\le m\) is automatic: \(t=i-d\le n-d=m\). It is stated in the conclusion only to make
the zero set self-contained.
\end{proof}

\begin{proof}[Proof of the vertex count in Theorem~\ref{thm:kozlov-matrix}]
By Lemma \ref{lem-vertexdecomp}(1) every vertex of a Minkowski sum is a sum of vertices of the summands,
so the vertices are among the \(n^2\) points \(v_i+u_j\), and
Theorem \ref{thm-kozlov-matrix} says exactly which of them are vertices --- this is the definition of
\(\extten(\vec a;d)\). Distinctness holds because a vertex has only one decomposition as a sum of
points of the summands, which is Lemma \ref{lem-vertexdecomp}(2).

It remains to count the zeros of \(\extten(\vec a;d)\). By Theorem \ref{thm-kozlov-matrix} these lie in
the rows with \(1\le t\le m\), and the row with a given \(t\) contributes \(t-1\) zeros from
\(1\le j<t\), together with one more from \(j=t+1\) exactly when \(t+1\le m\). Hence the number of
zeros is
\begin{equation}\label{eq:gapK-6}
  \sum_{t=1}^{m}(t-1)\;+\;(m-1)\;=\;\frac{m(m-1)}{2}+(m-1)\;=\;\frac{(m-1)(m+2)}{2}. \qedhere
\end{equation}
\end{proof}
\section*{Acknowledgements}
\label{sec:orgf801a1a}
Aliaksandra Kupreyeva's Bachelor's thesis \emph{Simplices of \(f\)-vectors} (Linköping University, 2019;
report LiTH-MAT-EX--2019/05--SE), written under my supervision, is where the lattice-point question
inside the Kozlov simplex was first raised in the form used in Part III. The name given in
Definition \ref{def-ids} replaced an earlier and worse coinage of my own, at a colleague's suggestion.

Just as Geheimrat Goethe had the diligent Friedrich Wilhelm Riemer to help him write \emph{Die
Wahlverwandtschaften}, just as Grothendieck had Jean Dieudonné to help him complete EGA, I have
\textbf{Claude}! What exactly that amounted to is set out in the disclosure section below.

Computations were carried out in SageMath and in Macaulay2, the latter using
Amata and Crupi's own \texttt{ExteriorIdeals} and \texttt{ExteriorModules} packages
\autocite{AmataCrupi2018ExteriorIdeals,AmataCrupi2021,Macaulay2,SageMath}.
\section*{Disclosure of AI assistance}
\label{sec:orga23d458}
This manuscript was produced with substantial assistance from a large language model, Claude
(Anthropic), and the extent of it is set out here rather than left to be inferred.

\textbf{What the model was given.} The work started from a manual draft of my own, together with some
experimental and conjectural work on the extremal tensor, and a list of observations and
definitions that were easy to state and not yet proved --- among them the \emph{ff}-vector of a vector
of simplicial complexes, and the guess that the Minkowski sum of shifted Kozlov simplices gives the
convex hull of the Hilbert functions of Amata--Crupi modules, which became Theorem \ref{thm-main}.
The programme and its objects are therefore mine; what follows is what the model did with them.

The model: found and wrote up the facet description of the Kozlov simplex (Corollary \ref{cor-facets})
and verified it by exact polyhedral computation; supplied the proof of Theorem \ref{thm-kozlov} from the
local LYM inequality (Lemma \ref{lem-lym}); formulated the \(r\)-vector//ff/-vector language of
Section \ref{sec-combinatorial} and the indicator module of Definition \ref{def-indicator-module}; proved
Proposition \ref{prop-sumset} and Theorem \ref{thm-main}, including the normalization
Corollary \ref{cor-normalization}, the dimension formula Proposition \ref{prop-dimension} and the edge-case
analysis; found Example \ref{ex-nonids}, the sharper Remark \ref{rem-ac-reduction} and
Proposition \ref{prop-in-ids}, which together fix the scope of the main theorem; and proved Theorem \ref{thm-kozlov-matrix} in full, the
gap-\(\ge2\) half by the explicit functional of Proposition \ref{prop-gapK}, and with it
its vertex count; found the \(H\)-representation stated after Theorem \ref{thm-ones}; found and
verified the pairwise reduction, in the sharpened consecutive form; and ruled out the
certificate-level formulation discussed in the corresponding remark in Part IV. It
drafted the exposition throughout, with the author directing the scope and the order of
presentation.

Every polyhedral claim in this paper was checked along a route independent of the formula being
tested --- vertex descriptions against facet descriptions, Minkowski sums against direct enumeration
of the achievable set, and the algebraic statements against Macaulay2 --- and the reference list has
been checked against primary sources. The author has verified the results to the best of his ability,
made the final edits, and assumes full responsibility for any errors, mistakes or gaps that remain.
\section*{Funding}
\label{sec:org2d585af}
This research received no external funding. It was carried out as part of the author's regular duties
at Linköping University, within the fraction of that employment allocated to research.
\section*{Conflicts of interest}
\label{sec:orgd8cd07a}
The author declares no financial or non-financial conflicts of interest. The nature and extent of
large-language-model assistance in producing this manuscript is disclosed in full in the section
above.

\end{document}